\documentclass[11pt,a4paper,reqno]{amsart}
\usepackage{verbatim}
\usepackage{graphicx}
\usepackage[utf8]{inputenc}
\usepackage{amsmath,amsfonts,amssymb,amscd,amsthm,txfonts,hyperref}
\usepackage{xcolor}
\newtheorem{theorem}{Theorem}[section]
\newtheorem{proposition}[theorem]{Proposition}
\newtheorem{lemma}[theorem]{Lemma}
\newtheorem{corollary}[theorem]{Corollary}
\newtheorem{remark}{Remark}[section]

\theoremstyle{definition}

\newtheorem{notation}[theorem]{Notation}

\newcommand{\R}{\mathbb{R}}
\newcommand{\C}{\mathbb{C}}
\newcommand{\N}{\mathbb{N}}

\newcommand{\E}{\mathbb E}
\renewcommand{\Re}{\mbox{Re}}

\newcommand{\re}{\textrm{Re}}

\newcommand{\an}[1]{\langle #1 \rangle}
\newcommand{\grad}{\bigtriangledown}

\newcommand{\be}{\begin{equation}}
\newcommand{\ee}{\end{equation}}
\newcommand{\bee}{\begin{eqnarray*}}
\newcommand{\eee}{\end{eqnarray*}}
\newcommand{\lab}{\label}
\newcommand{\fref}{\eqref}
\newcommand{\ba}{\begin{array}}
\newcommand{\ea}{\end{array}}
\newcommand{\pa}{\partial}

\newcommand{\bea}{\begin{eqnarray}}
\newcommand{\eea}{\end{eqnarray}}

\newcommand{\barR}{\overline{\mathbb R}}

\begin{document}

\title{Stability of steady states for Hartree and Schr\"odinger equations for infinitely many particles}

% Stability of fermionic distributions for a Hartree equation on random fields %

\author[C. Collot]{Charles Collot}
\address{CNRS, and Laboratoire Analyse G\'eometrie Mod\'elisation, CY Cergy Paris Universit\'e, 33 Boulevard du Port, 95000 Cergy, France}
\email{ccollot@cyu.fr}
\author[A.-S. de Suzzoni]{Anne-Sophie de Suzzoni}
\address{CMLS, \'Ecole polytechnique, CNRS, Universit\'e Paris-Saclay, 91128 Palaiseau Cedex, France.}
\email{anne-sophie.de-suzzoni@polytechnique.edu}

%\begin{keyword}
%Hartree Equation, Random Fields, Stability, Scattering.
%\end{keyword}

\begin{abstract} 

We prove a scattering result near certain steady states for a Hartree equation for a random field. This equation describes the evolution of a system of infinitely many particles. It is an analogous formulation of the usual Hartree equation for density matrices. We treat dimensions $2$ and $3$, extending our previous result \cite{CodS}. We reach a large class of interaction potentials, which includes the nonlinear Schr\"odinger equation. This result has an incidence in the density matrices framework. The proof relies on dispersive techniques used for the study of scattering for the nonlinear Schr\"odinger equation, and on the use of explicit low frequency cancellations as in Lewin and Sabin \cite{lewsab2}. To relate to density matrices, we use Strichartz estimates for orthonormal systems from Frank and Sabin \cite{FS}, and Leibniz rules for integral operators.

\end{abstract}


\maketitle

\tableofcontents

\section{Introduction}

\noindent We consider the Cauchy problem:
\be \label{hartree}
\left\lbrace{\begin{array}{l l}
i\partial_t X  = -\Delta X + (w*\E(|X|^2)) X, \\
X(t=0) = X_0, \end{array}} \right. \quad x\in \mathbb R^3,
\ee
where $X : \Omega \times \R_t \times \R^3_x\rightarrow \C$ is a random field defined over a probability space $( \Omega,\mathcal A, \mathbb P)$ with expectation $\mathbb E$, $*$ is the convolution product on $\R^3$, and $w$ is an even pair interaction potential. We are able here to cover the case of the mild assumption for $w$ that:
\be \label{id:hpw}
w \mbox{ is a finite Borel measure on } \mathbb R^3.
\ee
This is the case if, for example, $w\in L^1(\mathbb R^3)$, or, if $w=\pm \delta$ is the Dirac mass, in which case \fref{hartree} is a variant of the nonlinear Schr\"odinger equation
$$
i\partial_t X  = -\Delta X  \pm \E(|X|^2) X.
$$
Equation \fref{hartree} admits the following phase invariance: if $X(\omega ,t,x)$ is a solution to \fref{hartree} then so is
\be \label{id:phase}
e^{ia(\omega)}X(t,x,\omega) \quad \mbox{for all measurable } a:\Omega \rightarrow \mathbb R.
\ee
Equation \fref{hartree} arises in the study of large fermionic systems (see below), and may be understood as a system of coupled Hartree equations, indexed by $\omega \in \Omega$, writing it as $i\pa_t X(\omega)=-\Delta X(\omega)+(w*\int_{\Omega}|X(\omega')|^2d\mathbb P (\omega'))X(\omega)$. For the problem at hand, we choose to keep the random field point of view as a convenient framework.

\subsection{Steady states}
Equation \fref{hartree} admits particular steady states. We assume that on $\Omega$ is defined a Wiener process $W$ of dimension $3$ (a white noise on $\mathbb R^3$), namely $(dW(\xi))_{\xi \in \R^3}$ is a family of infinitesimal independent Gaussian fields with values in $\C$, such that for all $\xi,\eta \in \R^3$
$$
\E(dW(\xi) \overline{dW(\eta)}) = \delta (\xi - \eta) d\xi d\eta.
$$
We refer to the appendix \ref{sec:tech} for some basic results and references on Wiener integration used in this article. Consider for $f\in L^2(\mathbb R^3, \mathbb C)$ and $m = \int_{\R^3} w \cdot \int_{\R^3}|f|^2\in \R$ (where $\int_{\R^3} w$ denotes the total mass of $w$) the random field:
\begin{equation}\label{id:Y}
 \begin{array}{  l l l l}
Y_f:&\Omega \times \R \times \R^3 &\rightarrow &\C \\
&(\omega, t, x)& \mapsto & \int_{\mathbb R^3} f(\xi) e^{-it(m+\xi^2)+i\xi \cdot x} dW(\xi)(\omega)\end{array}.
\end{equation}
For every $(t,x)$, $Y_f(t,x)$ is a centred Gaussian variable with constant variance $\mathbb E(|Y_f|^2(t,x))=\int_{\mathbb R^3} |f|^2$. If for $k\in \{0,1,...\}$ and $s>k$, $\int_{\mathbb R^3} |f(\xi)|^2\langle \xi\rangle^{2s}d\xi<\infty$, then for almost every $\omega$, $Y_f(\omega)$ is a continuous function with subpolynomial growth at infinity on $\mathbb R^{1+3}$, and with continuous $\pa_t^{\alpha}\nabla_x^\beta Y_f$ derivatives for $2\alpha+\beta\leq k$. For $s>2$ in particular, almost surely, the identity
$$
i\partial_t Y_f = \int f(\xi) (m+ \xi^2) e^{-it(m+\xi^2) + ix\cdot \xi}dW(\xi) = (m-\Delta) Y_f
$$
holds in a classical sense everywhere on $\mathbb R\times \mathbb R^3$, showing that $Y_f$ is a solution to \fref{hartree}. Assuming solely $s>0$, we still have that $Y_f$ is a weak solution to \fref{hartree} almost surely.

The field $Y_f$ is a Gaussian field whose law is invariant under space translations, which makes it non-localised, and time translations, which suggests the denomination "equilibrium" even though $Y_f$ is not a invariant state. In the sequel we omit the dependance in $f$ in the notation and write $Y$ for $Y_f$, and $Y_0$ for $Y_f(t=0)$. Note also that one can assume $f\geq 0$ without loss of generality.

\subsection{An effective equation for Fermions} The equation \fref{hartree} at stake here is closely related to the commonly used Hartree equation for density matrices. To study systems of infinitely many Fermions, it is customary to study the equation:
\be \lab{eq:hartreeoperator}
i\pa_t \gamma = [-\Delta +w*\rho_{\gamma},\gamma].
\ee
Above, $\gamma$ is a time dependent bounded operator on $L^2(\mathbb R^3)$ with integral kernel $\tilde \gamma (x,y) $, $[\cdot,\cdot]$ denotes the commutator, and $\rho_{\gamma}(x)=\tilde \gamma (x,x)$ is the density of particles, that is the diagonal of $\tilde \gamma$. An infinite number of particles can then be modelled by a solution of \fref{eq:hartreeoperator} which is not of finite trace (the trace of the operator being, by the derivation of the model, the number of particles). In \cite{dS}, the second author proposed the equation \ref{hartree} as an alternative equation to \fref{eq:hartreeoperator}. For a detailed paralell between \eqref{hartree} and \eqref{eq:hartreeoperator}, we refer to \cite{CodS,dS,CdS2}.

Solutions of \fref{eq:hartreeoperator} with an infinite number of particles were studied previously in \cite{BPF,BPF2,C,Z} for exemple, and more recently in \cite{CHP,CHP2,lewsab2,lewsabI}. The derivation of equation \eqref{eq:hartreeoperator} from many body quantum mechanics has been treated in \cite{BEGMY,BGGM,BJPSS,derfer3,EESY,derfer5}.

In the seminal work \cite{lewsab2}, the authors show the stability of the above equilibria for the Equation \fref{eq:hartreeoperator} for density matrices in dimension 2. Important tools are dispersive estimates for orthonormal systems \cite{FLLS,FS}. This work has been extended to higher dimension in \cite{CHP2}. Note that in higher dimension, some structural hypothesis is made on the interaction potential $w$, in particular, in dimension $3$, it imposes $\hat w(0) = 0$, to solve some technical difficulties about a singularity in low frequencies of the equation that we will identify precisely in the sequel. The stability result corresponds to a scattering property in the vicinity of these equilibria: any small and localized perturbation evolves asymptotically into a linear wave which disperses. We mention equally \cite{CHP,lewsabI} about problems of global well-posedness for the equation on density matrices.\\

\noindent A relevant recent result about equation \eqref{eq:hartreeoperator} is by Lewin and Sabin, \cite{LS3} in which the authors prove that the semi-classical limit of the Hartree equation : 
\begin{equation}\label{hartreesemicl}
ih \partial_t \gamma = [-h^2 \Delta + h^d w* \rho_\gamma , \gamma]
\end{equation}
where $\gamma$ is a positive integral operator and $\rho_\gamma$ is the diagonal of its integral kernel is the Vlasov equation
\begin{equation}\label{Vlasov}
\partial_t W + 2 v\cdot \nabla_x W - \nabla_x (w*\rho_W) \cdot \nabla_v W = 0
\end{equation}
where $\rho_W = (2\pi)^{-d} \int W(v,x) dv$. In the course of the proof, they prove functionnal inequalities such as Lieb-Thirring inequalities. The data are taken in entropy classes relative to a referential state corresponding to our equilibria.

Recentely again, Pusateri and Sigal, \cite{PuSi19} proved scattering near the $0$ solution for the equation 
$$
i\partial_t \gamma = [-\Delta + f(\rho_\gamma),\gamma]
$$
for a large class of nonlinearities $f$ that include $w*\rho_\gamma$ where $w$ belongs to a weak $L^r$, for $r\in (1,d)$. Their result is quite general as it narrowly misses the Coulomb potential. They give a conjecture of modified scattering for the Coulomb potential.

In \cite{CodS}, we proved the asymptotic stability of equilibria for \eqref{hartree} in dimension higher than $4$ without the structural hypothesis on the potential of interaction at low frequencies of \cite{CHP2}. The problem of the stability of the equilibria \fref{id:Y} for Equation \fref{hartree} shares similarities with the mechanism of scattering for the Gross-Pitaevskii equation $i\pa_t \psi=-\Delta \psi+(|\psi|^2-1)\psi$. In both problems the linearized dynamics have distinct dispersive properties at low and high frequencies, making the nonlinear stability problem harder, especially in low dimensions, where dispersion is weaker. The proof of scattering for small data for the Gross-Pitaevskii equation was done in \cite{GHN,GNT3,GNT2,GNT}. The solution in dimension higher than $4$ was to use spaces with different regularities at low and high frequencies, inspired by \cite{GNT}. Indeed, this was sufficient because it gave rise to a quadratic Schr\"odinger-type equation, where dispersive techniques are sufficient to prove scattering. In dimension $3$, this strategy is not sufficient. Therefore, we adopt a strategy similar to \cite{GHN,GNT3,GNT2,lewsab2}, which is to treat differently the first Picard interaction of the solution than the rest of the solution. In \cite{GHN,GNT3,GNT2}, this was done through a normal form to remove the difficult quadratic part of the equation, and in \cite{lewsab2}, this was done through a complete expansion of the solution into Picard interactions. Random cancellations and homogeneous Strichartz estimates allow us to close the argument. 

We give a fairly detailed strategy of the proof, and then compare our techniques more extensively with the existing literature at the end of this introductory section.

\subsection{Main result}

\noindent We state here our result in dimension $3$. An analogous result in dimension $2$ is given in Appendix \ref{sec:dim2}. In what follows, we write $\langle \xi \rangle =(1+|\xi|^2)^{1/2}$ and, given $a\in \mathbb R$, $(a)_+=\max (a,0)$ and $(a)_-=\max (-a,0)$ the nonnegative and nonpositive parts of $s$. We write with an abuse of notation $f(\xi)=f(r)$ with $r=|\xi|$, if $f$ has spherical symmetry. The space $L^2_\omega H^{s}_x$ is the set of measurable functions $Z: \Omega \times \mathbb R^d \rightarrow \mathbb C$ such that $Z(\omega , \cdot)\in H^s(\mathbb R^3)$ almost surely and 
$$
\int_{\mathbb R^d\times \Omega} \langle \xi \rangle^{2s} |\hat Z(\omega , \xi)|^2d\xi d \omega<+\infty.
$$
We introduce the notation for the solutions to $iu_t=(-\Delta +m)u+V Y$ for $V \in L^1_{\textrm loc}(\R, L^2(\R^3)) $: 
\be \label{def:WV}
S(t) = e^{-it(m-\Delta)} \quad, \quad W_V(t) = -i\int_0^{t} S(t-\tau) (V(\tau) Y(\tau) )d\tau.
\ee

\begin{theorem} \label{th:main} We denote by $h$ the Fourier transform of $|f|^2$ on $\mathbb R^3$.
Assume the momentum distribution function $f$ satisfies : \begin{itemize}
\item $f>0$ is a bounded $C^1$ radial function on $\mathbb R^3$, with $\partial_r f<0$,
\item $\int_{\mathbb R^3} \langle \xi \rangle f^2(\xi)d\xi<\infty$ and $\int_{\mathbb R^3}|\xi|^{-1} |f(\xi)\nabla  f(\xi)|d\xi < \infty$,
\item $\int_0^\infty (1+r)|h|(r)dr<\infty$ and $\int_0^\infty \left(\frac{|h'|(r)}{r}+|h''|(r)\right)dr<\infty$ where the derivatives $h'$ and $h''$ are defined in the sense of distributions,
\end{itemize}
and that $w$ satisfies \fref{id:hpw} and\footnote{Note that since $w$ is an even finite Borel measure, its Fourier transform is continuous and real, so that $w_+$ and $w_-$ are well-defined.} (where below $\epsilon_h$ is a constant depending on $h$ defined by \fref{def:eph})
$$
\| (\hat w)_-\|_{L^{\infty}}\left(\int_{0}^{\infty} r|h(r)|dr\right)<2 \quad \mbox{and} \quad \hat w (0)_+ \epsilon_h <1.
$$
Then there exists $\delta>0$ such that for all $Z_0\in L^2_\omega H^{1/2}_x \cap L^{3/2}_{x}L^2_\omega$ with $\| Z_0\|_{L^2_\omega H^{1/2}_x \cap L^{3/2}_{\omega}L^2_x}\leq \delta$ the following holds true. The Cauchy problem \fref{hartree} with initial datum $Y_0+Z_0$ is globally well-posed in $Y  + \mathcal{C} (\R, L^2_\omega,H^{1/2})$, and what is more, there exist $Z_\pm \in L^2_\omega H^{1/2}_x$ and $V\in L^{2}_{t}H^{1/2}_x\cap L^{5/2}_{t,x}$ such that
\be \label{id:mainresult}
X(t) = Y(t) + W_V(Y)(t) +S(t) Z_\pm +o_{L^2_\omega H^{1/2}_x}(1) \quad \quad \mbox{as } t\rightarrow \pm \infty.
\ee
For the third term above, there exists $\tilde Z_\pm \in L^3_x L^2_\omega$ with $S(t) \tilde Z_\pm \in C(\mathbb R,L^3_x L^2_\omega)$ such that 
\be \label{id:mainresult2}
W_V(Y) =  S(t) \left(\tilde Z_\pm + o_{L^3_x L^2_\omega}(1) \right)= S(t) \tilde Z_\pm + o_{L^3_xL^2_\omega} (1)  \quad \quad \mbox{as } t\rightarrow \pm \infty.
\ee
\end{theorem}

Relating the framework of random fields to that of density matrices, from the above Theorem \ref{th:main} one obtains a scattering result for the operator:
\be \label{def:gamma}
\gamma =\mathbb E (|X\rangle  \langle X|):= u \mapsto \Big( x\mapsto  \mathbb E (X(x) \langle X,u\rangle_{L^2(\mathbb R^d)})\Big),
\ee
with respect to the one associated to the equilibrium $Y$:
\be \label{def:gammaf}
\gamma_f= \mathbb E (|Y\rangle \langle Y |),
\ee
which is the Fourier multiplier by $|f|^2(\xi)$. This convergence holds in Schatten-Sobolev spaces (where below $\mathfrak S^p$ is the standard Schatten space $p$-norm for operators on $L^2(\mathbb R^3)$ and $\alpha\in\mathbb R$):
$$
\| \gamma \|_{\mathfrak S^{\alpha,p}}=\| \langle \nabla \rangle^{\alpha} \gamma \langle \nabla \rangle^{\alpha} \|_{\mathfrak S^p}.
$$

\begin{corollary} \label{corollaire}

Under the hypotheses of Theorem \ref{hartree}, defining the operators on $L^{2}(\mathbb R^3)$:
\be \lab{def:WVpm}
W_{V,\pm}:u\mapsto-i\int_0^{\pm \infty} S(-\tau) (V(\tau) S(\tau)u)d\tau,
\ee
and
\be \lab{def:gammapm}
\gamma_\pm= \mathbb E\left(|W_{V,\pm}Y_0+Z_{\pm}\rangle \langle Y_0|+|Y_0 \rangle \langle W_{V,\pm}Y_0+Z_\pm|+|W_{V,\pm}Y_0+Z_{\pm}\rangle \langle W_{V,\pm}Y_0+Z_\pm |\right)
\ee
there holds for any $\epsilon>0$ that $\gamma_\pm \in \mathfrak S^{1/2,4+\epsilon}$ and the convergence:
$$
\gamma=\gamma_f+e^{i\Delta t}\gamma_\pm e^{-i \Delta t}+o_{\mathfrak  S^{\frac 12,4+\epsilon}}(1) \ \mbox{ as } t\rightarrow \pm \infty.
$$

\end{corollary}

\begin{remark} The conditions on $f$ are satisfied by thermodynamical equilibria for bosonic or fermionic gases at a positive temperature $T$, and the Bessel potential distribution:
$$
|f(\xi)|^2=\frac{1}{e^{\frac{|\xi|^2-\mu}{T}}-1}, \ \ \mu <0, \ \ \text{ or } \ \ |f(\xi)|^2=\frac{1}{e^{\frac{|\xi|^2-\mu}{T}}+1}, \ \ \mu\in \mathbb R, \ \ \mbox{ or } \ \ |f(\xi)|^2=\langle \xi \rangle^{-\alpha}, \ \ \alpha >4,
$$
respectively, but it is not the case of the fermionic gases at zero temperature $|f(\xi)|^2={\bf 1}_{|\xi|^2\leq \mu}$ for $\mu>0$. Given an $f$ satisfying the hypotheses, interaction potentials satisfying the requirements are for example any Borel measure with total mass $c$ (for example  $w=\pm c\delta$) or $w\in L^1$ with $\| w\|_{L^1}=c$ for $c\leq 2(\int_0^\infty r|h|)^{-1}$.

\end{remark}

\begin{remark}

The sobolev regularity $s=\frac12$ for the initial perturbation appears optimal since $H^{1/2}_x$ is the critical space for local well-posedness in dimension $3$ for the usual NLS. The space $L^{3/2}_xL^2_\omega$ asks for an additional spatial localisation of $Z_0$, and ensures the potential generated by the interaction of the free evolution with the equilibrium $2\re \E( \bar Y S(t)Z_0)$ is not singular at low frequencies (it is related to taking the initial datum in \cite{CHP2,lewsab2} in low Schatten spaces).

The space $\mathfrak S^{1/2,4+\epsilon}$ appears to be optimal in view of the regularity of the perturbation $Z_0$ (for $1/2$), and of the Strichartz estimate for orthonormal systems \fref{bd:strichartzpotential} from \cite{FLLS,FS} (for $4$). It sharpens the result of \cite{CHP2} where scattering was proved to hold in $\mathfrak S^{0,6}$.

\end{remark}

\begin{remark}

Theorem \ref{th:main} has a direct consequence for Equation \fref{eq:hartreeoperator} on density matrices. It implies scattering for \fref{eq:hartreeoperator} near $\gamma_f$ for all perturbations in $\mathfrak S^{1,1/2}$ (with a finite number of particles). Indeed, the density matrix associated to $Y_0+Z_0$ is $\gamma_f+\gamma'$, where $\gamma'=\gamma'_1+\gamma_2'$ with $\gamma_1'=\mathbb E |Z_0\rangle \langle Z_0|$ and $\gamma_2'=\mathbb E (|Z_0\rangle \langle Y_0|+|Y_0\rangle \langle Z_0|)$. By taking $Z_0\in L^2_\omega H^{1/2}$ independent in probability of $Y$ we have $\gamma'_2=0$, and that the operator $\gamma'_1$ can be any non-negative operator in $\mathfrak S^{1,1/2}$. Theorem \ref{th:main} implies also scattering for \fref{eq:hartreeoperator} near $\gamma_f$ for perturbations in a subset of $\mathfrak S^{2,1/2}$ (infinite number of particles). This is obtained by taking $Z_0$ not independent of $Y_0$, so that $\gamma'_2\in \mathfrak S^{2,1/2}$. The appearance of $Y$ in $\gamma'_2$ has a regularising effect. Hence, the operators that can be written as $\gamma'_1+\gamma_2'$ for $Z_0\in L^2_\omega H^{1/2}$ are a subset of $\mathfrak S^{2,1/2}$ with higher regularity that we did not try to characterize.

\end{remark}

\subsection{Set-up, strategy of the proof, comparison with literature} \label{setup}

We describe here the change of unknown that transforms the stability problem at hand into an amenable fixed point problem. We then explain formally the bounds obtained for the linear and multilinear terms.

\subsubsection{Set-up} From standard arguments, Equation \eqref{eqZtilde} is locally well-posed in 
$$
\mathcal C([-T,T],H^{1/2}L^2_{\omega}) \cap L^{10/3}_{t\in [-T,T]}W^{1/2,10/3}_xL^2_\omega \cap L^5_{t\in [-T,T],x}L^2_\omega,
$$
for small data in $L^2_\omega H^s_x$ for times $T\sim 1$. We shall look for solutions satisfying the global bound:
\be \label{def:spacegwp}
X\in Y + \mathcal C_b(\R,H^{1/2}L^2_{\omega}) \cap L^{10/3}_{t}W^{1/2,10/3}_xL^2_\omega \cap L^5_{t,x}L^2_\omega,
\ee
which will therefore coincide with the above ones. We decompose the initial datum and solution: $Z_0 = X_0 - Y_0$ and $\tilde Z = X-Y$, giving the following Cauchy problem for $\tilde Z$:
\begin{equation}\label{eqZtilde}
\left \lbrace{\begin{array}{ll}
i\partial_t \tilde Z =(m -\Delta) \tilde Z +\left[w*\Big( \E(|\tilde Z|^2 )+ 2\textrm{Re }\E(\bar Y \tilde Z) \Big)\right] (Y+\tilde Z), \\
\tilde Z(t=0) = Z_0.
\end{array}} \right.
\end{equation}
To study the perturbation $\tilde Z$, let us denote by $V = \E(|\tilde Z|^2 )+ 2\textrm{Re }\E(\bar Y \tilde Z)$ the potential appearing above, that is a more convenient unknown than $\tilde Z$, yielding to the system:
$$
\left \lbrace{\begin{array}{ll}
i\partial_t \tilde Z =(m -\Delta) \tilde Z + (w*V) (Y+\tilde Z), \quad \quad \tilde Z(t=0) = Z_0, \\
V = \E(|\tilde Z|^2 )+ 2\textrm{Re }\E(\bar Y \tilde Z).
\end{array}} \right.
$$
We introduce the notations
$$
V'=w*V,
$$
and recall the notation \fref{def:WV}. Using Duhamel's formula, we aim at solving the fixed point equation:
$$
\left \lbrace{\begin{array}{l l}
\tilde Z = S(t) Z_0 + W_{V'}(\tilde Z+ Y),\\
V = \E(|\tilde Z|^2 )+ 2\textrm{Re }\E(\bar Y \tilde Z).
\end{array}} \right.
$$
The main issue is that $W_{V'}(Y)$ is less regular than other parts of $\tilde Z$. We thus plug back the first decomposition in the second equation, giving the next order expansion in $V$:
$$
\left \lbrace{ \begin{array}{ll}
\tilde Z = S(t) Z_0 + W_{V'}(Y) + W_{V'}(\tilde Z), \\
V = \E (|\tilde Z|^2) + 2 \textrm{Re } \E( \bar Y S(t) Z_0) +  2 \textrm{Re } \E(\bar Y W_{V'} (Y)) + 2 \textrm{Re } \E(\bar Y W_{V'}(\tilde Z)),\end{array}} \right.
$$
and then we write $\tilde Z = Z + W_{V'}(Y)$ which leads to the actual fixed point problem we will solve: 
\begin{equation}\label{fixedpoint}
\left \lbrace{\begin{array}{l l l}
Z &=& S(t) Z_0 + W_V^2(Y) + W_V(Z), \\
V &=& 2 \textrm{Re } \E(\bar Y S(t) Z_0) + 2 \textrm{Re } \E(\bar Y W_V(Y))+ \E(|Z|^2)  +
2  \textrm{Re } \E(\overline{W_V(Y)} Z + \bar Y W_V(Z))\\&&\quad \quad + \E (|W_V(Y)|^2) +2 \textrm{Re } \E(\bar Y W_V^2 (Y)).\end{array}} \right. 
\end{equation}
We write this fixed point problem in vectorial notation as
\be \label{fixedpoint2}
\begin{pmatrix} Z \\ V \end{pmatrix} = C_0 + L \begin{pmatrix} Z \\ V \end{pmatrix} + Q \begin{pmatrix} Z \\ V' \end{pmatrix}
\ee
where $C_0$ contains the terms depending on the initial datum, $L$ and $Q$ are the linear and quadratic terms:
$$
C_0 = \begin{pmatrix}
S(t) Z_0 \\ 2\textrm{Re }\E(\bar Y S(t) Z_0)
\end{pmatrix},
$$
\be \label{def:L2}
L \begin{pmatrix} Z \\ V \end{pmatrix} = \begin{pmatrix}
0 \\ L_2(V)
\end{pmatrix}, \quad \quad L_2(V)=2 \textrm{Re } \E(\bar Y W_{V'}(Y)),
\ee
\be \label{def:Q}
Q \begin{pmatrix} Z \\ V' \end{pmatrix} = \begin{pmatrix}
W_{V'}(Z) +  W_{V'}^2(Y) \\
\E(|Z|^2) + Q_1 (Z,V') +  Q_2(V')
\end{pmatrix},
\ee
with 
\be \label{def:Q1}
Q_1(Z,V') = 2  \textrm{Re } \E(\overline{W_{V'}(Y)} Z + \bar Y W_{V'}(Z)) 
\ee
and 
\be \label{def:Q2}
Q_2 (V') = \E (|W_{V'}(Y)|^2) +2 \textrm{Re } \E(\bar Y W_{V'}^2 (Y)).
\ee
We aim at solving the fixed point problem \fref{fixedpoint2} in a classical way by finding the right Banach space $\Theta$ for $Z,V$ and proving suitable estimates on the linear and nonlinear terms.

\subsubsection{Strategy of the proof} Here are the main points in our study of the problem \fref{fixedpoint2}.

\emph{A good unknown:} We mentioned that $V$ is a convenient unknown. This is firstly because $L_2(V)$ can be treated explicitly and $\text{Id}-L_2$ can be inverted (this idea was originally used in \cite{hainzl2005existence,lewsab2}). Secondly, in the course of the proof of Theorem \ref{th:main}, we take $Q_2(V')$ as a whole since cancellations occur between the different terms. The proof does not work if we treat separately $\E (|W_{V'}(Y)|^2)$ and $2 \textrm{Re } \E(\bar Y W_{V'}^2 (Y))$.

\emph{Standard Strichartz estimates:} We bound the linear and quadratic terms for large exponent space-time Lebesgue spaces using standard Strichartz estimates \fref{lem:strichartz}, employed here in the framework of a system of linear Schr\"odinger equation.

\emph{Additional linear decay at high spatial frequencies:} As the equilibrium $Y$ is a random superposition of independent linear waves $e^{i\xi.x}$, and from the commutator relation $[\Delta,e^{ix.\xi}]=e^{i\xi.x}(|\xi|^2-2\xi.\nabla)$, there holds (Lemma \ref{lem:expressionWVY}) for the perturbation generated when applying a potential $V$ to $Y$:
\bee
W_V(Y)(t,x)&=&-i\left[\int_0^t S(t-\tau) [V(\tau) Y(\tau)] d\tau\right](x) \\
&&\quad \quad =-i \int_{\mathbb R^3} dW(\xi) f(\xi) e^{-it(m+ |\xi|^2)+i\xi.x}  \int_{0}^t \left[S (t-\tau)  V(\tau)\right](x-2\xi (t-\tau)) .d\tau
\eee
At each fixed $\xi$ on the right, we notice transport at speed $2\xi$ in addition to the free evolution. This results in an additional spreading effect that averages out high frequencies, creating an additional damping and regularising effect (Proposition \ref{pr:freeandpotential} and Remark \ref{re:spreading}). We show the above function is in $L^{\infty}_tL^2_{x,\omega}$ if $V\in L^{2}_t\dot H^{-1/2}_x$ (instead of $V\in L^1_tL^2_x$ in the case of the usual inhomogeneous linear Schr\"odinger equation). This effect is present for all linear terms, and was used in \cite{CodS}.

\emph{Linear cancellation at low frequencies:} The two previous points control the part of the solution at high spatial frequencies. Low frequencies will however inflate with time. Indeed, given a potential $V$ at low spatial frequencies, so that $V(t,x)\approx V(t)$ does not depend on the spatial variable $x$ formally, the linear perturbation it generates is:
\be \label{id:formalcomputlowfreq}
W_V (Y)(t)=-i \int_{\mathbb R^3} dW(\xi) f(\xi) e^{-it(m+ |\xi|^2)+i\xi.x}  \int_{0}^t  V(\tau)d\tau =-iY(t)\int_0^t V(\tau)d\tau.
\ee
This quantity may grow for $V$ solely in $L^2_t$. But note that it is out of phase with $Y$. The linearised potential created by $W_V (Y)$ then displays the following cancellation since $V$ is real-valued:
$$
\Re \mathbb E (\bar YW_V (Y))=  \Re \mathbb E \left(\bar Y \left(-iY\int_0^t V(\tau)d\tau \right)\right)=- \Re \mathbb E \left(i |Y|^2\int_0^t V(\tau)d\tau \right)=0.
$$
Thus, the induced linear perturbation $W_V (Y)$ might grow, but the linear potential it creates in response, $\Re (\bar Y W_V (Y))$, does not. This cancellation is at the heart of the proof that the linearised operator $L_2:V\mapsto \Re (\bar Y W_{V'} Y)$, is in fact bounded on $L^2_{t,x}$ \cite{lewsab2}. Said differently, there is a triangular structure at low frequency between the parts $Z_1$ and $Z_2$ of the perturbation $Z$ that are respectively in phase and out of phase with $Y$: the part $Z_1$ decays due to dispersion and forces linearly $Z_2$, while $Z_2$ decays linearly but does not force $Z_1$, ensuring decay for $(Z_1,Z_2)$. Note that this is linked to the phase invariance \fref{id:phase}.

\emph{Quadratic cancellation at low frequencies:} The term $Q_2$ needs to be treated separately. It is not so complicated, following the previous discussion concerning high frequencies, to obtain that $Q_2(U,V)$ is integrable for large space-time Lebesgue exponent (Lemma \ref{prop:QuadV2.1} shows an $L^{10/3}_{t,x}$ integrability). The problem is rather here again to show that low frequencies do not inflate. Taking two potentials $V(t,x)\approx V(t)$ and $U(t,x)\approx U(t)$ at low spatial frequencies, and ignoring formally their dependance in the spatial variable in the next computation, we have the following cancellation at the heart of Proposition \ref{prop:QuadV2.2} (behind Proposition 4 in \cite{lewsab2}):
\bee
&& \overline{W_U(Y)}W_V(Y)+\bar YW_V\circ W_U (Y)+\bar YW_U\circ W_V (Y)\\
&= &|\bar Y|^2(t)\int_0^tU \cdot \int_0^t V-|Y|^2(t) \int_0^tds V(s)\int_0^sU(s')ds'-|Y|^2(t) \int_0^tds U(s)\int_0^sV(s')ds'=0
\eee
(dividing $(s,s')\in [0,t]^2$ into $s\leq s'$ and $s'>s$).
That is, the linearised potential generated by the second iterate $W_V\circ W_V$ cancels with the quadratic potential created by $W_V$. This would still be the case if the nonlinearity $|X|^2$ was replaced by a general one $f(|X|^2)$.

\emph{Technicalities:} To work at critical regularity $1/2$ in the framework of a system, and with a rough potential $w$ will sometimes raise technical issues. One has to refine using standard harmonic analysis tools and duality arguments a large part of the analysis of \cite{lewsab2,CodS}. This is in particular the case for the proof of Corollary \ref{corollaire}, as an endpoint Strichartz estimate for orthonormal systems is known to fail \cite{FLLS,FS} that forces us to prove additional bounds.\\

\emph{Summary} Here are the key points already mentioned and a brief comparison to related works \cite{CHP2,lewsab2,CodS}. Equation \eqref{eqZtilde} contains a $\R$ but not $\C$ linear term, a quadratic term, and a cubic term in $\tilde Z$. The linear term has been dealt with in all the above papers in a similar (if not the same) way as what is already done in the seminal work \cite{lewsab2}, by proving the invertibility of $\textrm{Id} - L_2$. 

In dimension higher than 4, quadratic Schr\"odinger-type equations scatter. The issue of singularities at low frequencies remains, but can be treated by either adding assumptions on the interaction potential $w$, as in \cite{CHP2} (allowing a large class of initial data thanks to the use of Strichartz estimates for orthogonal systems), or by using inhomogeneous Sobolev spaces, borrowed from the scattering for the Gross-Pitaevskii equation literature (\cite{GNT}), as in \cite{CodS}. 

In dimension 3 (or 2), a contraction argument using solely Strichartz estimates is not sufficient to prove scattering for quadratic Schr\"odinger-type equations. By rewriting the equation on $(\tilde Z,V)$ (or $(\gamma,V)$) and exploiting the structure of the nonlinearity, we are reduced again to a low frequencies singularity issue. In \cite{CHP2}, this fell under the hypothesis on $w$. However, the strategy we adopted in \cite{CodS}, getting a large class of $w$, does not work out in dimension 3. Therefore, we adopt here a different one, dealing with the less regular part of the solution separately, in the spirit of a normal form, as in \cite{GHN,GNT2,GNT3}, or in the spirit of \cite{lewsab2}, where the authors treat the "lower order iterates" differently from the "higher order ones". 

Note that compared to \cite{CodS}, we manage to reach an even greater class of interaction potentials, by proving specific Leibniz rules (Step 2 of the proof of Corollary \ref{corollaire}), but also by refining the estimates of the linear part and of the quadratic terms, see Remark \ref{rem:laRemarque}.   

A final word to conclude this summary is that Corollary \ref{corollaire} is intrinsically a density matrix result, not only because it deals with density matrices, but more meaningfully because it requires a specific tool for density matrices to be proven, that is Strichartz estimates for orthogonal systems, borrowed from \cite{FS}.

%Let us compare briefly with the related works \cite{CHP2,lewsab2}. All works use a linear stability result that is the invertibility of the linear response operator. Additional dispersive effects are used in \cite{CHP2,lewsab2} by means of new Strichartz estimates in the framework of density matrices \cite{BHLNS,CHP2,FLLS,FS}. We solely use here standard Strichartz estimates, but find that the response of the equilibrium also enjoys an improved decay at low frequencies due to randomness. This improved decay does not appear in \cite{CHP2,lewsab2}, and in \cite{CHP2} a singularity that appears at low frequencies is handled by making stronger assumptions on their interaction potential. Our analysis eventually yields a scattering result at critical regularity, as in \cite{lewsab2}, whereas \cite{CHP2} holds for supercritical regularity.
%
%The difficulty compared to \cite{CodS} comes from the fact that $W_V(Y_f)$ does not belong to a space that enables us to close the contraction argument used there at low frequencies. In \cite{CodS}, this was solved by using inhomogeneous Besov spaces developed in \cite{GNT}. In dimension 3, dispersion is not sufficient to conclude in the same way. Nevertheless, higher order iterates as $W_V^2(Y_f)$ are in function spaces that enable us to close the contraction argument, and this is why we consider the equation on $X - Y_f - W_V(Y_f)$. 
%
%\end{remark}

\subsection{Organization and aknowledgements}

The paper is organized as follows. We first set notations in Section \ref{sec:notations}. In Section \ref{sec:proof}, we prove the main Theorem \ref{th:main}, provided some linear and bilinear bounds hold true. These bounds are proved afterwards: Section \ref{sec:linear} deals with the linearised evolution around $Y$, Section \ref{sec:spefquad} provides bounds for the specific quadratic term $Q_2$, and Section \ref{sec:remaining} gives bounds for the remaining quadratic terms. In Section \ref{sec:coro} we prove Corollary \ref{corollaire}. In Appendix \ref{sec:dim2}, we make a remark about dimension 2. In Appendix \ref{sec:tech}, we give some further insight about Wiener integrals and recall some results about Strichartz estimates.\\

C. Collot is supported by the ERC-2014-CoG 646650 SingWave. A-S de Suzzoni is supported by ESSED ANR-18-CE40-0028. Part of this work was done when C. Collot was working at New York University, and he thanks the Courant Institute.

\section{Notations}\label{sec:notations}

\begin{notation}[Fourier transform] We define the Fourier and inverse Fourier transforms with the following constants : for $g \in \mathcal S(\mathbb R^d)$,
$$
\hat g (\xi) = \mathcal F (g) (\xi) = \int_{\R^d} g(x) e^{-ix\xi}dx, \quad \quad \mathcal F^{-1} (g) (x) = (2\pi)^{-d}\int_{\R^d} g(\xi) e^{ix\xi}d\xi.
$$
We use the following notation for Fourier multipliers:
$$
|\nabla |^s g=\mathcal F^{-1}(|\xi|^s \hat g),\quad \quad \langle \nabla \rangle^s g=\mathcal F^{-1}(\langle \xi \rangle^s \hat g).
$$
\end{notation}

\begin{notation}[Time-space norms] For $p,q\in [1,\infty]$, we denote by $L^p_tL^q_x$ the space $L^p(\R,L^q(\R^3))$. For $s\geq 0$ we denote by $L^p_tW^{s,q}_x$ the space $ L^p(\R,W^{s,q}(\R^3))$ which has norm:
$$
\| g \|_{L^p_tW^{s,q}_x}=\| \langle \nabla \rangle^s g \|_{L^p_tL^q_x}
$$
and recall this norm is equivalent to that of $L^p_tW^{k,q}_x$ for $s=k\in \mathbb N$. In the case $q=2$ use the notation $L^p_tH^s_x$ for $L^p_tW^{s,2}_x$ and $L^p_t \dot H^s_x$ for the space $L^p(\R, \dot H^s(\R^3))$. When $p=q$, we may write $L^p_{t,x}$ for $L^p_tL^p_x$.

\end{notation}

\begin{notation}[Probability-time-space norms]
For $p,q\in [1,\infty]$ and $s\geq 0$, we denote by $L^2_\omega L^p_t L^q_x$ the space $L^2(\Omega,L^p(\R,L^q(\R^d)))$, by $L^2_\omega L^p_t W^{s,q}_x$ the space $L^2(\Omega,L^p(\R,W^{s,q}(\R^d)))$. In the case $q=2$ we also write $L^2_\omega L^p_t H^s_x=L^2_\omega L^p_tW^{s,2}_x$ and $L^2_\omega L^p_t \dot H^s_x =L^2(\Omega,L^p(\R, \dot H^s(\R^3)))$.
\end{notation}

\begin{notation}[Time-space-probability norms]
For $p,q\in [1,\infty]$ we write 
\[
L^p_tL^q_x L^2_\omega= L^p(\R,L^q(\R^d(L^2(\Omega))).
\]
 For $p,q\neq (1,\infty)$ and $s\in \R$, we abuse notations and denote by $L^p_t W^{s,q}_x L^2_\omega $ the vector valued Bessel-potential space $(1-\Delta_x)^{-s/2}L^p(\R,L^{q}(\R^d,L^2(\Omega)))$ with norm:
$$
\| u \|_{L^p_t W^{s,q}_x L^2_\omega}= \| \langle \nabla \rangle^{s} u \|_{L^p_tL^q_xL^2_\omega}.
$$
Note from the extension of the Littlewood-Paley theory to $\ell^2$ valued functions, that $\| \langle \nabla \rangle^s u \|_{L^p_xL^2_\omega}\lesssim \| \langle \nabla \rangle^{s'} u \|_{L^p_xL^2_\omega}$ if $s\leq s'$, and that for $k\in \mathbb N$:
\be \label{id:vectorvaluedbesselpotential}
\sum_{j\leq k} \| \nabla^j u \|_{L^q_x L^2_\omega}\approx  \| \langle \nabla \rangle^k u \|_{L^q_xL^2_\omega}
\ee

In the case $q=2$ remark that $L^p_tW^{s,2}_xL^2_\omega =L^p_t L^2_\omega H^{s}_x$ by Fubini.

\end{notation}

\section{Proof of the main Theorem \ref{th:main}}\label{sec:proof}

The issue is to find the proper functional framework, which fits the study of the linearised problem and allows to bound the nonlinear terms. The heart of the proof is Proposition \ref{prop-principle}.

We introduce the notation
\be \label{id:defQ2}
Q_2(U,V) := 2\textrm{Re }\E\Big[\overline{W_V(Y)} W_U(Y) + \bar Y  \Big(W_V(W_U(Y)) + W_U(W_V(Y))\Big) \Big].
\ee
We now set 
\be \label{id:ThetaZ}
\Theta_Z = L^2_\omega\mathcal C(\R, H^{1/2}_x) \cap L^2_\omega L^5_{t,x} \cap L^2_\omega L^{10/3}_t W^{1/2,10/3}_x,
\ee
\be \label{id:ThetaV}
\Theta_V = L^2_{t}H^{1/2}_x \cap L^{5/2}_{t,x},
\ee
and 
\be \label{id:Theta0}
\Theta_0 = L^2_\omega H^{1/2}_x \cap L^{3/2}_x L^2_\omega.
\ee

\begin{proposition}\label{prop-principle} Assume that the spaces $\Theta_Z, \Theta_V$ and $\Theta_0$ defined by \fref{id:ThetaZ}, \fref{id:ThetaV} and \fref{id:Theta0} satisfy, with $\Theta = \Theta_Z \times \Theta_V$ the list of the following properties:
\begin{description}
\item [Minimal Space]  $\Theta_Z$ is continuously embedded in $\mathcal C (\R, L^2_\omega H^{1/2}) $.
\item [Initial datum] $\|C_0\|_\Theta \lesssim \|Z_0\|_{\Theta_0}$,
\item [Linear invertibility] $Id-L$ is invertible as a continuous operator from $\Theta$ to $\Theta$,
\item [Continuity of $w*$] $V\mapsto w* V$ is continuous on $\Theta_V$,
\item [First quadratic term on $Z$] $\|W_V(Z)\|_{\Theta_Z} \lesssim \|Z\|_{\Theta_Z} \|V\|_{\Theta_V}$,
\item [Second quadratic term on $Z$] $\|W_V(W_U(Y))\|_{\Theta_Z} \lesssim \|U\|_{\Theta_V} \|V\|_{\Theta_V}$,
\item [Embedding] $\Theta_Z \times \Theta_Z$ is continuously embedded in $\Theta_V$, as in for all $u,v \in \Theta_Z$, $\|E(uv)\|_{\Theta_V} \lesssim \|u\|_{\Theta_Z} \|v\|_{\Theta_Z}$,
\item [First quadratic term on $V$] $\|Q_1(Z,V)\|_{\Theta_V} \lesssim \|Z\|_{\Theta_Z} \|V\|_{\Theta_V}$,
\item [Second quadratic term on $V$] $
\big \|Q_2(U,V) \big \|_{\Theta_V} \lesssim \|U\|_{\Theta_V} \|V\|_{\Theta_V},
$
\item[Scattering spaces] $W_V(Y)$ belongs to $L^5_t W^{1/2,5}_xL^2_\omega$ if $V\in \Theta_V$, $\Theta_Z$ is continuously embedded in $L^{10/3}_t W^{1/2,10/3}_x L^2_\omega$ and
$$
\big\|\int_{0}^\infty S(t-\tau) \Big( (w*V(\tau)) Y(\tau) \Big)d\tau\big\|_{L^{\infty}_tL^3_x L^2_\omega} \lesssim \|V\|_{L^2_{t,x}},
$$
\end{description}
then the conclusions of Theorem \ref{th:main} hold true.
\end{proposition}

We prove Proposition \ref{prop-principle} below at the end of this section. Proving the main Theorem \ref{th:main} amounts to prove that the hypotheses of Proposition \ref{prop-principle} are satisfied, which is done in the remaining part of the paper.

\begin{proof}[Proof of Theorem \ref{th:main}]

Thanks to Proposition \ref{prop-principle}, we only have to check the hypotheses of this Proposition. Below are their meanings and their locations in the rest of the paper.
\begin{itemize}
\item \textbf{Minimal space} is a consequence of the definition \fref{id:ThetaZ} of $\Theta_Z$ and of the result of Step 1 in the proof of Lemma \ref{lem:strichartzrandom}. It states that $\Theta_Z$ has to be included in $\mathcal C(\R, L^2_\omega H^s_x)$. This space is the one in which local well-posedness holds.
\item \textbf{Initial datum} is satisfied from Lemma \ref{lem:strichartzrandom} and the bound \fref{bd:improvedrandomtopotential} of Proposition \ref{pr:freeandpotential}. It means that the constant source term $C_0$ is controlled by the initial datum.
\item \textbf{Linear invertibility} is satisfied from Proposition \ref{pr:bdL2}. It corresponds to the invertibility of the linear part of the equation.
\item \textbf{Continuity of $w*$} is satisfied from the very definition \fref{id:ThetaZ} of $\Theta_V$ and the hypothesis \fref{id:hpw}.
\item \textbf{First quadratic term on $Z$} is satisfied from \fref{bd:quadperturbation}.
\item \textbf{Second quadratic term on $Z$} is satisfied from the estimate \fref{bd:quadperturbation}.
\item \textbf{Embedding} is satisfied from \fref{bd:quadembed}.
\item \textbf{First quadratic term on $V$} is satisfied from \fref{bd:Q1}.
\item \textbf{Second quadratic term on $V$} is satisfied from Proposition \ref{pr:continuiteQ2}. Those last items are the bilinear estimates required to perform a contraction argument.
\item \textbf{Scattering spaces} is satisfied from Proposition \ref{pr:continuiteQ2}, the definition \fref{id:ThetaZ} of $\Theta_Z$, and \fref{bd:potentialtorandom} with $s=1/2$, $p=\infty$ and $q=3$. It corresponds to the description of the minimal spaces to which $W_V(Y)$ and $Z$ must belong to get scattering as described in Theorem \ref{th:main}.
\end{itemize}

\end{proof}

We now prove Proposition \ref{prop-principle} to conclude the proof of the main Theorem \ref{th:main}.

\begin{proof}[Proof of Proposition \ref{prop-principle}]

We prove it in two steps, in a standard fashion for a scattering problem in nonlinear dispersive evolution equations. First we show global existence in the vicinity of $Y$ and global bounds using a fixed point argument. Then we prove scattering by using one more time these bounds on the fixed point equation.

\textbf{Step 1} \emph{Global existence near $Y$}. Because of item Linear invertibility, we have that $(Id-L)^{-1}$ is a well-defined and continuous operator from $\Theta$ to itself. We thus consider in view of \fref{fixedpoint} and \fref{fixedpoint2} the following map
\be \label{id:fixedpointmainth}
A_{Z_0} \begin{pmatrix} Z \\V \end{pmatrix} = (Id-L)^{-1} \Big[ C_0 + Q\begin{pmatrix} Z \\V' \end{pmatrix} \Big].
\ee
and claim that for $Z_0$ small enough in $\Theta_0$, there exists $R$ such that $A_{Z_0}$ maps $B_{\Theta}(0,R)$ onto itself and is a contraction. We now show the claim.

We first show that $B_{\Theta}(0,R)$ is stable under $A_{Z_0}$. Assume $\begin{pmatrix} Z \\V \end{pmatrix} \in B_\Theta (0,R)$. We have 
$$
\big\| A_{Z_0} \begin{pmatrix} Z \\V \end{pmatrix} \big\|_{\Theta} \leq \|(Id-L)^{-1}\|_{\Theta \rightarrow \Theta} \Big( \|C_0\|_\Theta + \big\| Q \begin{pmatrix} Z \\V' \end{pmatrix} \big\|_\Theta \Big) \lesssim \Big( \|C_0\|_\Theta + \big\| Q \begin{pmatrix} Z \\V' \end{pmatrix} \big\|_\Theta \Big).
$$
Because of item Initial datum, there exists $C_1$ such that
$$
\big\| A_{Z_0} \begin{pmatrix} Z \\V \end{pmatrix} \big\|_{\Theta} \leq C_1 \|Z_0\|_{\Theta_0} +C_1\big\| Q \begin{pmatrix} Z \\V' \end{pmatrix} \big\|_\Theta .
$$
Because of items \begin{itemize}
\item Continuity of $w*$,
\item First quadratic term on $Z$,
\item Second quadratic term on $Z$ applied to $U=V$,
\item Embedding, 
\item First quadratic term on $V$,
\item Second quadratic term on $V$ applied to $U=V$,
\end{itemize}
 there exists $C_2$ such that
$$
\big\|A_{Z_0} \begin{pmatrix}
Z\\V
\end{pmatrix}\big\|_\Theta \leq C_1 \|Z_0\|_{\Theta_0} + C_2 R^2 \leq R,
$$
where the last inequality holds if $C_1 \|Z_0\|_{\Theta_0}\leq \frac{R}{2}$ and $C_2R\leq \frac12$. This is possible if $4C_1C_2 \|Z_0\|_{\Theta_0} \leq 1$, that is, for $Z_0$ small enough. Under these conditions, the ball $B_\Theta (0,R)$ is stable under $A_{Z_0}$.

We now show that for $R$ small enough (up to taking $\| Z_0\|_{\Theta_0}$ small enough) that $A_{Z_0}$ is contracting on $B_\Theta (0,R)$. 

Note that we have the identity
$$
W_{V+U}^2  = W_V^2 + W_V\circ W_U + W_U \circ W_V + W_U^2,
$$
hence $W_V\circ W_U$ is the bilinearization of $W_V^2$. We have similarly that
$$
Q_2(V+U) = Q_2(V) + Q_2(U,V) + Q_2(U)
$$
and $2Q_2(V) = Q_2(V,V)$ so that $Q_2(\cdot,\cdot)$ is the bilinearization of $Q_2(\cdot)$.

We take $(Z_1,V_1),(Z_2,V_2)\in B_\Theta(0,R)$. One first has the identity:
\be \label{id:AV1-V2inter}
A_{Z_0}\begin{pmatrix}
Z_1\\ V_1
\end{pmatrix} - A_{Z_0}\begin{pmatrix}
Z_2\\ V_2
\end{pmatrix} = (Id-L)^{-1} \left( Q\begin{pmatrix}
Z_1\\ V_1'
\end{pmatrix} - Q\begin{pmatrix}
Z_2\\ V_2'
\end{pmatrix} \right) .
\ee
%We have from \fref{def:WV} the bilinearity of $(Z,V) \mapsto W_{V'}(Z)$:
%$$
%W_{V_1'}(Z_1) - W_{V_2'}(Z_2) = W_{V_1'}(Z_1 - Z_2) + W_{V_1' - V_2'} (Z_2)
%$$
%and thus by item QuadCl1 and Conv, there exists $C_8$ such that
%$$
%\|W_{V_1'}(Z_1) - W_{V_2'}(Z_2) \|_{\Theta_Z} \leq C_8 R \big\| \begin{pmatrix} Z_1\\V_1 \end{pmatrix} -  \begin{pmatrix} Z_2\\V_2 \end{pmatrix} \big\|_{\Theta}.
%$$
%What is more, again from \fref{def:WV}:
%$$
%W_{V_1'}^2 (Y) - W_{V_2'}^2(Y) = W_{V_1'}(W_{V_1'-V_2'}(Y)) + W_{V_1'-V_2'}(W_{V_2'}(Y))
%$$
%and thus, by item QuadCl2 and Conv, there exists $C_9$ such that
%$$
%\| W_{V_1'}^2 (Y) - W_{V_2'}^2(Y)\|_{\Theta_Z} \leq C_9 R \|V_1 - V_2\|_{\Theta_V}.
%$$
%Then, we have by QuadEmbed that there exists $C_{10}$ such that:
%$$
%\| \E(|Z_1|^2) - \E(|Z_2|^2)\|_{\Theta_V}= \|\E(\bar Z_1 (Z_1 - Z_2)) + \E(\overline{(Z_1 - Z_2)} Z_2)\|_{\Theta_V} \leq C_{10} R \|Z_1 - Z_2\|_{\Theta_Z}.
%$$
%By bilinearity of $Q_1$ from \fref{def:Q1}, we have 
%$$
%Q_1(Z_1,V_1') - Q_1(Z_2,V_2') = Q_1(Z_1 - Z_2, V_1') + Q_1(Z_2,V_1'-V_2'),
%$$
%and thus, by item QuadV1 and Conv, there exists $C_{11}$ such that
%$$
%\|Q_1(Z_1,V_1') - Q_1(Z_2,V_2')\|_{\Theta_V} \leq C_{11} R \big\| \begin{pmatrix} Z_1\\V_1 \end{pmatrix} -  \begin{pmatrix} Z_2\\V_2 \end{pmatrix} \big\|_{\Theta}.
%$$
%Finally, we have from \fref{def:Q2} and \fref{id:defQ2}:
%$$
%Q_2(V_1') - Q_2(V_2') = \frac12 Q_2 (V_1' + V_2', V_1'-V_2')
%$$
%and thus, by item QuadV2 and Conv, there exists $C_{12}$ such that
%$$
%\|Q_2(V_1') - Q_2(V_2')\|_{\Theta_V} \leq C_{12} R \|V_1-V_2\|_{\Theta_V}.
%$$
Using the different items of Proposition \ref{prop-principle} in the same way as to prove the stability of the ball (except item Initial datum), we get that there exists $C_3$ such that
$$
\big \| A_{Z_0}\begin{pmatrix} Z_1\\V_1 \end{pmatrix} - A_{Z_0} \begin{pmatrix} Z_2\\V_2 \end{pmatrix} \big\|_{\Theta} \leq C_{3} R \big\| \begin{pmatrix} Z_1\\V_1 \end{pmatrix} -  \begin{pmatrix} Z_2\\V_2 \end{pmatrix} \big\|_{\Theta}
$$
and thus $A_{Z_0}$ is contracting if $C_{3}R <1$. 

Therefore, if $\|Z_0\|_{\Theta_0} \leq \frac1{4C_1C_2} $ and $\|Z_0\|_{\Theta_0} < \frac1{2C_1C_{3}}$, then from Banach-Picard fixed point Theorem there exists $R>0$ such that $A_{Z_0}$ admits a unique fixed point in $B_{\Theta}(0,R)$. We thus deduce that the Cauchy problem 
$$
\left \lbrace{ \begin{array}{l l}
i\partial_t Z = (m-\Delta) Z + (w*V) (Z+W_V(Y)) ,\quad \quad Z (t=0) = Z_0, \\
V = \E(|Z|^2) + 2\textrm{Re }\E (\bar Y Z) + 2\textrm{Re }\E (\overline{ W_V(Y)}(Y+ Z)) + \E(|W_V(Y)|^2)\\
\end{array}} \right.
$$
has a unique solution $Z \in \mathcal C(\R,L^2_\omega H^{1/2}_x)$ for small enough $Z_0$ and therefore, the Cauchy problem
$$
\left \lbrace{\begin{array}{c}
i\partial_t X  = -\Delta X + (w*\E(|X|^2))X \\
X(t=0) = X_0 := Z_0 + Y(t=0)
\end{array}} \right.
$$
is globally well-posed in the space \fref{def:spacegwp} for $Z_0$ small enough in $\Theta_0$.

\textbf{Step 2} \emph{Scattering to linear waves}. We now prove the scattering part, and only consider the case $t\rightarrow \infty$ as the case $t\rightarrow -\infty$ is similar. We rewrite the decomposition \fref{fixedpoint} using \fref{def:WV} as:
\be \label{def:Z+}
\begin{array}{l l l l l}
Z(t)&=& S(t)\Big( Z_0 -i \int_{0}^\infty S(-\tau) \Big[ V'(\tau) (W_{V'}(Y)(\tau) +Z(\tau)) \Big] d\tau \Big)\quad \quad \quad &\}&=S(t)Z_+\\
&&\quad \quad +iS(t ) \int_{t}^\infty S( - \tau) \Big[ V'(\tau) (W_{V'}(Y)(\tau) +Z(\tau)) \Big] d\tau \quad \quad \quad &\}&=S(t)\bar Z(t)
\end{array}
\ee
As a consequence of the standard dual Strichartz's \fref{bd:dualstrichartz} inequality and of the Christ-Kiselev Lemma \cite{ChKi}, for any $f\in L^2_\omega L^{10/7}_{t,x}$, we have 
\be \label{bd:strichartzdual}
\left\|  \int_0^t S(t-\tau)f(\tau)d\tau \right\|_{L^2_\omega L^{10/3}_{tx} \cap L^2_\omega  L^{5}_tL^{30/11}_x\cap L^2_\omega L^{\infty}_t L^2_x }\lesssim \| f \|_{ L^2_\omega L^{10/7}_{\tau,x}}.
\ee
Applying \fref{bd:strichartzdual} we obtain for $t\geq 0$, using Minkowski inequality as $2\leq \infty$ and $10/7\leq 2$:
$$
\big\| \bar Z(t )\big\|_{L^2_\omega H^{1/2}_x} \lesssim \|{\bf 1}_{[t,\infty)}V' W_{V'}(Y)\|_{L^{10/7}_\tau W^{1/2,10/7}_x L^2_\omega}+\| {\bf 1}_{[t,\infty)}V' Z\|_{L^2_\omega L^{10/7}_\tau W^{1/2,10/7}_x}.
$$
We now appeal to two Leibniz-type bounds for $0\leq s \leq 1$, $1<p,q_1,q_2,r_1,r_2<\infty$ with $1/p=1/q_1+1/r_1=1/q_2+1/r_2$. The first one is a classical result of Littlewood-Paley theory:
\be \label{bd:fractionalleibniz1}
\| \langle \nabla \rangle^{s}fg \|_{L^p_x}\lesssim \| \langle \nabla \rangle^{s} f\|_{L^{q_1}_x}\| g\|_{L^{r_1}_x}+\| \langle \nabla \rangle^{s} g\|_{L^{q_2}_x}\| f\|_{L^{r_2}_x},
\ee
and we claim the second one of vectorial type:
\be \label{bd:fractionalleibniz2}
\| \langle \nabla \rangle^{s}fg \|_{L^p_x L^2_\omega}\lesssim \| \langle \nabla \rangle^{s} f\|_{L^{q_1}_xL^2_\omega}\| \langle \nabla \rangle^{s} g\|_{L^{r_1}_x L^2_\omega}.
\ee
Indeed, \fref{bd:fractionalleibniz2} for $s=0$ is a direct consequence of H\"older inequality, while for $s=1$ we get it as a consequence of $\| \nabla (fg) \|_{L^p_x L^2_\omega}\leq \| \nabla f \|_{L^{q_1}_x L^2_\omega}\| g \|_{L^{r_1}_x L^2_\omega}+\| \nabla g \|_{L^{r_1}_x L^2_\omega}\| f \|_{L^{q_1}_x L^2_\omega}\lesssim \| f \|_{W^{1,q}_xL^2_\omega} \| g \|_{W^{1,q}_xL^2_\omega}$ (using Leibniz, H\"older and \fref{id:vectorvaluedbesselpotential}). We then get \fref{bd:fractionalleibniz2} for $0<s<1$ via complex interpolation. Using \fref{bd:fractionalleibniz2} with $1/2+1/5=7/10$ and \fref{bd:fractionalleibniz1} with $2/5+3/10=7/10$ and $1/2+1/5=7/10$ we bound:
\bee
\big\| \bar Z(t )\big\|_{L^2_\omega H^{1/2}_x}& \lesssim& \|{\bf 1}_{[t,\infty)}V'\|_{L^2_{\tau}H^{1/2}_x} \|W_{V'}(Y)\|_{L^{5}_\tau W^{1/2,5}_x L^2_\omega} + \|{\bf 1}_{[t,\infty)}V'\|_{L^{5/2}_{\tau,x} } \|Z\|_{L^2_\omega L^{10/3}_\tau W^{1/2,10/3}_x}\\
&& \quad + \|{\bf 1}_{[t,\infty)}V'\|_{L^{2}_{\tau}H^{1/2}_x } \|Z\|_{L^2_\omega L^{5}_{\tau,x}} \quad \longrightarrow 0
\eee
as $t\rightarrow \infty$, because $\| Z \|_{\Theta_Z}+\| V' \|_{\Theta_V}<\infty$ from Step 1 and item Continuity of $w*$, and be  due to item Scattering spaces. The same argument to bound $\bar Z(t)$ in $L^2_\omega H^{1/2}_x$ applies to bound $Z_+$ defined by \fref{def:Z+}, and we obtain that $Z_+$ belongs to $L^2_\omega H^{1/2}_x$. We have established that $ \|Z(t) - S(t)Z_+\|_{L^2_{\omega} H^{1/2}_x}\rightarrow 0$ as $t\rightarrow \infty$, which proves \fref{id:mainresult}. Finally, note that
$$
W_{V'}(Y) = -iS(t)\int_{0}^\infty S(-\tau) (V'(\tau) Y(\tau)) d\tau  + i\int_{t}^\infty S(t-\tau) (V'(\tau) Y(\tau)) d\tau.
$$ 
Because of items Scattering spaces and Continuity of $w*$, we have that for all $t \in \R$,
$$
\big\|\int_{A}^\infty S(t-\tau) (V'(\tau) Y(\tau)) d\tau \big\|_{L^3_x L^2_\omega} \lesssim \|{\bf 1}_{[A,\infty)} V\|_{L^2_{t,x}}.
$$
From this bound, the above decomposition and Christ-Kiselev's lemma, \cite{ChKi}, $W_{V'}(Y)$ can be written as 
$$
W_{V'}(Y) = S(t)\left( \tilde Z_++ o_{L^3_x L^2_\omega}(1)\right) = S(t) \tilde Z_+ + o_{L^3_x L^2_\omega}(1)\quad \quad \mbox{as }t\rightarrow \pm \infty.
$$
This shows \fref{id:mainresult2}, and ends the proof of Proposition \ref{prop-principle}.

\end{proof}

The rest of the paper is devoted mainly to prove that the hypotheses of Proposition \ref{prop-principle} hold true.

\section{Linear analysis}\label{sec:linear}

The linearised evolution problem is $i\partial_t Z=(m-\Delta) Z+2Re \mathbb E(\bar Y Z)Y$. Setting $w*V=w*2Re \mathbb E(\bar Y Z)$ to be the linearised potential generated by the perturbation $Z$, it becomes:
\be \label{id:pblineaire}
\left\{ \begin{array}{l l} i\partial_t Z=(m-\Delta) Z+(w*V)Y, \quad Z(t=0)=Z_0,\\
V=2Re \mathbb E(\bar Y Z). \end{array}\right.
\ee
The aim of this section is to show the following dispersive estimates.

\begin{proposition} \label{pr:lineaire}

Under the hypotheses on $f$ and $w$ of Theorem \ref{th:main}, for any $Z_0\in \Theta_0$, the solution $Z(t),V(t)$ of \fref{id:pblineaire} is such that $Z(t)=S(t)Z_0+W_{w*V}(Y)$ with:
$$
\| S(t)Z_0\|_{\Theta_Z}+\| V(t)\|_{\Theta_V}\leq C \| Z_0 \|_{\Theta_0}, \quad \|W_{w*V}(Y)\|_{L^{\infty}_tL^3_xL^2_\omega}+ \| \langle \nabla \rangle^{1/2}W_{w*V}(Y)\|_{ L^5_{t,x}L^2_\omega  } \leq C \| Z_0 \|_{\Theta_0}.
$$

\end{proposition}

The proof of Proposition \ref{pr:lineaire} requires preliminary results and is done at the end of this section. In the first place we study the contributions in \fref{id:pblineaire} separately. To begin with, we study the free Schr\"odinger evolution of a random field $S(t)Z_0$. Then, we turn to the linearised potential field generated by such perturbation, that is, a potential field of the form $w*2Re\mathbb E(\bar YS(t)Z_0)$. Next, we analyse the perturbation $W_U(Y)$ generated by the effects of a potential field on the background equilibrium. Finally, the full linearised problem is studied appealing to the so-called linear response theory. More precisely, in response to an input potential $w*V$, the output potential created by the response of the equilibrium is $2Re \mathbb E(\bar Y W_{w*V}Y)$. This allows to solve \fref{id:pblineaire} since the potential $V$ is a good unknown and satisfies the fixed point equation:
\be 
V=2Re \mathbb E(\bar Y S(t)Z_0)+L_2(V)=2Re \mathbb E(\bar Y (S(t)Z_0+W_{w*V}(Y)))
\ee
compatible with notation \ref{def:L2}. The properties of $L_2$ and the invertibility of $Id-L_2$ are studied in the last subsection.

\subsection{Free evolution of random fields}

We first establish continuity and dispersive estimates for the free evolution of random fields. For the problem at hand, we can use homogeneous Strichartz estimates at regularity $1/2$ since it is less than $d/2=3/2$.

\begin{lemma} \label{lem:strichartzrandom}

For all $Z_0 \in L^2_\omega H^{1/2}_x$, we have $S(t)Z_0\in C_t(L^2_\omega H^{1/2}_x)$ and moreover:
\be \label{bd:strichartzrandom}
\displaystyle \| S(t)Z_0\|_{\Theta_Z}=\|S(t)Z_0\|_{L^2_\omega C_t( H^{1/2}_x)\cap  L^2_\omega L^5_{t,x}}+\| \langle \nabla \rangle^{\frac 12}S(t)Z_0\|_{ L^2_\omega L^{10/3}_{t,x}} \lesssim \|Z_0\|_{L^2_\omega H^{1/2}_x}.
\ee
\end{lemma}

\begin{proof}

\textbf{Step 1} \emph{An embedding.} We claim that $L^2_\omega C(\mathbb R,H^{1/2}_x)$ embeds continuously in $ C(\mathbb R,L^2_\omega H^{1/2}_x)$. Indeed, $\|u\|_{L^{\infty}_t L^2_\omega H^{1/2}_x}\leq \|u\|_{L^2_\omega L^{\infty}_tH^{1/2}_x}$ by Minkowski's inequality. Moreover, as almost surely in $\Omega$, $u(t')\rightarrow u(t)$ in $H^{1/2}_x$ as $t'\rightarrow t$, we obtain that $\| u(t')-u(t)\|_{L^2_\omega H^{1/2}_x}\rightarrow 0$ as $t'\rightarrow t$ as an application of Lebesgue's dominated convergence Theorem.

\textbf{Step 2} \emph{Proof of the Lemma}. This is a consequence of usual Strichartz and continuity estimates \fref{bd:homostrichartz} for $S(t)$ and Sobolev embeddings. First, $S(t)Z_0:\mathbb R\times \Omega \times \mathbb R^3\rightarrow \mathbb C$ is measurable (using for example the fact that $Z_0:\Omega \times \mathbb R^3\rightarrow \mathbb C$ is measurable with $\| Z_0\|_{L^2_\omega L^2_x}<\infty$ if and only if the same holds for $\hat Z_0$). Since $\| e^{it\Delta}f\|_{C_t(H^{1/2}_x)}\leq \| f \|_{H^{1/2}_x}$ for any $f\in H^{1/2}_x$, we then obtain $\|S(t)Z_0\|_{L^2_\omega C_t( H^{1/2}_x)}\leq \| f \|_{L^2_\omega H^{1/2}_x}$. Hence $S(t)Z_0\in C_t(L^2_\omega H^{1/2}_x)$ by applying the embedding of Step 1. Next, $(10/3,10/3)$ and $(5,30/11)$ are admissible pairs for the usual Strichartz estimates \fref{bd:homostrichartz} so that:
$$
\|S(t)Z_0\|_{L^2_\omega L^{10/3}_tL^{10/3}_x}+\|S(t)Z_0\|_{L^2_\omega L^{5}_tL^{30/11}_x} \lesssim \| Z_0 \|_{L^2_\omega L^2_x}.
$$
Hence we obtain from the first inequality above that:
$$
\|S(t)Z_0\|_{ L^2_\omega L^{10/3}_tL^{10/3}_x } \lesssim \| Z_0 \|_{L^2_\omega L^2_x}, \quad \quad \||\nabla|^{\frac 12}S(t)Z_0\|_{L^2_\omega L^{10/3}_tL^{10/3}_x}\lesssim \|Z_0\|_{L^2_\omega \dot H^{\frac 12}_x},
$$
while the second inequality with the embedding of $\dot W^{1/2,30/11}(\mathbb R^3)$ into $L^{5}(\mathbb R^3)$ give:
$$
\|S(t)Z_0\|_{L^2_\omega L^{5}_tL^{5}_x}\lesssim \| |\nabla|^{\frac 12}S(t)Z_0\|_{L^2_\omega L^{5}_tL^{30/11}_x}\lesssim  \| Z_0\|_{L^2_\omega \dot H^{\frac 12}_x}.
$$
The three bounds above establish the Lemma.

\end{proof}

It is instructive (and useful in the sequel) to consider initial perturbations of a specific form, which represent a perturbation of the distribution function $f$ of $Y_0$:
\be \label{def:Zg}
Z_g(\omega , x)=\int_{\mathbb R^3} dW(\xi)e^{i\xi.x}g(x,\xi).
\ee
Above, $g$ is the normalised (with respect to $Y_0$) distribution function of the perturbation $Z_g$, and encodes its correlation with $Y_0$ since there holds:
$$
\mathbb E\left(\bar Y_0(x)Z_g(y) \right)=\int_{\mathbb R^3} e^{i\xi.(y-x)}g(y,\xi)f(\xi)d\xi.
$$
Their regularity is measured through that of the normalised distribution function via the adapted homogeneous Sobolev spaces $\dot{\tilde H}^s_{x,\omega}$ associated to the norm:
$$
\| Z_g\|_{\dot{\tilde{H}}^s_{x,\omega}}= \| g \|_{L^2_\xi \dot H^s_x}.
$$
Due to the commutator relation $[\Delta,e^{ix.\xi}]=e^{i\xi.x}(|\xi|^2-2\xi.\nabla)$, the free Schr\"odinger evolution will induce an additional transport with speed $\xi$ on each "component" $dW(\xi)$ and eventually result in a spreading effect. We introduce the operator:
$$
S_\xi (t)(u)=(S(t)u)(x-2\xi t), \quad \quad \mathcal F(S_\xi(t) u)(\eta)=e^{-it|\eta|^2-2i\eta.\xi}\hat u(\eta)=e^{-it (-\Delta -2i \xi \cdot \grad)}\hat u(\eta).
$$

\begin{lemma}\label{lem-Fouriertrick} For all $U \in \mathcal S'(\R^3)$, $t,\tau\in \mathbb R$, and $\xi \in \mathbb R^3$ we have
\be \label{id:comutateurphaseschrod}
S(t-\tau)\Big[  e^{-i\tau(m+\xi^2) + ix\cdot \xi }U  \Big]= e^{-it(m+|\xi|^2) + ix\cdot \xi} S_\xi (t-\tau) U.
\ee

\end{lemma}

\begin{proof} If $U\in \mathcal  \mathcal S(\R^3)$, we pass in Fourier mode to obtain
\bee
\mathcal F\Big(S(t-\tau)\Big[ U  e^{-i\tau(m+|\xi|^2) + ix\cdot \xi}  \Big]\Big)(\eta) &=& e^{-i(t-\tau)(m+|\eta|^2)} \mathcal F \Big( U  e^{-i\tau(m+|\xi|^2) + ix\cdot \xi}\Big)(\eta)\\
&=& e^{-i(t-\tau)(m+|\eta|^2)-i\tau(m+|\xi|^2)}\mathcal F \Big( U  e^{ix\cdot \xi}\Big)(\eta) \\
&=& e^{-i(t-\tau)(m+|\eta|^2)-i\tau(m+|\xi|^2)}\hat U( \eta - \xi).
\eee
In the expression above we compute:
$$
-i(t-\tau)(m+|\eta|^2)-i\tau(m+|\xi|^2) =-it (m+|\xi|^2)-i(t-\tau)(|\eta-\xi|^2-2\xi.(\eta-\xi)).
$$
Hence we recognise in the previous identity:
\bee
\mathcal F\Big(S(t-\tau)\Big[ U  e^{-i\tau(m+|\xi|^2) + ix\cdot \xi}  \Big]\Big)(\eta) &=& e^{-it (m+|\xi|^2)-i(t-\tau)(|\eta-\xi|^2-2\xi.(\eta-\xi))}\hat U( \eta - \xi) \\
&=&e^{-it(m+|\xi|^2) } \left(e^{-i(t-\tau) (-\Delta -2i \xi \cdot \grad)} \hat U\right)(\eta-\xi) \\
&=& e^{-it(m+|\xi|^2) }\mathcal F\left(e^{i\xi.x}e^{-i(t-\tau) (-\Delta -2i \xi \cdot \grad)} U  \right).
\eee
Applying the inverse Fourier transform proves \fref{id:comutateurphaseschrod}, and the result for $U\in \mathcal S'(\mathbb R^3)$ follows by duality.
\end{proof}

The spaces $\dot{\tilde{H}}^s_{x,\omega}$ are well adapted to measure the regularity of the free evolution of random fields of the form $Z_g$, as standard Strichartz estimates hold true.

\begin{lemma} \label{lem:strichartzZg}

For any $0\leq s < 3/2$ and $2\leq p,q\leq \infty$ satisfying $\frac2{p}+\frac3{q} = \frac 32-s$, for any $Z_g\in \dot{\tilde H}^{s}_{x,\omega}$ there holds:
$$
\| S(t)Z_g \|_{L^p_tL^q_xL^2_\omega}\leq C \| Z_g \|_{\dot{\tilde H}^{s}_{x,\omega}}. 
$$

\end{lemma}

\begin{proof}
Using \fref{id:comutateurphaseschrod}, \fref{id:fubinischrodingerwiener} and \fref{id:correlationgaussian} we get that
$$
S(t)Z_g=\int_{\mathbb R^3} dW(\xi)e^{-it(m+|\xi|^2)+ix.\xi}S_\xi (t)g(\xi,x), \quad \quad \mathbb E(|S(t)Z_g|^2)=\int_{\mathbb R^3} |S_\xi (t)g(\xi,x)|^2d\xi.
$$
Therefore, applying Minkowski inequality and then standard homogeneous Strichartz estimates \fref{bd:dualstrichartz} give the desired result:
$$
\| S(t)Z_g \|_{L^p_tL^q_xL^2_\omega}=  \| S_{\xi}(t)(g(\xi,x))\|_{L^p_tL^q_xL^2_\xi} \lesssim \| S_{\xi}(t)(g(\xi,x))\|_{L^2_\xi L^p_tL^q_x}  \lesssim \| g \|_{L^2_\xi\dot H^s_x}=\| Z_g\|_{\dot{\tilde H}^s_{x,\omega}}.
$$
\end{proof}

\subsection{Potential induced by a perturbation and perturbation induced by a potential}

As a direct consequence of Lemma \ref{lem:strichartzrandom} and of the fact that the expectation of $|Y|^2$ is uniformly bounded, we obtain the following bounds for the potential generated by free evolution of random fields.

\begin{lemma}

For any $Z_0\in L^2_\omega H^{1/2}_x$ or $Z_0\in L^2_\omega H^{1}_x$ there holds:
\be \label{bd:strichartzrandomtopotential}
\| \mathbb E(\bar YS(t)Z_0)\|_{L^5_tL^5_x}\lesssim \| Z_0 \|_{L^2_\omega H^{1/2}_x}.
\ee

\end{lemma}

\begin{proof}

By Cauchy-Schwarz, H\"older, and Minkowski inequalities
\bee
\|\mathbb E(\bar Y S(t) Z_0)\|_{L^{5}_{t,x}} &\leq& \| \|Y\|_{L^2_\omega} \| S(t) Z_0\|_{L^2_\omega}\|_{L^{5}_{t,x}}\\
&&\quad \quad \lesssim \|Y\|_{L^\infty_{t,x},L^2_\omega} \|S(t)Z_0\|_{L^{5}_{t,x},L^2_\omega} \lesssim \|S(t)Z_0\|_{L^2_\omega L^{5}_{t,x}} \lesssim \|Z_0\|_{L^2_\omega H^{1/2}_{x}}
\eee
where we used Lemma \ref{lem:strichartzrandom} for the last inequality.

\end{proof}

We now establish additional long time decay. We also study perturbations $W_U(Y)$ generated by the effects of a potential field on the background equilibrium. The equilibrium $Y$ being a random superposition of linear waves $e^{i\xi.x}$, it will induce additional transport on top of the free Schr\"odinger evolution. As an application of Lemma \ref{lem-Fouriertrick} we obtain the following.

\begin{lemma} \label{lem:expressionWVY} We have for all $V \in \mathcal S(\R\times \R^3)$ 
$$
-i\int_{\R} S(t-\tau) [V(\tau) Y(\tau)] d\tau = -i \int_{\mathbb R^3} dW(\xi) f(\xi) e^{-it(m+ |\xi|^2)} e^{i\xi \cdot x} \int_{\R} S_\xi (t-\tau) \Big( V(\tau)\Big) d\tau.
$$  
\end{lemma}

\begin{proof} We set
\bea
\nonumber W_V^{\R}(Y) &=& -i \int_{\R} S(t-\tau) \Big[ V(\tau) Y(\tau) \Big] d\tau \\
\label{def:mathcalWV}&&\quad \quad = -i \int_{\R}S(t-\tau)\Big[ V(\tau) \int_{\mathbb R^3} dW(\xi) f(\xi) e^{-i\tau(m+|\xi|^2) + ix\cdot \xi} \Big] d\tau
\eea
where we replaced $Y$ by its definition in the last identity. Since $W(\xi)$ and $f(\xi)$ do not depend on the space variable, we can factorize them out of $S(t-\tau)$ using \fref{id:fubinischrodingerwiener} and get
$$
W_V^{\R}(Y) = -i \int_{\R} \int_{\mathbb R^3} dW(\xi) f(\xi) S(t-\tau)\Big[ V(\tau)  e^{-i\tau(m+|\xi|^2) + ix\cdot \xi} \Big] d\tau.
$$
We can then exploit the fact that $V \in \mathcal S(\R\times \R^3)$ and \fref{id:fubinilebesguewiener} to exchange the integrals, and then use Lemma \ref{lem-Fouriertrick} to get the desired identity
\bee
W_V^{\R}(Y) &=& -i \int_{\mathbb R^3} dW(\xi) f(\xi) \int_{\R}  S(t-\tau)\Big[ V(\tau)  e^{-i\tau(m+|\xi|^2) + ix\cdot \xi} \Big] d\tau\\
&&\quad =-i \int_{\mathbb R^3} dW(\xi) f(\xi)  e^{-it(m+|\xi|^2) + ix\cdot \xi} \int_{\R}  S_\xi (t-\tau)\Big[ (V(\tau))  \Big] d\tau.
\eee
\end{proof}

The next proposition exploits the spreading induced by the transport described in the previous Lemma. It studies simultaneously via a duality argument the potential created by the free evolution of random fields, either general or of the form $Z_g$, and the random field created by the application of a potential to the background equilibrium $Y$.

\begin{proposition} \label{pr:freeandpotential}

For any $s\in \mathbb R$, there holds for $C=C\left(\int_0^\infty |h|(r)dr \right)$:
\be \label{bd:improvedrandomtopotential2}
\| \mathbb E\left(\bar Y (t) S(t)Z_g\right)\|_{L^2_t\dot H^{s+\frac 12}_x}\leq C \| Z_g \|_{\dot{\tilde H}^{s}_{x,\omega}}.
\ee
Moreover, for any $2\leq p,q\leq \infty $ and $0\leq s <3/2$ such that $\frac2{p}+\frac3{q} = \frac 32-s$ there holds:
\begin{equation} \label{bd:potentialtorandom}
\| \int_{\R} S(t-\tau)\left(Y(\tau)U(\tau)\right)d\tau \|_{L^p_t L^q_xL^2_\omega\cap C_t( \dot{\tilde H}^s_{x,\omega})}\leq C \| U \|_{L^2_t\dot H^{s-\frac 12}_x}.
\end{equation}
In particular for any $Z_0\in L^{\frac 32}_xL^2_\omega$ and $C=C\left(\int_0^\infty (|h|+|h''|+r^{-1}|h'|)dr \right)$:
\be \label{bd:improvedrandomtopotential}
\| \mathbb E\left([\nabla \bar Y (t)] S(t)Z_0\right)\|_{L^2_tL^2_x}+\| \mathbb E\left(\bar Y (t) S(t)Z_0\right)\|_{L^2_tL^2_x}\leq C \| Z_0 \|_{L^{\frac 32}_xL^2_\omega}.
\ee

\end{proposition}

\begin{remark} \label{re:spreading}

In comparison with usual estimates for the Schr\"odinger evolution, the first bound of the above Lemma (from which the others follow) gives an additional $L^2$ integrability in time at a cost of half a derivative in low frequencies. The explanation is as follows: the correlation between the free evolution $S(t)Z_g$ and the background equilibrium $Y$ is:
$$
\mathbb E(\bar YS(t)Z_g)=\int_{\mathbb R^3} f(\xi)S_{\xi}(t)(g(\xi,\cdot))d\xi=\mathcal F^{-1}\left(\int_{\mathbb R^3} f(\xi) e^{-it(|\eta|^2-2\eta.\xi)}\hat g(\xi,\eta)d\xi\right).
$$
Formally, in addition to the free Schr\"odinger evolution, the information along $dW(\xi)$ has been displaced at speed $\xi$. The consequence of this spreading effect is the appearance of a non-stationary phase in $\xi$ variable in the last identity, which gives an additional $1/\sqrt t$ decay at a cost of $1/\sqrt{|\eta|}$.

\end{remark}

\begin{proof}

\textbf{Step 1}: \emph{Proof of the first bound}. We employ a duality argument to establish first two continuity bounds simultaneously and consider the operators:
$$
T:Z_0 \mapsto \mathbb E\left(\bar Y (t) S(t)Z_0\right), \quad \quad T^*:U\mapsto \int_0^\infty S(-\tau)\left(Y(\tau)U(\tau)\right)d\tau .
$$
We claim that for any $s\in \mathbb R$ there holds:
\begin{equation} \label{bd:continuityTT*}
\| \mathbb E\left(\bar Y (t) S(t)Z_g\right)\|_{L^2_t\dot H^{s+\frac 12}_x}\leq C \| Z_g \|_{\dot{\tilde H}^{s}_{x,\omega}}  \quad \mbox{and}  \quad  \| \int_0^\infty S(-\tau)\left(Y(\tau)U(\tau)\right)d\tau \|_{\dot{\tilde H}^{s+\frac 12}_{x,\omega}}\leq C \| U \|_{L^2_t\dot H^s_x}.
\end{equation}
The above continuity bounds for $T$ from $\dot{\tilde H}_{x,\omega}^s$ into $L^2_t\dot H^{s+1/2}$ and for $T^*$ from $L^2_t\dot H^{s}$ into $\dot{\tilde H}_{x,\omega}^{s+1/2}$ for any $s\in \mathbb R$, are equivalent to that of $TT^*$ from $L^2_t\dot H^{s}$ into $L^2_t\dot H^{s+1}$. Using Lemma \ref{lem:expressionWVY} there holds:
\begin{equation} \label{id:continuityTT*}
T(Z_g)=\int_{\mathbb R^3} f(\xi)S_{\xi}(t)(g(\xi,\cdot))d\xi, \quad \quad T^*(U)= \int_0^\infty \int_{\mathbb R^3} dW(\xi) f(\xi)e^{i\xi.x}S_\xi(-\tau)(U(\tau))d\tau .
\end{equation}
Hence, using \fref{id:correlationgaussian} and getting rid of the weights in Fourier, the continuity of $TT^*$ from $L^2_t\dot H^{s}$ into $L^2_t\dot H^{s+1}$ is equivalent to the continuity from $L^2_tL^2_x$ into itself of the operator:
$$
\mathcal T:U\mapsto  \int_0^\infty \int_{\mathbb R^3}f^2(\xi) |\nabla| S_{\xi}(t-\tau)U(\tau)d\tau 
$$
In Fourier this is:
\begin{eqnarray*}
\mathcal F(\mathcal T U)(t,\eta) & = &  \int_0^\infty \int_{\mathbb R^3}f^2(\xi)e^{-i(t-\tau)(|\eta|^2+2\xi.\eta)}|\eta| \hat U(\tau,\eta)d\tau d\xi \\
&=& \int_0^\infty \int_{\mathbb R^3} \eta h(2\eta (t-\tau)) e^{-i(t-\tau) |\eta|^2} \hat U(\tau,\eta)d\tau .
\end{eqnarray*}
At each fixed $\eta$, we recognise a convolution in time, so from Parseval's and Young's inequality:
\be \label{bd:lineaireconstante}
\| \mathcal T U  \|_{L^2_{t,x}}\leq \| |\eta| h(2\eta \cdot)\|_{L^{\infty}_\eta L^1_\tau}\| \hat U \|_{L^2_\xi L^2_\tau}\lesssim \sup_{\eta \in \mathbb S^3}\int |h|(\eta r)dr \| \hat U \|_{L^2_{x,t}}
\ee
where we changed variables $\tau \mapsto \tilde \tau /(2|\eta|)$ in the integral on $h$. So $\mathcal T $ is indeed continuous from $L^2_tL^2_x$ into itself, which establishes \fref{bd:continuityTT*} and the first bound the Lemma.\\

\textbf{Step 2}: \emph{Proof of the second bound}. One has that $\nabla Y=\int \xi f(\xi)e^{i\xi.x-it(|\xi|^2+m)}dW(\xi)$. Since $h$ is the Fourier transform of $f^2$, then $-\pa_{x_jx_j}h$ is that of $\xi_j^2f^2$. Since $h$ is radial, writing $r=|x|$ we get $\pa_{x_jx_j}h=r^{-1}h'(r)+x_j^2r^{-2}(h''(r)-r^{-1}h'(r))$. Hence we bound:
$$
\sup_{j=1,2,3} \sup_{x\in \mathbb S^3} \int_0^\infty |\pa_{x_jx_j}h(xr)|dr \lesssim \int_0^\infty \left(|h''|(r)+\frac{|h'|(r)}{r}\right)dr<\infty.
$$
Therefore, the very same analysis of Step 1, but this time replacing $f$ by $\xi_j f$ for $j=1,2,3$, using the explicit formula for the continuity constant \fref{bd:lineaireconstante} and the above bound, yields:
\be \label{bd:continuityTT*2}
\| \mathbb E\left([\nabla \bar Y (t)] S(t)Z_g\right)\|_{L^2_t\dot H^{s+\frac 12}_x}\lesssim \left(\int_0^\infty \left(|h''|(r)+\frac{|h'|(r)}{r}\right)dr\right)^{\frac 12}  \| Z_g \|_{\dot{\tilde H}^{s}_{x,\omega}}.
\ee
and that, for any $2\leq p,q\leq \infty $ and $0\leq s <3/2$ such that $\frac2{p}+\frac3{q} = \frac 32-s$ there holds:
\be \label{bd:continuityTT*3}
 \| \int_\R S(t-\tau)\left([\nabla Y(\tau)]U(\tau)\right)d\tau \|_{L^p_t L^q_xL^2_\omega\cap C_t( \dot{\tilde H}^s_{x,\omega})}\lesssim \left(\int_0^\infty \left(|h''|(r)+\frac{|h'|(r)}{r}\right)dr\right)^{\frac 12}  \| U \|_{L^2_t\dot H^{s-\frac 12}_x}.
\ee
Next, we have for any $0\leq s < 3/2$ that $\dot{\tilde H}^s_{x,\omega}$ is continuously embedded in $L^{2^*}_xL^2_\omega$ where $2^*$ is the Lebesgue exponent of the standard Sobolev embedding $\dot H^s_x\mapsto L^{2^*}_x$. Indeed, by \fref{def:Zg}, \fref{id:correlationgaussian}, Minkowski inequality and Sobolev embedding:
$$
\| Z_g \|_{L^{2^*}_xL^2_\omega} =\left(\int_x \left(\int_\xi |g(x,\xi)|^2d\xi \right)^{\frac{2^*}{2}}\right)^{\frac{1}{2^*}}\lesssim \|\;\| g(\xi,x)\|_{L^{2^*}_x}\|_{L^2_\xi}\lesssim \|\; \| g (\xi,x)\|_{\dot H^s_x}\|_{L^2_\xi}=\| Z_g\|_{\tilde{\dot{H}}^s_{x,\omega}}.
$$
We then take $s=-1/2$. The dual version of the above embedding is:
\be
L^{3/2}_xL^2_\omega \quad \mbox{embeds continuously into} \quad \dot{\tilde H}^{-1/2}_{x,\omega}.
\ee
Applying \fref{bd:continuityTT*} and \fref{bd:continuityTT*2} then gives the second bound \fref{bd:improvedrandomtopotential} of the Proposition.\\

\textbf{Step 3}: \emph{Proof of the third bound}. It follows as a direct consequence of the second bound in \fref{bd:continuityTT*} and of the dispersive estimates of Lemma \ref{lem:strichartzZg} for $Z_g$ fields. We have
$$
\| \int S(t-\tau)\left(Y(\tau)U(\tau)\right)d\tau \|_{L^p_t L^q_xL^2_\omega}\lesssim \| \int S(-\tau)\left(Y(\tau)U(\tau)\right)d\tau \|_{\dot{\tilde H}^{s}}\lesssim C \| U \|_{L^2_t\dot H^{s-\frac 12}_x}.
$$

\end{proof}

A direct application of the above Proposition, useful for our setting for nonlinear stability, is the following.

\begin{proposition}\label{lem:estonWUY} For all $U\in L^2_{t}H^{1/2}_x$, we have if $\int_0^\infty \left(||h|+\frac{|h'|(r)}{r}+|h''|(r)\right)dr<\infty$:
\be \label{bd:WVY}
\|W_U(Y)\|_{L^{\infty}_tL^3_xL^2_\omega} +\| \langle \nabla \rangle^{\frac 12} W_U(Y)\|_{L^5_{t,x}L^2_\omega}\lesssim \|U\|_{L^2_{t}H^{1/2}_x}.
\ee
Under the same hypothesis, if $U\in L^2_{t}L^2_x$:
\be \label{bd:WVnablaY}
\|W_U(Y)\|_{L^5_{t,x}L^2_\omega}+\|W_U(\nabla Y)\|_{L^5_{t,x}L^2_\omega}\lesssim \|U\|_{L^2_{t,x}}.
\ee
\end{proposition}

\begin{proof}[Proof of Proposition \ref{lem:estonWUY}]

Because of the density of $S(\R\times \R^3)$ in $L^2_{t,x}$, and by Christ-Kiselev lemma, it is sufficient to prove that for all $U \in S(\R\times \R^3)$,
$$
\| W_{U}^{\R}(Y)\|_{L^{\infty}_tL^3_xL^2_\omega}+\| \langle \nabla \rangle^{\frac 12} W_{U}^{\R}(Y)\|_{L^5 L^5 L^2_\omega} \lesssim \|U\|_{L^2_{t,x}}.
$$
 We get from \fref{bd:potentialtorandom} with $s=1/2$ the first bound of the Proposition:
\be \label{bd:WA0inter}
\|W^{\R}_U(Y)\|_{L^5_tL^5_xL^2_\omega\cap L^{\infty}_tL^3_xL^2_\omega}\lesssim \| U\|_{L^2_tL^2_x}
\ee
This shows the first bound in \fref{bd:WVnablaY}. Next, we differentiate using Lemma \ref{lem:expressionWVY} for any $j=1,2,3$:
\be \label{id:differentiationWUY}
 \partial_{x_j} W^{\R}_U(Y)= -i\partial_{x_j}\left( \int S(t-\tau)(U(\tau)Y(\tau))d\tau\right)=W^{\R}_{U}(\partial_{x_j}Y)+W^{\R}_{\partial_{x_j}U}(Y).
\ee
We bound the first term in \fref{id:differentiationWUY} via \fref{bd:continuityTT*3}, and obtain: $\| W^{\R}_U(\partial_{x_j}Y)\|_{L^5_{t,x}L^2_\omega} \lesssim \| U \|_{L^2_{t,x}}$. This shows the second bound in \fref{bd:WVnablaY}, hence \fref{bd:WVnablaY} is proved. We bound the second term in \fref{id:differentiationWUY} via \fref{bd:WA0inter}, and get:
$$
\| \nabla W^{\R}_U(Y)\|_{L^5_tL^5_xL^2_\omega}\lesssim \| U\|_{L^2_{t,x}}+ \| \nabla U\|_{L^2_{t,x}}.
$$
Hence \fref{bd:WA0inter} and the above bound, together with \fref{id:vectorvaluedbesselpotential}, show that $\| \langle \nabla \rangle W^{\R}_U[f,w]\|_{L^5_{t,x}L^2_\omega}\lesssim \| \langle \nabla \rangle U\|_{L^2_{t,x}}$. Applying complex interpolation between \fref{bd:WA0inter} and this  bound gives the second estimate in \fref{bd:WVY}.

\end{proof}

\subsection{Linear response of the equilibrium to a potential}

We now study the operator $L_2$ defined in \fref{def:L2}, which ends the linear analysis and allows us to prove Proposition \ref{pr:lineaire} below. The operator $L_2$ is the linear response of the equilibrium $Y$ to a potential $V$ \cite{GV,L,Mih}. Its property was previously studied in \cite{lewsab2}, and extended to the current setup in \cite{CodS}. In particular, it is proved in \cite{lewsab2} Proposition 1, \cite{CodS} Lemma 5.6, that $L_2$ is the following Fourier multiplier (for space-time Fourier transformation):
$$
\mathcal F_{t,x}\left(L_2(V)\right)(\omega,\xi) = \hat w (\xi)m_f(\omega,\xi) \mathcal F_{t,x} V(\omega,\xi),
$$
where, only here, $\omega$ denotes the dual time variable and not an element of the probability space $\Omega$, and:
\be \label{def:mf}
m_f(\omega,\xi)= -2\mathcal F_t \left(\sin(|\xi|^2t) h(2\xi t)1_{t\geq 0} \right)(\omega)=-2 \int_0^{+\infty} e^{-i\omega t} \sin (|\xi|^2t)h(2\xi t) dt .
\ee

\begin{proposition}[Lewin-Sabin Proposition 1 and Corollary 1 \cite{lewsab2}] \lab{pr:bdL2}

Assume that $f$ satisfies the following conditions:

\begin{itemize} 
\item $f>0$ is a bounded radial $C^1$ fonction, with, writing $r= |\xi|$ the radial variable, $\partial_r f <0$,
\item $\int_0^\infty r^2f^2(r)dr<\infty$ and $\int_0^\infty r|f(r)\partial_r  f(r)|dr < \infty$,
\item $\int_0^\infty (1+r)|h|(r)dr<\infty$ and $\int_0^\infty \left(\frac{|h'|(r)}{r}+|h''|(r)\right)dr<\infty$,
\end{itemize}
and that $w$ satisfies \fref{id:hpw} and
\be \label{def:eph}
\| (\hat w)_-\|_{L^{\infty}}\left(\int_{0}^{\infty} r|h(r)|dr\right)<2 \quad \mbox{and} \quad \hat w (0)_+\epsilon_h <1, \quad\mbox{where} \quad \epsilon_h= \limsup_{(\tau,\xi)\rightarrow (0,0)} \left( \Re m_f(\tau,\xi)\right)
\ee
Then $m_f(\omega,\xi)$ is a bounded function, continuous outside $(0,0)$, and $L_2$ is continuous on $L^2_tL^2_x$. Moreover, $\text{Id}-L_2$ has a continuous inverse on $L^2_{t,x}$. In addition, one has that $(\text{Id}-L_2)^{-1}$ is continuous from $L^2_{t,x}\cap L^{5/2}_{t,x}$ onto $L^{5/2}_{t,x}$.

\end{proposition}

\begin{remark} \label{rem:cancellationlineaire}

Proposition \ref{pr:freeandpotential} shows a loss of half a derivative at low frequencies, whereas the above Proposition \ref{pr:bdL2} shows no loss. The reason is a cancellation explained in details in the paragraph \emph{Linear cancellation at low frequencies} of the Strategy of the proof Subsection \ref{setup}.

\end{remark}

\begin{proof}[Proof of Proposition \ref{pr:bdL2}]

All results are contained in \cite{lewsab2} Proposition 1 and Corollary 1, except the fact that $(\text{Id}-L_2)^{-1}$ is continuous from $L^2_{t,x}\cap L^{5/2}_{t,x}$ onto $L^{5/2}_{t,x}$. Thus, we solely prove this fact, avoiding advanced harmonic analysis tools as in \cite{lewsab2}, but relying on dispersive estimates. We will use that all other results are known to be true, in particular that $L_2$ is continuous on $L^2_{t,x}$, and that $\text{Id}-L_2$ is invertible on $L^2_{t,x}$ or equivalently that $c\leq |1-\hat w(\xi)m_f(\tau,\xi)|\leq c^{-1}$ for some $c>0$.

First, by using Cauchy-Schwarz , H\"older and \fref{bd:WVnablaY}, we bound:
$$
\| \mathbb E (\bar Y W_V (Y)) \|_{L^5_{t,x}}\leq \|  W_V (Y)) \|_{L^5_{t,x}L^2_\omega} \| Y \|_{L^{\infty}_{t,x}L^2_\omega}\lesssim \|  W_V (Y)) \|_{L^5_{t,x}L^2_\omega} \lesssim \| V \|_{L^2_{t,x}}
$$
so that $L_2$ is continuous from $L^2_{t,x}$ onto $L^{5}_{t,x}$. As it is also continuous on $L^2_{t,x}$, by interpolation we get that $L_2$ is continuous from $L^2_{t,x}$ onto $L^{5/2}_{t,x}$. We then decompose:
$$
(\text{Id}-L_2)^{-1}=\text{Id}+L_2(\text{Id}-L_2)^{-1}.
$$
Above, the first term $\text{Id}$ is continuous on $L^{5/2}_{t,x}$. The second $L_2(\text{Id}-L_2)^{-1}$ is continuous from $L^2_{t,x}$ onto $L^{5/2}_{t,x}$, because $(\text{Id}-L_2)^{-1}$ is continuous on $L^2_{t,x}$, and $L_2$ is continuous from $L^2_{t,x}$ onto $L^{5/2}_{t,x}$. Hence $(\text{Id}-L_2)^{-1}$ is continuous from $L^2_{t,x}\cap L^{5/2}_{t,x}$ onto $L^{5/2}_{t,x}$.

\end{proof}

We can now prove Proposition \ref{pr:lineaire}.

\begin{proof}[Proof of Proposition \ref{pr:lineaire}]

From the fixed point equation \fref{def:L2} we take 
$$
V=(1-L_2)^{-1}(2Re\mathbb E(\bar Y S(t)Z_0)).
$$
From \fref{bd:strichartzrandomtopotential} we get $\|\mathbb E(\bar Y S(t)Z_0)\|_{L^{5}_{t,x}}\lesssim \| Z_0 \|_{H^{1/2}_x}\lesssim \| Z_0 \|_{\Theta_0}$. From \fref{bd:improvedrandomtopotential2} and \fref{bd:strichartzrandomtopotential} one obtains:
$$
\|\mathbb E(\bar Y S(t)Z_0)\|_{L^2_tH^{1/2}_x}\lesssim \| Z_0 \|_{L^2_{x,\omega}}+\| Z_0 \|_{L^{3/2}_xL^2_\omega}\lesssim \| Z_0 \|_{\Theta_0}.
$$
The last two estimates and H\"older show $\|\mathbb E(\bar Y S(t)Z_0)\|_{L^{5/2}_{t,x}\cap L^2_tH^{1/2}_x}\lesssim \| Z_0 \|_{\Theta_0}$. Proposition  \ref{pr:bdL2} ensures that $(1-L_2)^{-1}$ is continuous on $L^{2}_{t,x}$, hence it is continuous on $L^2_{t}H^{1/2}_x$ since this is a space-time Fourier multiplier. Moreover, it is also continuous from $L^{2}_{t,x}\cap L^{5/2}_{t,x}$ into $L^{5/2}_{t,x}$. Therefore:
$$
\| V\|_{\Theta_V}=\| (1-L_2)^{-1}(\mathbb E(\bar Y S(t)Z_0))\|_{L^{5/2}_{t,x}\cap L^2_{t}H^{1/2}_x}\lesssim \|\mathbb E(\bar Y S(t)Z_0))\|_{L^{5/2}_{t,x}\cap L^2_{t}H^{1/2}_x} \lesssim \| Z_0 \|_{\Theta_0}.
$$
Then, we set $Z=S(t)Z_0+W_V(Y)$. For the first part there holds $\| S(t)Z_0\|_{\Theta_Z}\lesssim \| Z_0 \|_{\Theta_0}$ from Lemma \ref{lem:strichartzrandom}, while the last bound in Proposition \ref{pr:lineaire} is a direct consequence of Proposition \ref{lem:estonWUY}.

\end{proof}

\section{Study of a specific quadratic term}\label{sec:spefquad}

This Section is devoted to the study of the quadratic term $Q_2$. We recall a formal explanation is included in the strategy of the proof Subsection \ref{setup} and we claim:

\begin{proposition} \label{pr:continuiteQ2} There exists $C(h,w)$ such that for all $U,V \in \Theta_V$, we have 
$$
\|Q_2(V,U)\|_{\Theta_V} \leq C(h,w) \|V\|_{\Theta_V} \|U\|_{\Theta_V}.
$$
\end{proposition}

Recall that $\Theta_V= L^{5/2}_{t,x}\cap L^2_{t}H^{1/2}_x$. We first establish a bilinear estimate for $Q_2$ for the low space-time Lebesgue exponent $2$, inspired by \cite{lewsab2}. This requires explicit computations, analogue to those performed for a similar term in \cite{lewsab2}, that are performed in the next Lemmas. We prove such formulas for real valued potentials $U$ and $V$ belonging to the Schwartz class to avoid technical issues. The general estimate in $(L^2_{t}H^{1/2}_x)^2$ will then follow from a density argument. For estimating in $L^2_t\dot H^{1/2}$ or $L^{5/2}_{t,x}$ we use the linear bounds of the previous Section. By H\"older's inequality, $\E(\overline{W_U(Y)} W_V(Y))$ belongs to $L^{5/2}_{t,x}$, since $W_V(Y)$ belongs to $L^5_{t,x} L^2_\omega$. But it is far from obvious for $\E(\bar Y W_U(W_V(Y)))$. Indeed, what we do is that we prove that $Q_2$ belongs to $L^2_{t,x}\cap L^{10/3}_{t,x}$ by a duality argument. Since $\frac52 \in [2,\frac{10}{3}]$, we will get the desired result. The next Lemma gives an exact expression of the first term in $Q_2$.

\begin{lemma} \label{lem:exprQ21} For $U,V \in \mathcal \mathcal S(\R\times \R^3)$, set $J_1 = 2\textrm{Re }\E\Big( \overline{W_U(Y)} W_V(Y)\Big)$. Then we have for all $\eta \in \R^3$:
\begin{multline*}
\hat J_1(\eta) = 2\int_{0}^td\tau_1 \int_{0}^{\tau_1}d\tau_2 \int_{\mathbb R^3} d\tilde{\eta} h((t-\tau_1)2\eta + (\tau_1-\tau_2) 2\tilde{\eta})\cos\Big((t-\tau_1)(|\eta|^2-2\tilde{\eta}\cdot \eta)-(\tau_1-\tau_2)|\tilde{\eta}|^2 \Big)\\
\Big[\hat U(\eta-\tilde{\eta},\tau_1) \hat V(\tilde{\eta},\tau_2)+ \hat V(\eta-\tilde{\eta},\tau_1) \hat U(\tilde{\eta},\tau_2)\Big].
\end{multline*}
\end{lemma}

\begin{proof}We have by Lemma \ref{lem:expressionWVY} and \fref{id:correlationgaussian}:
$$
\E\Big( \overline{W_U(Y)} W_V(Y)\Big) = \int_{\mathbb R^3} d\xi |f(\xi)|^2 \int_0^t d\tau_1 \int_0^{t} d\tau_2 \overline{S_\xi(t-\tau_1) U(\tau_1)} S_\xi (t-\tau_2) V(\tau_2).
$$
Let $J_{U,V}$ be the Fourier transform of $\overline{S_\xi(t-\tau_1) U(\tau_1)} S_\xi (t-\tau_2) V(\tau_2)$. We have since $U$ is real-valued,
$$
J_{U,V}(\eta)= \int_{\mathbb R^3} d\tilde{\eta} e^{i(t-\tau_1)((\eta-\tilde{\eta})^2-2\xi \cdot (\eta-\tilde{\eta}) }\hat U(\eta-\tilde{\eta},\tau_1)e^{-i(t-\tau_2)(|\tilde{\eta}|^2 +2\xi \cdot \tilde{\eta}) } \hat V(\tilde{\eta},\tau_2).
$$
We have the identity
\begin{multline*}
(t-\tau_1)(|\eta-\tilde{\eta}|^2-2\xi \cdot (\eta-\tilde{\eta}))  -(t-\tau_2)(|\tilde{\eta}|^2 +2\xi \cdot \tilde{\eta}) = \\
(t-\tau_1)(|\eta|^2-2\tilde{\eta}\cdot \eta) -(\tau_1-\tau_2)|\tilde{\eta}|^2-\xi\cdot \left( (t-\tau_1)2\eta+(\tau_1-\tau_2)2\tilde \eta\right).
\end{multline*}
Hence, integrating over $\xi$, we get, with $h$ the Fourier transform of $|f|^2$, 
$$
\int_{\mathbb R^3} |f(\xi)|^2 J(\eta) d\xi =  \int_{\mathbb R^3} d\tilde{\eta} h((t-\tau_1)2\eta + (\tau_1-\tau_2) 2\tilde{\eta}) e^{i(t-\tau_1)(|\eta|^2-2\tilde{\eta}\cdot \eta)-(\tau_1-\tau_2)|\tilde{\eta}|^2 }\hat U(\eta-\tilde{\eta},\tau_1) \hat V(\tilde{\eta},\tau_2).
$$
Therefore we obtain the identity:
\begin{align}
\nonumber &\mathcal F \left(\E\Big( \overline{W_U(Y)} W_V(Y)\Big)(\eta)\right)\\
\label{id:J1part1}=& \int_{[0,t]^2} d\tau_1d\tau_2\int_{\mathbb R^3} d\tilde{\eta} h((t-\tau_1)2\eta + (\tau_1-\tau_2) 2\tilde{\eta}) e^{i(t-\tau_1)(|\eta|^2-2\tilde{\eta}\cdot \eta)-(\tau_1-\tau_2)|\tilde{\eta}|^2 }\hat U(\eta-\tilde{\eta},\tau_1) \hat V(\tilde{\eta},\tau_2).
\end{align}
By symmetry the complex conjugate of the above is:
\begin{align*}
&\mathcal F \left(\E\Big( \overline{W_V(Y)} W_U(Y)\Big)(\eta)\right)\\
=& \int_{[0,t]^2} d\tau_1d\tau_2\int_{\mathbb R^3} d\tilde{\eta} h((t-\tau_1)2\eta + (\tau_1-\tau_2) 2\tilde{\eta}) e^{i(t-\tau_1)(|\eta|^2-2\tilde{\eta}\cdot \eta)-(\tau_1-\tau_2)|\tilde{\eta}|^2 }\hat V(\eta-\tilde{\eta},\tau_1) \hat U(\tilde{\eta},\tau_2)\\
\end{align*}
We change variables in the above identity, with $(\tilde \eta,\tau_1,\tau_2) \rightarrow (\eta-\tilde \eta',\tau_2',\tau_1')$. Noticing that
\be \label{id:changeofvarinter1}
(t-\tau_1)2\eta + (\tau_1-\tau_2) 2\tilde{\eta}=(t-\tau_2')2\eta + (\tau_2'-\tau_1') 2(\eta-\tilde{\eta}')=(t-\tau_1')2\eta + (\tau_1'-\tau_2') 2\tilde{\eta}'
\ee
and that
\bea
\nonumber (t-\tau_1)(|\eta|^2-2\tilde{\eta}\cdot \eta)-(\tau_1-\tau_2)|\tilde{\eta}|^2&=&(t-\tau_2')(|\eta|^2-2(\eta-\tilde{\eta}')\cdot \eta)-(\tau_2'-\tau_1')|\eta-\tilde{\eta}'|^2\\
\label{id:changeofvarinter2}&=&-(t-\tau_1')(|\eta|^2-2\tilde{\eta}'\cdot \eta)+(\tau_1'-\tau_2')|\tilde{\eta}'|^2,
\eea
we obtain (replacing $(\tilde \eta',\tau_1',\tau_2')$ by $(\tilde \eta,\tau_1,\tau_2)$):
\begin{align*}
&\mathcal F \left(\E\Big( \overline{W_V(Y)} W_U(Y)\Big)(\eta)\right)\\
=& \int_{[0,t]^2} d\tau_1d\tau_2\int_{\mathbb R^3} d\tilde{\eta} h((t-\tau_1)2\eta + (\tau_1-\tau_2) 2\tilde{\eta}) e^{-i(t-\tau_1)(|\eta|^2-2\tilde{\eta}\cdot \eta)+(\tau_1-\tau_2)|\tilde{\eta}|^2 }\hat U(\eta-\tilde{\eta},\tau_1) \hat V(\tilde{\eta},\tau_2).
\end{align*}
Using the identity $e^{i\theta}+e^{-i\theta}=2\cos (\theta)$, summing \fref{id:J1part1} and the above identity one obtains:
\bee
\hat J_1(\eta)&=&2 \int_{[0,t]^2} d\tau_1d\tau_2\int_{\mathbb R^3} d\tilde{\eta}h((t-\tau_1)2\eta + (\tau_1-\tau_2) 2\tilde{\eta}) \\
&&\quad \quad \quad \quad \quad \quad \quad \cos((t-\tau_1)(|\eta|^2-2\tilde{\eta}\cdot \eta)-(\tau_1-\tau_2)|\tilde{\eta}|^2 )\hat U(\eta-\tilde{\eta},\tau_1) \hat V(\tilde{\eta},\tau_2).
\eee
Using now that $J_1$ is symmetric in $U$ and $V$, we get
\begin{multline*}
\hat J_1(\eta) = \int_{[0,t]^2}d\tau_1 d\tau_2 \int_{\mathbb R^3} d\tilde{\eta} h((t-\tau_1)2\eta + (\tau_1-\tau_2) 2\tilde{\eta}) \cos\Big((t-\tau_1)(|\eta|^2-2\tilde{\eta}\cdot \eta)-(\tau_1-\tau_2)|\tilde{\eta}|^2 \Big)\\
\Big[\hat U(\eta-\tilde{\eta},\tau_1) \hat V(\tilde{\eta},\tau_2)+ \hat V(\eta-\tilde{\eta},\tau_1) \hat U(\tilde{\eta},\tau_2)\Big].
\end{multline*}
We now check that the integrand on $[0,t]^2$ above is symmetric in $\tau_1$ and $\tau_2$. Let
\begin{multline*}
g(\tau_1,\tau_2) := \int_{\mathbb R^3} d\tilde{\eta} h((t-\tau_1)2\eta + (\tau_1-\tau_2) 2\tilde{\eta}) \cos\Big((t-\tau_1)(|\eta|^2-2\tilde{\eta}\cdot \eta)-(\tau_1-\tau_2)|\tilde{\eta}|^2 \Big)\\
 \Big[\hat U(\eta-\tilde{\eta},\tau_1) \hat V(\tilde{\eta},\tau_2)+ \hat V(\eta-\tilde{\eta},\tau_1) \hat U(\tilde{\eta},\tau_2)\Big]
\end{multline*}
Changing variables $\tilde \eta \rightarrow \eta-\tilde \eta'$, we get
\begin{multline*}
g(\tau_1,\tau_2) = \int_{\mathbb R^3} d\tilde{\eta}' h((t-\tau_1)2\eta + (\tau_1-\tau_2) 2(\eta-\tilde{\eta}')) \cos\Big((t-\tau_1)(|\eta|^2-2(\eta-\tilde{\eta}')\cdot \eta)-(\tau_1-\tau_2)|\eta-\tilde{\eta}'|^2 \Big)\\
 \Big[\hat U(\tilde{\eta}',\tau_1) \hat V(\eta-\tilde{\eta}',\tau_2)+ \hat V(\tilde{\eta}',\tau_1) \hat U(\eta-\tilde{\eta}',\tau_2)\Big].
\end{multline*}
Then using \fref{id:changeofvarinter1} and \fref{id:changeofvarinter2} and the fact that $\cos (\theta)=\cos(-\theta)$:
\begin{multline*}
g(\tau_1,\tau_2) = \int_{\mathbb R^3} d\tilde{\eta}' h((t-\tau_2)2\eta + (\tau_2-\tau_1) 2\tilde{\eta}) \cos\Big(-(t-\tau_2)(|\eta|^2-2\tilde{\eta}'\cdot \eta)+(\tau_2-\tau_1)|\tilde{\eta}'|^2 \Big)\\
\Big[\hat U(\eta-\tilde{\eta}',\tau_2) \hat V(\tilde{\eta}',\tau_1)+ \hat V(\eta-\tilde{\eta}',\tau_2) \hat U(\tilde{\eta}',\tau_1)\Big]\\
=g(\tau_2,\tau_1).
\end{multline*}

Therefore, using this symmetry, we have that $\int_{[0,t]^2}d\tau_1d\tau_2 g(\tau_1,\tau_2)=2\int_0^td\tau_1\int_0^{\tau_1}d\tau_2 g(\tau_1,\tau_2)$ which give the desired identity for the Lemma:
\begin{multline*}
\hat J_1(\eta) = 2\int_{0}^td\tau_1 \int_{0}^{\tau_1}d\tau_2 \int_{\mathbb R^3} d\tilde{\eta} h((t-\tau_1)2\eta + (\tau_1-\tau_2) 2\tilde{\eta}) \cos\Big((t-\tau_1)(|\eta|^2-2\tilde{\eta}\cdot \eta)-(\tau_1-\tau_2)|\tilde{\eta}|^2 \Big)\\
\Big[\hat U(\eta-\tilde{\eta},\tau_1) \hat V(\tilde{\eta},\tau_2)+ \hat V(\eta-\tilde{\eta},\tau_1) \hat U(\eta_2,\tau_2)\Big].
\end{multline*}

\end{proof}

The next Lemma will help give an exact expression of the second term in $Q_2$.

\begin{lemma} \label{lem:expressionWVWU} We have for all $U,V \in \mathcal S(\R\times \R^3)$,
$$
\E\Big( \bar Y W_V(W_U(Y))\Big) = - \int_{\mathbb R^3} d\xi |f(\xi)|^2 \int_{0}^t d\tau_1 S_\xi (t-\tau_1)\Big[ V(\tau_1) \int_{0}^{\tau_1} d\tau_2 S_\xi (\tau_1 - \tau_2) U (\tau_2) \Big].
$$
\end{lemma}

\begin{proof}Recall that from Lemma \ref{lem:expressionWVY}:
$$
W_U(Y)(\tau_1) = -i\int_{\mathbb R^3} f(\xi) dW(\xi) e^{-i\tau_1(m+\xi^2) + ix\cdot \xi} \int_{0}^{\tau_1}d\tau_2 S_\xi(\tau_1 - \tau_2) U(\tau_2).
$$
Therefore, we have 
$$
W_V(W_U (Y)) = -  \int_{0}^t d\tau_1 S(t-\tau_1) \Big[ V(\tau_1) \int_{\mathbb R^3} f(\xi) dW(\xi) e^{-it(m+\xi^2) + ix\cdot \xi} \int_{0}^{\tau_1}d\tau_2 S_\xi(\tau_1 - \tau_2) U(\tau_2)\Big].
$$
Using Lemma \ref{lem-Fouriertrick} and \fref{id:fubinischrodingerwiener}, we get
$$
W_V(W_U(Y)) = - \int_{\mathbb R^3} f(\xi) dW(\xi) e^{-it(m+\xi^2) + ix\cdot \xi} \int_{0}^t d\tau_1 S_\xi(t-\tau_1) \Big[V(\tau_1) \int_{0}^{\tau_1} d\tau_2  S_\xi(\tau_1 - \tau_2) U(\tau_2)\Big].
$$
Using that $Y = \int_{\mathbb R^3} f(\xi)  e^{-it(m+\xi^2) + ix\cdot \xi}  d W(\xi)$ and \fref{id:correlationgaussian}, we get the desired result. 

\end{proof}

The next Lemma now gives an exact expression of the second term in $Q_2$.

\begin{lemma} \label{lem:exprQ22} For all $U,V\in \mathcal S(\R\times \R^3)$, set $J_2$ be the Fourier transform of $2\textrm{Re }\E(\bar Y W_V(W_U(Y)))$, we have for all $\eta \in \R^3$,
\begin{multline*}
J_2(\eta) = -2 \int_{0}^t d\tau_1 \int_{0}^{\tau_1} d\tau_2 \int_{\mathbb R^3} d\tilde{\eta} h((t-\tau_1)2\eta + (\tau_1 - \tau_2) 2\tilde{\eta} ) \cos \Big( (t-\tau_1) |\eta|^2 + (\tau_1 - \tau_2) |\tilde{\eta}|^2 \Big) \\
\hat V(\eta - \tilde{\eta}, \tau_1) \hat U(\tilde{\eta},\tau_2).
\end{multline*}
\end{lemma}

\begin{proof} Let $J$ be the Fourier transform of  $\E(\bar Y W_V(W_U(Y)))$, we have by Lemma \ref{lem:expressionWVWU}:
\begin{multline*}
J(\eta) = - \int_{\mathbb R^3} d\xi |f(\xi)|^2 \int_{0}^t d\tau_1 e^{-i (t-\tau_1)(|\eta|^2 + 2\xi \cdot \eta)}\int_{\mathbb R^3} d\tilde{\eta} \hat V(\eta-\tilde{\eta},\tau_1) \\
\int_{0}^{\tau_1} d\tau_2 e^{-i (\tau_1 - \tau_2)(|\tilde{\eta}|^2 + 2\xi\cdot \eta)} \hat U (\tilde{\eta},\tau_2) .
\end{multline*}
Integrating in $\xi$, we get
\begin{multline*}
J(\eta) = -\int_{0}^t d\tau_1 \int_{0}^{\tau_1} d\tau_2 \int_{\mathbb R^3} d\tilde{\eta} h( (t-\tau_1)2\eta + (\tau_1 - \tau_2) 2\tilde{\eta} ) e^{-i( t-\tau_1) |\eta|^2 -i (\tau_1 - \tau_2)|\tilde{\eta}|^2}\\
 \hat V(\eta-\tilde{\eta},\tau_1) \hat U (\tilde{\eta},\tau_2) .
\end{multline*}
Using that $J_2(\eta) = J(\eta) + \bar J(-\eta)$, and that $U, V$ and $|f|^2$ are real-valued, we get
\begin{multline*}
J_2(\eta) = -2  \int_{0}^t d\tau_1 \int_{0}^{\tau_1} d\tau_2 \int d\tilde{\eta} h( (t-\tau_1)2\eta + (\tau_1 - \tau_2) 2\tilde{\eta} ) \cos\Big(( t-\tau_1) |\eta|^2 + (\tau_1 - \tau_2)|\tilde{\eta}|^2\Big)\\ \hat V(\eta-\tilde{\eta},\tau_1) \hat U (\tilde{\eta},\tau_2) .
\end{multline*}
\end{proof}

The explicit computation of $Q_2$ is then the following.

\begin{proposition} For all $U,V \in \mathcal S(\R\times \R^3)$ and all $\eta \in \R^3$, 
\bea
\nonumber \hat Q_2(U,V) (\eta)&=&  4 \int_{0}^t d\tau_1 \int_{0}^{\tau_1} d\tau_2 \int d\eta_2 h( (t-\tau_1)2\eta + (\tau_1 - \tau_2) 2\eta_2 ) \\
\nonumber && \quad \quad \quad \sin ((t-\tau_1)(\eta^2-\eta_2\cdot \eta)) \sin((t-\tau_1)\eta_2\cdot \eta + (\tau_1 - \tau_2)\eta_2^2) \\
&& \quad \quad \quad \quad \quad \quad \label{id:expressionQ2} \Big[ \hat V(\eta-\eta_2,\tau_1) \hat U (\eta_2,\tau_2) +  \hat U(\eta-\eta_2,\tau_1) \hat V (\eta_2,\tau_2) \Big].
\eea
\end{proposition}

\begin{proof} Recall \fref{id:defQ2}. We sum the expression found in Lemmas \ref{lem:exprQ21} and \ref{lem:exprQ22}:
\begin{multline*}
\hat Q_2(U,V) (\eta)=  2 \int_{0}^t d\tau_1 \int_{0}^{\tau_1} d\tau_2 \int d\eta_2 h( (t-\tau_1)2\eta + (\tau_1 - \tau_2) 2\eta_2 ) \\
\Big[ \cos\Big((t-\tau_1)(\eta^2-2\eta_2\cdot \eta)-(\tau_1-\tau_2)\eta_2^2 \Big) - \cos\Big(( t-\tau_1) \eta^2 + (\tau_1 - \tau_2)\eta_2^2\Big)\Big] \\
\Big[ \hat V(\eta-\eta_2,\tau_1) \hat U (\eta_2,\tau_2) +  \hat U(\eta-\eta_2,\tau_1) \hat V (\eta_2,\tau_2) \Big].
\end{multline*}
Then we just performed a trigonometric transformation: using
$$
\cos(\theta - \phi) - \cos(\theta + \phi) = 2\sin(\theta) \sin (\phi)
$$
with 
$$
\theta = (t-\tau_1)(\eta^2-\eta_2\cdot \eta) \textrm{ and }
\phi = (t-\tau_1)\eta_2\cdot \eta + (\tau_1 - \tau_2)\eta_2^2,
$$
we get the desired identity \fref{id:expressionQ2}.

\end{proof}

With the explicit expression \fref{id:expressionQ2} at hand for $Q_2$, we can now look in detail to bound this term from $(L^2_{t,x})^2$ into $L^2_{t,x}$ and $L^{\infty}_tL^2_x$. The kernel appearing in the expression \fref{id:expressionQ2} of $Q_2$ satisfies the following bound.

\begin{lemma}\label{lem-normK} Setting 
\be \label{id:defK}
K_{\eta,\eta_2}(t,s) = h\left( 2t\eta + s 2\eta_2 \right) 
\sin \left(t(|\eta|^2-\eta_2\cdot \eta)\right) \sin\left(t\eta_2\cdot \eta + s|\eta_2|^2\right) ,
\ee
we have for all $\eta,\eta_2\in \R^3$ not collinear, and $p=1,2$:
$$
\|K_{\eta,\eta_2}\|_{L^2_t,L^p_s}^2 \lesssim (|\eta|^2|\eta_2|^2 - (\eta\cdot\eta_2)^2)^{-\frac 12}C_p(h), \quad \quad C_p(h)=\int_{\mathbb R} dv \left( \int_{\mathbb R} du |u|^{\frac1p-\frac 12}|h|^p\left(\sqrt{u^2+v^2}\right) \right)^{\frac 2p}.
$$
Moreover, $C_1(h)$ and $C_2(h)$ are finite under the assumptions of the main Theorem \ref{th:main}
\end{lemma}

\begin{proof} We first perform a geometric change variables. Projecting $\eta$ onto the directions $\eta_2$ and $\eta_2^\perp$ gives $\eta=\frac{\eta.\eta_2}{|\eta_2|^2}\eta_2+\left(\eta-\frac{\eta_2.\eta}{|\eta_2|^2}\eta_2\right)$. Hence the following orthogonal decomposition for the argument of $h$ in \fref{id:defK} and the associated Pythagorean identity:
$$
t\eta + s \eta_2=\left(t\frac{\eta.\eta_2}{|\eta_2|^2}+s\right)\eta_2+t\left(\eta-\frac{\eta_2.\eta}{|\eta_2|^2}\eta_2\right), \quad \quad |t\eta + s \eta_2|^2=  u^2(t,s)+v(t)^2
$$
where
$$
u (t,s) = |\eta_2|s + t \frac{\eta\cdot \eta_2}{|\eta_2|} \textrm{ and }v(t) = t \frac{\sqrt{|\eta|^2|\eta_2|^2- (\eta\cdot\eta_2)^2}}{|\eta_2|}.
$$
The arguments of the sine in \fref{id:defK} are $t(|\eta|^2-\eta_2\cdot \eta)=v(t)(|\eta|^2|\eta_2|^2- (\eta\cdot\eta_2)^2)^{-1/2}|\eta_2|(|\eta|^2-\eta_2\cdot \eta)$ and $t\eta_2\cdot \eta + s|\eta_2|^2=u(t,s)|\eta_2|$. Therefore, recalling that $h$ is radial, we have
\begin{multline*}
|K_{\eta_,\eta_2}(t,s)| \\
=|h\left(\sqrt{u^2(t,s)+v(t)^2}\right) \sin\left(v(t)(|\eta|^2|\eta_2|^2- (\eta\cdot\eta_2)^2)^{-1/2}|\eta_2|(|\eta|^2-\eta_2\cdot \eta)\right) \sin\left(u(t,s)|\eta_2|\right)|\\
\leq  |h\left(\sqrt{u^2(t,s)+v(t)^2}\right)| |\sin\left(u(t,s)|\eta_2|\right)|,
\end{multline*}
using $|\sin|\leq 1$. Doing the change of variable $s \rightarrow u(t,s)$ at fixed $t$, we thus have
$$
\|K(t, \cdot)\|_{L^p_s}^p = \int_{\mathbb R} du |\eta_2|^{-1} \left|h\left(\sqrt{u^2+ v(t)^2}\right)\right|^p \left|\sin\left( |\eta_2| u\right)\right|^p.
$$
For the second part of the norm, we do the change of variable $t\rightarrow v(t)$ and get
$$
\|K\|_{L^2_t,L^p_s}^2  \leq \int_{\mathbb R^3} \frac{|\eta_2|dv}{\sqrt{|\eta|^2|\eta_2|^2- (\eta\cdot\eta_2)^2}} \left| \int_{\mathbb R^3} du |\eta_2|^{-1}\left|h\left(\sqrt{u^2+v^2}\right)\right|^p \left|\sin (|\eta_2|u)\right|^p \right|^{\frac 2p}.
$$
We finally use the inequality $|\sin (|\eta_2|u)| \leq (|\eta_2| \;|u|)^{1/p-1/2}$ for $p=1,2$ to get the desired bound:
$$
\|K_{\eta,\eta_2}\|_{L^2_t,L^p_s}^2 \lesssim (|\eta|^2|\eta_2|^2 - (\eta\cdot\eta_2)^2)^{-\frac 12}\int_{\mathbb R} dv \left( \int_{\mathbb R} du |u|^{\frac 1p-\frac 12}\left|h\right|^p\left(\sqrt{u^2+v^2}\right) \right)^{\frac 2p}.
$$
We finally bound the integral in $C_p(h)$ using the assumptions of Theorem \ref{th:main}. First, for $p=2$:
$$
C_2(h)=C \int_0^\infty r h^2\lesssim \| h \|_{L^{\infty}}\int_0^\infty r|h|(r)dr \lesssim \| f\|_{L^2(\mathbb R^3)}^2\int_0^\infty r|h|(r)dr,
$$
using that $h$ is the Fourier transform of $f$. Similarly, for $p=1$, developing the square, changing variables $u\mapsto \sqrt{u^2+v^2}=w$, and using $\int_0^a (a^2-v^2)^{-1/2}dv=\pi/2$ for any $a>0$:
\bee
C_1(h)&=&\int_{\mathbb R^3} dvdudu' |u|^{\frac 12}|h|\left(\sqrt{u^2+v^2}\right)|u'|^{\frac 12}|h|\left(\sqrt{u^{'2}+v^2}\right)\\
&=& 8 \int_{[0,\infty)^3} dvdwdw' \frac{w}{\sqrt[4]{w^2-v^2}}|h|(w)\frac{w'}{\sqrt[4]{w^{'2}-v^2}}|h|(w'){\bf 1}(|v|\leq w){\bf 1}(|v|\leq w')\\
&\leq & 8 \int_{[0,\infty)^2} dwdw' w|h|(w)w'|h|(w') \int_0^{\min(w,w')}\frac{dv}{\sqrt{\min(w,w')^2-v^2}}\leq \frac{\pi}{4} \left(\int_0^\infty r|h|(r)dr\right)^2.
\eee

\end{proof}

\begin{proposition} \label{prop:QuadV2.2} There exists $C(h,w)$ such that for all $U,V \in L^2_{t}H^{1/2}_x$, we have 
\be \label{bd:Q2}
\|Q_2(V,U)\|_{L^2_{t,x}\cap L^{\infty}_tL^2_x} \leq C(h,w) \|V\|_{L^2_{t}H^{1/2}_x} \|U\|_{L^2_{t}H^{1/2}_x}.
\ee
\end{proposition}

\begin{proof} 

We first recall the following continuity estimates for a kernel. Set for $g,h\in L^2 (\R)$ and $K: \mathbb R^2 \rightarrow \mathbb R$, the function $u(t)= \int_{\R^2} g(s)h(s') K(t-s, s - s')  dsds'$, then there holds:
$$
\left\| u \right\|_{L^2(\mathbb R)} \leq \|K\|_{L^2(\R, L^1(\R))} \|g\|_{L^2}\|h\|_{L^2}, \quad \mbox{and} \quad \left\| u \right\|_{L^{\infty}(\mathbb R)} \leq \|K\|_{L^2(\R^2)} \|g\|_{L^2}\|h\|_{L^2}.
$$
The proof of the above is classical, so we safely omit it. Next, from the exact expression \fref{id:expressionQ2} and the definition of the kernel \fref{id:defK}, we obtain:
\begin{multline*}
\hat Q_2(U,V) (\eta) = \\
 2\int d\eta_2 \int_{0}^t dt_1 \int_{0}^{t_1} dt_2 K_{\eta,\eta_2}(t-t_1,t_1-t_2) \Big[\hat U(\eta-\eta_2,t_1) \hat V(\eta_2,t_2) +  \hat V(\eta-\eta_2,t_1) \hat U(\eta_2,t_2)\Big].
\end{multline*}
Therefore, by first using Minkowski inequality, and then combining the above continuity bound and Lemma \ref{lem-normK}:
\begin{multline*}
\|\hat Q_2(U,V) (\eta)\|_{L^2_t\cap L^{\infty}_t} \leq \\
2 \int_{\mathbb R^3} d\eta_2 \big\| \int_{0}^t dt_1 \int_{0}^{t_1} dt_2 K_{\eta,\eta_2}(t-t_1,t_1-t_2) \Big[\hat U(\eta-\eta_2,t_1) \hat V(\eta_2,t_2) +  \hat V(\eta-\eta_2,t_1) \hat U(\eta_2,t_2)\Big]\big\|_{L^2_t\cap L^{\infty}_t}\\
\leq  C_h \int_{\mathbb R^3} d\eta_2 (|\eta|^2|\eta_2|^2 - (\eta\cdot\eta_2)^2)^{-1/4} \Big[\|\hat U(\eta-\eta_2,\cdot)\|_{L^2_t}\| \hat V(\eta_2,\cdot)\|_{L^2_t} +  \|\hat V(\eta-\eta_2,\cdot)\|_{L^2_t}\| \hat U(\eta_2,\cdot)\|_{L^2_t} \Big]
\end{multline*}
where $C_h$ is a constant depending only on $h$ (and not $\eta,\eta_2$). We only treat the first term above, as the same reasoning applies for the second. In order to show $\|\hat Q_2(U,V) (\eta)\|_{L^2_t\cap L^{\infty}_t}\in L^2_\eta$ we reason by duality. By the above formula are left with estimating 
\be \label{id:defIinter} 
I := \int_{(\R^3)^2} d\eta d\xi (|\eta|^2|\xi|^2 - (\eta\cdot\xi)^2)^{-1/4} u(\eta-\xi) v(\xi) \phi(\eta) \langle \eta-\xi\rangle^{-\frac 12}\langle \eta \rangle^{-\frac 12}
\ee
with $u(\eta)=\langle \eta \rangle^{1/2}\|\hat U(\eta,\cdot)\|_{L^2_t}\in L^2(\mathbb R^3)$, $v(\eta)=\langle \eta \rangle^{1/2} \|\hat V(\eta,\cdot)\|_{L^2_t}\in L^2(\mathbb R^3)$ and $\phi \in L^2(\R^3)$ with $\phi\geq0$. We find a lower bound for the singular weight in \fref{id:defIinter} by writing in an orthonormal basis $\xi = (\xi_1,\xi_2,\xi_3)$ and $ \eta = (\eta_1,\eta_2,\eta_3)$:
$$
|\eta|^2|\xi|^2 - (\eta\cdot\xi)^2 = (\xi_1\eta_2 - \xi_2\eta_1)^2 + (\xi_1\eta_3 - \xi_3\eta_1)^2 +(\xi_3\eta_2 - \xi_2\eta_3)^2 \geq (\xi_1\eta_2 - \xi_2\eta_1)^2.
$$
Therefore, with $\eta' = (\eta_1,\eta_2)\in \mathbb R^2$ and $\xi' = (\xi_1,\xi_2)\in \mathbb R^2$, we have
\bee
I &\leq& \int_{(\R^2)^2} d\eta' d\xi' |\xi_1\eta_2 - \xi_2\eta_1|^{-1/2} \int_{\R^2} d\xi_3d\eta_3 u(\eta'-\xi',\eta_3-\xi_3) v(\xi',\xi_3) \phi(\eta',\eta_3) \langle \eta-\xi\rangle^{-\frac 12}\langle \eta \rangle^{-\frac 12}\\
&& \quad \leq  \int_{(\R^2)^2} d\eta' d\xi' |\xi_1\eta_2 - \xi_2\eta_1|^{-1/2} \| u(\eta'-\xi',\cdot)\|_{L^2_{\eta_3}} \|v(\xi',\cdot)\|_{L^2_{\eta_3}}\| \phi(\eta',\cdot)\|_{L^2_{\eta_3}}
\eee
where we used H\"older inequality for the last line. We use Lemma 5, in \cite{lewsab2} and get:
$$
I \lesssim \|u\|_{L^2_{\eta'}L^2_{\eta_3}}\|v\|_{L^2_{\eta'}L^2_{\eta_3}}\| \phi \|_{L^2_{\eta'}L^2_{\eta_3}}\lesssim \| U \|_{L^2_{t}H^{1/2}_x}\| V\|_{L^2_{t}H^{1/2}_x}\| \phi \|_{L^2(\mathbb R^3)}.
$$
where the last inequality is a simple consequence of the definition of $u$ and $v$.

\end{proof}

The proof of Proposition \ref{pr:continuiteQ2} requires additional results than Proposition \ref{prop:QuadV2.2}, involving higher regularity or space-time Lebesgue exponents. They are easier consequences of the linear bounds proved in the previous Section.

\begin{lemma}\label{prop:QuadV2.1} For all $U,V\in \Theta_V$, there holds:
$$
\|Q_2(U,V)\|_{L^{10/3}_{t,x}} \lesssim \|U\|_{H^{1/2}_{t,x}} \|V\|_{H^{1/2}_{t,x}}, \quad \quad \| \langle \nabla \rangle^{\frac 12} Q_2(U,V)\|_{L^{2}_{t,x}} \lesssim \|U\|_{L^2_{t}H^{1/2}_x} \|V\|_{L^2_{t}H^{1/2}_x},
$$
\end{lemma}

\begin{proof} We start with the first estimate. We recall that
\be \label{id:Q2L103inter}
Q_2(U,V) = 2\Re \E \Big( \overline{W_V(Y)}W_U(Y) + \bar Y( W_V\circ W_U(Y) + W_U\circ W_V(Y))\Big).
\ee
To treat the first term above, we employ first Cauchy-Schwarz and H\"older's inequality, then \fref{bd:potentialtorandom} with $p=q=20/3$ and $s=3/4$ together with Christ-Kiselev Lemma:
\bea
\nonumber \| \mathbb E(\overline{W_V(Y)}W_U(Y)) \|_{L^{10/3}_{t,x}}&\lesssim &\| W_V(Y)\|_{L^{20/3}_{t,x}L^2_\omega} \| W_U(Y) \|_{L^{20/3}_{t,x}L^2_\omega}\\
\label{bd:WVWUinter}&&\quad \lesssim \| V \|_{L^2_{t}\dot H^{3/4-1/2}} \| U \|_{L^2_{t}\dot H^{3/4-1/2}}\lesssim \| V \|_{L^2_{t}H^{1/2}_x} \| U \|_{L^2_{t}H^{1/2}_x}.
\eea
For the second and third terms, we use first Cauchy-Schwarz and H\"older inequalities, then the dual Strichartz estimate \fref{bd:strichartzdual}, then H\"older with $\frac7{10} = \frac1{5} + \frac1{2}$, then the linear bound \fref{bd:WVnablaY} to bound:
\bee
&&\| \mathbb E ( \bar Y W_V\circ W_U(Y))\|_{L^{10/3}_{t,x}}\lesssim \| Y\|_{L^{\infty}_{t,x}L^2_\omega}\| W_V\circ W_U(Y) \|_{L^{10/3}_{t,x}L^2_\omega}\\
&& \quad \quad \quad \quad  \lesssim \| \int_0^t S(t-s)[V W_U(Y)] ds\|_{L^{10/3}_{t,x}L^2_\omega}\lesssim \| V W_U(Y)] \|_{L^{10/7}_{t,x}L^2_\omega}\\
&&\quad \quad \quad \quad \quad \quad \quad \quad \lesssim \| V \|_{L^2_{t,x}} \|W_U(Y) \|_{L^{5}_{t,x}L^2_\omega}\lesssim \| V \|_{L^2_{t,x}} \| U \|_{L^{2}_{t,x}}\lesssim \| V \|_{L^2_{t}H^{1/2}_x} \| U \|_{L^{2}_{t}H^{1/2}_x}.
\eee
We inject the bound \fref{bd:WVWUinter} and the above bound (noticing that it treats the second and third terms simultaneously) in \fref{id:Q2L103inter}, yielding the first estimate of the Lemma.

We now turn to the second estimate. Note that $ \|Q_2(U,V)\|_{L^{2}_{t,x}} \lesssim \|U\|_{L^2_{t}H^{1/2}_x} \|V\|_{L^2_{t}H^{1/2}_x}$ from Proposition \ref{prop:QuadV2.2}. Hence, to obtain the second estimate of the Lemma, it suffices to prove:
\be \label{bd:intermediateQ2highreg}
\|P_{\geq 1}Q_2(U,V)\|_{L^{2}_{t}H^{1/2}_x} \lesssim \|U\|_{L^2_{t}H^{1/2}_x} \|V\|_{L^2_{t}H^{1/2}_x}
\ee
where $P_{\geq 1}$ projects on frequencies $|\xi|\geq1$, that is $\mathcal F(P_{\geq 1}u)={\bf 1}_{|\xi|\geq 1}\hat u$.
For the first term in \fref{id:Q2L103inter}, notice that $\partial_{x_j}\E (\overline{W_V(Y)}W_U(Y))=\E (\overline{W_{\partial_{x_j}V}(Y)}W_{U}(Y))+\E (\overline{W_{V}(Y)}W_{\partial_{x_j}U}(Y))$ from Lemma \fref{lem:exprQ21}. We thus estimate it using the frequency localisation, Cauchy-Schwarz and H\"older, \fref{bd:potentialtorandom} with $p=q=10/3$ and $s=0$ and \fref{bd:WVnablaY}:
\begin{align}
\nonumber & \left\| P_{\geq 1}\E (\overline{W_{V}(Y)}W_U(Y))\right\|_{L^2_t H^{1/2}_x}\\
\nonumber & \quad \quad \lesssim  \left\| \E (\overline{W_{V}(Y)}W_U(Y))\right\|_{L^2_t \dot H^{1}}  \lesssim  \left\| \E (\overline{W_{\nabla V}(Y)}W_U(Y))\right\|_{L^2_{t,x}}+\left\| \E (\overline{W_{V}(Y)}W_{\nabla U}(Y))\right\|_{L^2_{t,x}}\\
 \nonumber&\quad  \qquad \qquad \qquad \qquad \lesssim  \| W_{\nabla V} Y \|_{L^{10/3}_{t,x}L^2_\omega }\| W_{U} Y \|_{L^{5}_{t,x}L^2_\omega }+\| W_{\nabla U} Y \|_{L^{10/3}_{t,x}L^2_\omega }\| W_{V} Y \|_{L^{5}_{t,x}L^2_\omega } \\
\nonumber &\quad \qquad \qquad \qquad \qquad \qquad \lesssim  \| \nabla V \|_{L^2_t\dot H^{-1/2}_xL^2_\omega }\| U \|_{L^2_{t,x}}+ \| \nabla U \|_{L^2_t\dot H^{-1/2}_xL^2_\omega }\| V \|_{L^2_{t,x}}\\
\label{bd:WVWUinter2} &\quad \qquad \qquad \qquad \qquad \qquad \qquad \lesssim \| V \|_{L^2_tH^{1/2}_x}\| W_{U} Y \|_{L^2_{t,x}}.
\end{align}
Next, for the second term in \fref{id:Q2L103inter} we reason by duality. For any $\phi \in L^2_t\dot H^{-1/2}_x$ by Fubini:
\begin{align*}
&\left|\an{ \phi, \E(\bar Y W_V\circ W_U(Y))}_{t,x}\right|  =\left| \an{Y\phi,W_V\circ W_U (Y)}_{t,x,\omega}\right|= \left|\an{Y \phi ,-i\int_0^t S(t-s)(VW_U(Y))ds}_{t,x,\omega}\right|\\
&\qquad \qquad =\left| \an{\int_{\tau}^\infty S(\tau - t)[Y(t) \phi (t)]dt,V W_U(Y)}_{t,x,\omega}\right| \\
&\qquad  \qquad \qquad  \qquad \lesssim \| \int_{\tau}^\infty S(\tau - t)[Y(t)\ \phi (t)]dt \|_{L^{10/3}_{t,x}L^2_\omega}\| V \|_{L^2_{t,x}} \| W_U(Y)\|_{L^{5}_{t,x}L^2_\omega}\\
&\qquad  \qquad  \qquad \qquad  \qquad  \qquad \lesssim \| \phi \|_{L^2_t\dot H^{-1/2}_x} \| V \|_{L^2_{t,x}} \| U \|_{L^2_{t,x}}
\end{align*}
where we applied Cauchy-Schwarz, H\"older, \fref{bd:potentialtorandom} with $s=0$, $p=q=10/3$ and \fref{bd:WVnablaY}. The last term in \fref{id:Q2L103inter} is estimated by duality the very same way as the one above. This proves:
$$
\| P_{\geq 1} \E( \bar Y( W_V\circ W_U(Y) + W_U\circ W_V(Y))) \|_{L^2_t H^{1/2}_x} \lesssim \| V \|_{L^2_{t,x}} \| U \|_{L^2_{t,x}}.
$$
The above inequality, \fref{bd:WVWUinter2} and the decomposition \fref{id:Q2L103inter} imply the desired estimate \fref{bd:intermediateQ2highreg}, ending the proof of the Lemma.

\end{proof}

\begin{remark}\label{rem:laRemarque} We compare briefly with \cite{lewsab2}. To get extra derivatives on $V$, we have refined the estimates on $K_{\eta,\eta_2}$, Lemma \ref{lem-normK}, we introduced derivatives on $Q_2$ using duality arguments, Lemma \ref{prop:QuadV2.1} and on the linear response on the potential, Proposition \ref{pr:freeandpotential}. \end{remark}

We can now end the proof of the main Proposition of this section, and bound the specific quadratic term at hand.

\begin{proof}[Proof of Proposition \ref{pr:continuiteQ2}]

Using interpolation as $2\leq 5/2\leq 10/3$, Lemma \ref{prop:QuadV2.1} and Proposition \ref{prop:QuadV2.2} we get the desired bound:
\bee
\| Q_2(U,V)\|_{\Theta_V} & = & \| Q_2(U,V)\|_{L^2_{t,x}}+\| Q_2(U,V)\|_{L^{5/2}_{t,x}} \lesssim \| Q_2(U,V)\|_{L^2_{t}H^{1/2}_x}+\| Q_2(U,V)\|_{L^{10/3}_{t,x}}\\
&& \quad \quad \lesssim \| U\|_{L^2_{t}H^{1/2}_x}\| V\|_{L^2_{t}H^{1/2}_x}+\| U\|_{L^2_{t,x}}\| V\|_{L^2_{t,x}} \lesssim \| U \|_{\Theta_V}\| V\|_{\Theta_V}.
\eee

\end{proof}

\section{Remaining nonlinear terms}\label{sec:remaining}

We establish here standard bilinear estimates for all nonlinear terms except the specific quadratic ones in $Q_2$ treated in the previous Section. We start with higher order iterates of the operators $W_V$ applied to the equilibrium and with the operator $W_V$ applied to a perturbation $Z$.

\begin{lemma}\label{prop:QuadCl2} There exists $C>0$ such that for all $U,V \in \Theta_V$ and $Z\in \Theta_Z $:
\be \label{bd:quadperturbation}
\|W_V\circ W_U(Y)\|_{\Theta_Z} \leq C \|U\|_{\Theta_V} \|V\|_{\Theta_V}, \quad \mbox{and} \quad \|W_V(Z)\|_{\Theta_Z} \leq C \|V\|_{\Theta_V} \|Z\|_{\Theta_Z}.
\ee

\end{lemma}

\begin{proof} Recall the dual Strichartz inequality \fref{bd:strichartzdual}. In particular, for any $f\in L^2_\omega L^{10/7}_tW^{1/2,10/7}_x$:
$$
\| \langle \nabla \rangle^{\frac 12}\int_0^t S(t-s)f(s)ds\|_{L^2_\omega L^{5}_tL^{30/11}_x\cap L^2_\omega  L^{10/3}_{t,x} \cap L^2_\omega L^{\infty}_t L^2_x}\lesssim \| \langle \nabla \rangle^{\frac 12}f \|_{  L^2_\omega L^{10/7}_{t,x}}.
$$
Hence, using in addition to \fref{bd:strichartzdual} and the above inequality the Sobolev embedding $W^{1/2,30/11}_x(\mathbb R^3)\rightarrow L^5_x(\mathbb R^3)$ we obtain the preliminary inequality:
\be \label{bd:WVZinter}
\|W_V (Z)\|_{\Theta_Z}=\| \int_0^t S(t-s)(VZ)ds\|_{\Theta_Z} \leq C \| \langle \nabla \rangle^{\frac 12} (V Z)\|_{L^2_\omega L^{10/7}_{t,x}}.
\ee
For the first term in \fref{bd:quadperturbation} we apply the above estimate with $Z=W_U(Y)$. Applying Minkowski and then the fractional Leibniz estimate \fref{bd:fractionalleibniz2} with $\frac7{10} = \frac12 + \frac15$ gives:
$$
\| \langle \nabla \rangle^{\frac 12} (V W_V (Y))\|_{L^2_\omega L^{10/7}_{t,x}}\lesssim \| \langle \nabla \rangle^{\frac 12} (V W_V (Y))\|_{L^{10/7 }_{t,x}L^2_\omega} \lesssim \| \langle \nabla \rangle^{1/2} V \|_{L^2_{t,x}}\| \langle \nabla \rangle^{1/2} W_V (Y) \|_{L^5_{t,x}L^2_\omega} \lesssim \| V \|_{L^2_tH^{1/2}_x}^2,
$$
where we used \fref{bd:WVY} for the last inequality. Injecting the above estimate in \fref{bd:WVZinter} with $Z=W_V(Y)$ proves the first bound in \fref{bd:quadperturbation}. For the second bound, we apply the fractional Leibniz bound \fref{bd:fractionalleibniz1} with $q_1=2$, $r_1=5$ and $q_2=10/3$ and $r_2=5/2$:
$$
\| \langle \nabla \rangle^{\frac 12} (w*V Z)\|_{L^2_\omega L^{10/7}_{t,x}} \leq \| w*\langle \nabla \rangle^{\frac 12} V \|_{L^{2}_{t,x}}\| Z \|_{L^2_\omega L^{5}_{t,x}}+\| \langle \nabla \rangle Z \|_{L^{10/3}_{t,x}L^2_\omega}\| w*V \|_{L^{5/2}_{t,x}}\lesssim \| V \|_{\Theta_V}\| Z\|_{\Theta_Z}.
$$
The above bound, injected in \fref{bd:WVZinter}, proves the second bound in \fref{bd:quadperturbation}.

\end{proof}

We now turn to the linearised potential created by a perturbation of the form $W_V(Z)$, and to the quadratic potential created by a perturbation $Z$.

\begin{lemma} 
For all $Z,Z'\in \Theta_Z$ there holds:
\be \label{bd:quadembed}
\|\mathbb E(ZZ')\|_{\Theta_V}\leq C \| Z\|_{\Theta_Z}\| Z'\|_{\Theta_Z}.
\ee
For all $V \in \Theta_V$ and $Z \in \Theta_Z$, we have 
\be \label{bd:Q1}
\|Q_1(Z,V)\|_{\Theta_V} \lesssim \|Z\|_{\Theta_Z} \|V\|_{\Theta_V}.
\ee
\end{lemma}

\begin{proof} 
For the first bound, as a consequence of H\"older's and Minkowski inequalities:
$$
\| \mathbb E (ZZ')\|_{L^{5/2}_{t,x}}\lesssim \| Z\|_{L^5_{t,x}L^2_\omega}\| Z'\|_{L^5_{t,x}L^2_\omega}\lesssim \| Z\|_{L^2_\omega L^5_{t,x}}\| Z'\|_{L^2_\omega L^5_{t,x}} \lesssim \| Z\|_{\Theta_Z}\| Z'\|_{\Theta_Z}.
$$
Next, using Minkowski inequality, then the fractional Leibniz rule \fref{bd:fractionalleibniz1} and then Cauchy-Schwarz:
\bee
&& \| \langle \nabla \rangle^{1/2} \mathbb E (ZZ')\|_{L^{2}_{t,x}} \lesssim  \| \langle \nabla \rangle^{1/2} (ZZ')\|_{L^{2}_{t,x}L^1_\omega}\lesssim \| \langle \nabla \rangle^{1/2} (ZZ')\|_{L^1_\omega L^{2}_{t,x}} \\
&&\quad \quad  \lesssim  \| \| \langle \nabla \rangle^{1/2} Z \|_{L^{10/3}_{t,x}} \| Z' \|_{L^{5}_{t,x}}\|_{L^1_\omega}+\| \| \langle \nabla \rangle^{1/2} Z' \|_{L^{10/3}_{t,x}} \| Z \|_{L^{5}_{t,x}}\|_{L^1_\omega} \lesssim \| Z \|_{\Theta_Z} \| Z'\|_{\Theta_Z}
\eee
The two above bounds give \fref{bd:quadembed}. To show the second bound of the Lemma, we recall that
\be \label{def:Q1inter}
Q_1(Z,V)= 2\Re \E( \overline{W_V(Y)} Z + \bar Y W_V(Z)).
\ee
For the first term, by Cauchy-Schwarz, H\"older, and Minkowski:
$$
\| \E (\overline{W_V(Y)} Z)  \|_{L^2_{t,x}} \leq \|W_V(Y)\|_{L^5_{t,x} L^2_\omega} \|Z\|_{L^{10/3}_{t,x} L^2_\omega},
$$
$$
\| \nabla \E (\overline{W_V(Y)} Z)  \|_{L^2_{t,x}} \leq \| \nabla W_V(Y)\|_{L^5_{t,x} L^2_\omega} \|Z\|_{L^{10/3}_{t,x} L^2_\omega} + \| W_V(Y)\|_{L^5_{t,x} L^2_\omega} \| \nabla Z\|_{L^{10/3}_{t,x} L^2_\omega},
$$
so that $\| \langle \nabla \rangle \E (\overline{W_V(Y)} Z)\|_{L^2_{t,x}}\lesssim \| \langle \nabla \rangle W_V(Y) \|_{L^5_{t,x} L^2_\omega}\| \langle \nabla \rangle Z \|_{L^{10/3}_{t,x}  L^2_\omega}$ from \fref{id:vectorvaluedbesselpotential}. Complex interpolation between this bound and the one with no derivatives gives:
\be \label{bd:Q1inter1}
\| \E (\overline{W_V(Y)} Z)  \|_{L^2_{t}H^{1/2}_x} \lesssim \| \langle \nabla \rangle^{\frac 12} W_V(Y) \|_{L^5_{t,x} L^2_\omega}\| \langle \nabla \rangle^{\frac 12} Z \|_{L^{10/3}_{t,x} L^2_\omega} \lesssim \|V\|_{L^2_tH^{1/2}_x} \| Z \|_{ L^2_\omega L^{10/3}_{t,x}}\lesssim \| V \|_{\Theta_V}\| Z\|_{\Theta_Z},
\ee
where we used Minkowski and \fref{bd:WVY}. On the other hand, using H\"older, Minkowski and \fref{bd:WVY} gives:
\be \label{bd:Q1inter2}
\| \E (\overline{W_V(Y)} Z)  \|_{L^{5/2}_{t,x}} \leq \|W_V(Y)\|_{L^5_{t,x} L^2_\omega} \|Z\|_{L^{5}_{t,x} L^2_\omega}\lesssim \| V \|_{L^2_tH^{1/2}_x}\| Z \|_{L^2_\omega L^{5}_{t,x}} \lesssim \| V \|_{\Theta_V}\| Z \|_{\Theta_Z}.
\ee
For the second term we start with the following bound using Cauchy-Schwarz and \fref{bd:quadperturbation}:
\be \label{bd:L103inter}
\|  \E(\bar Y W_V(Z))\|_{L^{10/3}_{t,x}} \lesssim \|Y\|_{L^\infty_{t,x} L^2_\omega} \|W_V(Z)\|_{L^{10/3}_{t,x} L^2_\omega}\lesssim \| W_V(Z)\|_{\Theta_{Z}}\lesssim \| V\|_{\Theta_{V}}\| Z\|_{\Theta_{Z}}.
\ee
We next prove that
\be \label{bd:L2dualityinter}
\| \E(\bar Y W_V(Z))\|_{L^2_t H^{1/2}_{x}} \lesssim  \|V\|_{\Theta_V} \|Z\|_{\Theta_Z} 
\ee
by duality. Let $U \in L^2_{t,x}$, we have 
\bee
\an{U, \E(\bar Y W_V(Z))}_{t,x} &=& \an{YU,W_V(Z)}_{t,x,\omega}=\an{YU,-i\int_0^t S(t-s)((w*V(s))Z(s))ds}_{t,x,\omega}\\
&&= \an{\int_{\tau}^\infty S(\tau - t)[Y(t)U(t)]dt,(w*V) Z}_{t,x,\omega}.
\eee
By H\"older's inequality since $\frac 15+\frac 12+\frac{3}{10}=1$ and \fref{bd:WVY}, we get
\bee
|\an{U, \E(\bar Y W_V(Z))}_{t,x}| &\lesssim& \big\|\int_{\tau}^\infty S(\tau - t)[Y(t)U(t)]dt\big\|_{L^5_{t,x} L^2_\omega}\|w*V\|_{L^2_{t,x}} \|Z\|_{L^{10/3}_{t,x} L^2_\omega}\\
&&\lesssim C\|U\|_{L^2_{t,x}}\|V\|_{L^2_{t,x}} \|Z\|_{L^{10/3}_{t,x} L^2_\omega}.
\eee
We differentiate and apply again H\"older with $\frac 15+\frac 12+\frac{3}{10}=1$ and \fref{bd:WVY} to find:
\bee
&& \left|\an{U, \nabla \E(\bar Y W_V(Z))}_{t,x}\right| =\left| \an{U,  \E(\nabla \bar Y W_V(Z)+\bar Y W_{\nabla V}(Z)+\bar Y W_{V}(\nabla Z))}_{t,x}\right| \\
&=&\left| \an{\nabla YU,\int_0^t S(t-s)((w*V)Z)ds}_{t,x,\omega}+\an{YU,\int_0^t S(t-s)((w*\nabla V)Z+(w*V)\nabla Z)ds}_{t,x,\omega}\right|\\
&\lesssim & C\|U\|_{L^2_{t,x}}\| \langle \nabla \rangle V\|_{L^2_{t,x}} \|\langle \nabla \rangle Z\|_{L^{10/3}_{t,x} L^2_\omega}.
\eee
The two bounds above show $\| \langle \nabla \rangle \E(\bar Y W_V(Z))\|_{L^2_t}\lesssim \| \langle \nabla \rangle V\|_{L^2_{t,x}} \|\langle \nabla \rangle Z\|_{L^{10/3}_{t,x} L^2_\omega}$. Applying complex interpolation with the bound with zero derivatives shows:
$$
\|  \langle \nabla \rangle^{\frac 12}\E(\bar Y W_V(Z))\|_{L^2_{t,x}}\lesssim \| \langle \nabla \rangle^{\frac 12}V\|_{L^2_{t,x}} \| \langle \nabla \rangle^{\frac 12} Z\|_{L^{10/3}_{t,x} L^2_\omega}\lesssim \| V \|_{\Theta_V}\| Z \|_{\Theta_Z}.
$$
The above inequality shows \fref{bd:L2dualityinter}. Since $\frac52 \in [2,\frac{10}{3}]$, we get from \fref{bd:L2dualityinter}, \fref{bd:L103inter} and interpolation:
$$
\| \E(\bar Y W_V(Z))\|_{\Theta_V}=\| \E(\bar Y W_V(Z))\|_{L^2_{t,x}}+\| \E(\bar Y W_V(Z))\|_{L^{5/2}_{t,x}} \lesssim \|V\|_{\Theta_V} \|Z\|_{\Theta_Z}.
$$
We inject the above bound, \fref{bd:Q1inter1} and \fref{bd:Q1inter2} in the identity \fref{def:Q1inter}, showing the desired bound \fref{bd:Q1}.

\end{proof}

\section{Proof of Corollary \ref{corollaire}} \label{sec:coro}

This section is devoted to the proof of Corollary \ref{corollaire}. The following preliminary estimate is a consequence of \cite{FS} Theorem 8 and of the dual argument explained before \cite{FLLS} Theorem 2:
\be \label{bd:strichartzpotential}
\| \int_0^\infty S(-s)V(s)S(s)ds \|_{\mathfrak S^{2q'}(L^2(\mathbb R^d))}\lesssim \| V \|_{L^{p'}_tL^{q'}(\mathbb R^d)} \quad \mbox{for any } \frac{1+d}{2}<q'\leq \infty \mbox{ and } \frac{2}{p'}+\frac{d}{q'}=2.
\ee
We shall use several times that Sobolev-Schatten spaces form an increasing sequence:
\be \label{bd:schattenincreasing}
\| \gamma \|_{\mathfrak S^{s,p}}\leq \| \gamma \|_{\mathfrak S^{s',p'}} \quad \mbox{whenever } s\leq s' \mbox{ and } p\geq p',
\ee
since $\ell ^{p'} \hookrightarrow \ell ^p$ when $p \geq p'$; 
and H\"older inequality in Schatten spaces for $\frac{1}{p}+\frac{1}{q}=\frac{1}{r}$:
\be \label{bd:holderschatten}
\| \gamma \circ \gamma' \|_{\mathfrak S^{r}}\leq \| \gamma \|_{\mathfrak S^p} \| \gamma' \|_{\mathfrak S^{q}}.
\ee

\begin{proof}[Proof of Corollary \ref{corollaire}]

We only treat the case $t\rightarrow +\infty$ without loss of generality. We fix $2<q'\leq 5/2$, $p'$ satisfying \fref{bd:strichartzpotential} and introduce the space (equipped with the usual norm for sums of Banach spaces):
\be \label{id:defE}
E=L^2_tL^2_x\cap L^{\infty}_tL^2_x+ L^2_tL^2_x\cap L^{p'}_tL^{q'}_x.
\ee

\textbf{Step 1} \emph{Preliminary bound}. We claim that there holds:
\be \label{bd:VinE}
V\in E.
\ee
and now show this bound. We rewrite the fixed point equation \fref{id:fixedpointmainth} for $V$:
$$
V=(\text{Id}-L_2)^{-1}\tilde V, \quad \tilde V=2\textrm{Re }\E(\bar Y S(t) Z_0)+\E(|Z|^2) +2  \textrm{Re } \E(\overline{W_V(Y)} Z + \bar Y W_V(Z)) +  Q_2(V),
$$
where $L_2$ and $Q_2$ are defined in \fref{def:Q2} and \fref{def:L2}. For the first, fourth and fifth terms, using 
\[
\|S(t)Z_0\|_{L^2_\omega L^2_x}\leq \|Z_0\|_{L^2_\omega L^2_x}
\]
 and $\| Y\|_{L^{\infty}_xL^2_\omega}<\infty$, Cauchy-Schwarz, \fref{bd:quadperturbation} and \fref{bd:Q2} we obtain:
$$
2\textrm{Re }\E(\bar Y S(t) Z_0)+2 \textrm{Re } \E (\bar Y W_V(Z)) +  Q_2(V)\in L^{\infty}_tL^2_x.
$$
Next, recall $Z,W_V(Y)\in L^{\infty}_tL^3_xL^2_\omega\cap L^{5}_tL^5_xL^2_\omega$ from \fref{id:ThetaZ} and the embedding of $H^{1/2}(\mathbb R^3)$ into $L^{3}(\mathbb R^3)$, and Proposition \ref{lem:estonWUY}. Hence by H\"older $|Z|^2,\overline{W_V(Y)} Z \in L^{\infty}_tL^{3/2}_xL^1_\omega\cap L^{5/2}_tL^{5/2}_xL^1_\omega$. Notice $(\infty,3/2)$ and $(5/2,5/2)$ both satisfy the condition in \fref{bd:strichartzpotential}. Hence by Cauchy-Schwarz and  interpolation:
$$
\E(|Z|^2) +2  \textrm{Re } \E(\overline{W_V(Y)} Z ) \in L^{p'}_tL^{q'}_x.
$$
We also recall that all terms in the definition of $\tilde V$ belong to $L^2_tL^2_x$ from \fref{bd:improvedrandomtopotential}, \fref{bd:quadembed}, \fref{bd:Q1inter1}, \fref{bd:L2dualityinter} and \fref{bd:Q2}. From this and the two bounds above we get the first estimate:
$$
\tilde V \in L^2_tL^2_x\cap L^{\infty}_tL^2_x+ L^2_tL^2_x\cap L^{p'}_tL^{q'}_x.
$$
Let us now prove that $(\text{Id}-L_2)^{-1}-\text{Id}$ is continuous from $L^2_tL^2_x$ onto $L^{\infty}_tL^{2}_x$. Recall that $L_2$ is a space-time Fourier multiplier of symbol $m_f$ defined in \fref{def:mf}. Hence $(\text{Id}-L_2)^{-1}-\text{Id}$ has space-time Fourier symbol $m_f/(1-m_f)$. Recall that there exists $c>0$ such that $c\leq |1-m_f|\leq c^{-1}$ from Proposition \ref{pr:bdL2}. Moreover, from \fref{def:mf} and Parseval for any $\xi\neq 0$:
$$
\| m_f(\cdot,\xi)\|_{L^{2}_\omega}^2\lesssim \int_{\mathbb R} |\sin(|\xi|^2t)|^2 |h(2\xi t)1_{t\geq 0}|^2dt \lesssim \int_0^\infty rh^2(r),
$$
using $|\sin(|\xi|^2t)|^2\leq |\xi|^2|t|$, the radiality of $h$ and performing the change of variables $r=2|\xi| t$. Indeed, we have, by Minkowski's inequality,
\begin{multline*}
 \|((\text{Id}-L_2)^{-1}-\text{Id})\tilde V\|_{L^\infty_t L^2_x} =\| \mathcal F_x[((\text{Id}-L_2)^{-1}-\text{Id})\tilde V]\|_{L^\infty_t L^2_\xi}\\
 \leq \| \mathcal F_x[((\text{Id}-L_2)^{-1}-\text{Id})\tilde V]\|_{L^2_\xi L^\infty_t}\\
 \leq  \| \mathcal F_{t,x}[((\text{Id}-L_2)^{-1}-\text{Id})\tilde V]\|_{L^2_\xi L^1_\omega}= \|\frac{m_f}{1-m_f}\mathcal F_{t,x}\tilde V\|_{L^2_\xi,L^1_\omega}\\
 \lesssim \| m_f\|_{L^{\infty}_\xi L^{2}_\omega}\| \mathcal F_{t,x}\tilde V\|_{L^2_\xi,L^2_\omega}\lesssim \| \tilde V \|_{L^2_{t,x}}.
\end{multline*}
Therefore, writing $V=\tilde V+((\text{Id}-L_2)^{-1}-\text{Id})\tilde V$, we see that we proved above that $\tilde V$ belongs to $L^2_tL^2_x\cap L^{\infty}_tL^2_x+ L^2_tL^2_x\cap L^{p'}_tL^{q'}_x$, while $((\text{Id}-L_2)^{-1}-\text{Id})\tilde V$ belongs to $L^{\infty}_tL^{2}_x\cap L^2_tL^2_x$. This proves the claim.\\

\textbf{Step 2} \emph{Bounds for $W_{V,\pm}$}. We decompose between high and low frequencies:
$$
\mbox{for }\eta^*>0, \quad P_{\leq \eta^*}u=\mathcal F^{-1} (\chi (\frac{\eta}{\eta^*})\hat u), \quad P_{\geq \eta^*}=1-P_{\leq \eta^*}, \quad \mbox{and}\quad V_{\leq 1}=P_{\leq 1}V, \quad \quad V_{\geq 1}=P_{\geq 1}V,
$$
where $\chi$ is smooth with $\chi(\xi)\equiv 1$ for $|\xi|\leq 1$ and $\chi(\xi)\equiv 0$ for $|\xi|\geq 2$ and claim that:
\be \label{bd:WVlocalised}
\| \langle \nabla \rangle^{\frac 12} W_{V\leq 1,+}\langle \nabla \rangle^{-\frac 12}\|_{\mathfrak S^{2q'}}\lesssim \| V \|_{E}\quad \mbox{and} \quad \| \langle \nabla \rangle^{\frac 12} W_{V\geq 1,+} \gamma_f  \langle \nabla \rangle^{\frac 12}\|_{\mathfrak S^{2}}\lesssim \| V \|_{L^2_tL^2_x}.
\ee
We now prove the above bounds. For the first one, need a Leibniz type formula for fractional differentiation. We decompose:
\be \label{id:decompositionWVleq1}
\langle \nabla \rangle^{\frac 12} W_{V\leq 1,+}\langle \nabla \rangle^{-\frac 12}=\langle \nabla \rangle^{\frac 12} W_{V\leq 1,+}\langle \nabla \rangle^{-\frac 12}P_{\leq \eta^*}+\langle \nabla \rangle^{\frac 12} W_{V\leq 1,+}\langle \nabla \rangle^{-\frac 12}P_{\geq \eta^*}.
\ee
Since $V_{\leq 1}$ is localised in frequencies $|\eta|\leq 2$, since $p',q'>2$, we have from Bernstein's inequality at $t$ fixed,
$$
\|V_{\leq 1}(t,\cdot)\|_{L^{q'}_x} \lesssim \|V_{\leq 1}(t,\cdot)\|_{L^2_x}
$$ 
and since $p' \in [1,\infty]$, $\| V \|_{L^{p'}_tL^{q'}_x}\lesssim \| V \|_{L^{2}_tL^{2}_x}+\| V \|_{L^{\infty}_tL^{2}_x}$. Therefore from \fref{id:defE}:
\be \label{bd:Vlocalisedstrichartz}
\| V_{\leq 1}\|_{L^{p'}_tL^{q'}_x}\lesssim \| V \|_{E}.
\ee
From this, \fref{bd:strichartzpotential} and \fref{bd:holderschatten}, since $V_{\leq 1}$ is localised in frequencies $|\eta|\leq 2$, for any $u$ localised in frequencies $|\eta| \leq 2\eta^*$, their product is localised in frequencies $|\eta| \leq 2\eta^* + 2$, therefore
$$
\an{\nabla}^{\frac12} W_{V_{\leq 1},+}\an{\nabla}^{-\frac12} P_{\leq \eta^*} = P_{\leq 2\eta^*+2} \an{\nabla}^{\frac12} W_{V_{\leq 1},+}\an{\nabla}^{-\frac12} P_{\leq \eta^*}.
$$
We get for the first term, using H\"older's inequality for Schatten spaces \eqref{bd:holderschatten}:
\bea 
\nonumber && \| \langle \nabla \rangle^{\frac 12}W_{V_{\leq 1},+} \langle \nabla \rangle^{-\frac 12} P_{\leq \eta^*}\|_{\mathfrak S^{2q'}} 
 \leq \| P_{\leq 2\eta^* +2}\langle \nabla \rangle^{\frac 12}\|_{C(L^2(\mathbb R^3))} \cdot \| W_{V_{\leq 1},+} \|_{\mathfrak S^{2q'}} \cdot \| \langle \nabla \rangle^{-\frac 12} P_{\leq \eta^*}\|_{C(L^2(\mathbb R^3))}\\
 \label{bc:interWV2} && \lesssim \an{2 \eta^* + 2}^{\frac12} \cdot \| V \|_{L^{p'}_tL^{q'}_x}\cdot 1 \lesssim \| V\|_{E}
\eea
where $ \|\cdot \|{C(L^2(\mathbb R^3))}$ is the operator norm on $L^2$ which corresponds to $\mathfrak S^\infty$. For the second part, since $V_{\leq 1}$ is localised in frequencies $|\eta|\leq 2$, if $u$ is localised in frequencies $|\eta|\geq \eta^*$ their product is localised in frequencies $|\eta|\geq \eta^* -2$. Hence, we have 
$$
 \langle \nabla \rangle^{\frac 12}W_{V_{\leq 1},+} \langle \nabla \rangle^{-\frac 12} P_{\geq \eta^*} =  \langle \nabla \rangle^{\frac 12}P_{\geq \eta^* -2}W_{V_{\leq 1},+} \langle \nabla \rangle^{-\frac 12} P_{\geq \eta^*} .
$$
Taking $\eta^* > 2$, using the commutativity of Fourier multipliers, we get:
$$
 \langle \nabla \rangle^{\frac 12}W_{V_{\leq 1},+} \langle \nabla \rangle^{-\frac 12} P_{\geq \eta^*} =  \langle \nabla \rangle^{\frac 12} |\nabla|^{-\frac12}P_{\geq \eta^* -2} \quad |\nabla|^{\frac12} W_{V_{\leq 1},+} |\nabla|^{-\frac12}\quad \langle \nabla \rangle^{-\frac 12}|\nabla|^{1/2} P_{\geq \eta^*} .
$$
Using that both $ \langle \nabla \rangle^{\frac 12} |\nabla|^{-\frac12}P_{\geq \eta^* -2} $ and $\langle \nabla \rangle^{-\frac 12}|\nabla|^{1/2} P_{\geq \eta^*}$ belong to $C(L^2_x) = \mathfrak S^\infty$, we get by H\"older's inequality \fref{bd:holderschatten}
$$
\| \langle \nabla \rangle^{\frac 12}W_{V_{\leq 1},+} \langle \nabla \rangle^{-\frac 12} P_{\geq \eta^*} \|_{\mathfrak S^{2q'}}\lesssim \| |\nabla |^{\frac 12}W_{V_{\leq 1},+}P_{\geq \eta^*} |\nabla |^{-\frac 12} \|_{\mathfrak S^{2q'}}.
$$
We then prove a Leibniz type inequality. Assume $|\xi-\eta|\leq 2$ and $|\eta|\geq \eta^*$, then the entire series expansion
\bee
&&|\xi|^{\frac 12}=|\eta|^{\frac 12}\left(1+2\frac{\eta_1(\xi_1-\eta_1)+\eta_2(\xi_2-\eta_2)+\eta_3(\xi_3-\eta_3)}{|\eta|^2}+\frac{|\xi-\eta|^2}{|\eta|^2}\right)^{\frac 14}\\
&&\quad \quad =|\eta|^{\frac 12}\sum_{k,l,m,n \in \mathbb N}a_{k,l,m,n} \frac{\eta^{k,l,m}}{|\eta|^{2(k+l+m)}}(\xi-\eta)^{k,l,m} \frac{|\xi-\eta|^{2n}}{|\eta|^{2n}}\quad \quad \mbox{where}\quad \eta^{k,l,m}=\eta_1^k\eta_2^l\eta_3^m,
\eee
holds for $\eta^*$ large enough. The coefficients $a_{k,l,m,n}$ are given by
$$
a_{k,l,m,n} = \prod_{j=0}^{k+l+n+m-1} \Big( \frac14 -j\Big) \frac1{k!l!m!n!} 2^{k+l+m}
$$
and thus satisfy $|a_{k,l,m,n}|\leq 7^{k+l+m+n}$. We thus decompose:
\bee
\mathcal F\left( |\nabla|^{\frac 12} V_{\leq 1}(s)P_{\geq \eta^*}u \right) & = & \int_{\mathbb R^3}d\eta |\xi|^{\frac 12} \hat V_{\leq 1}(s,\xi-\eta)\hat u(\eta) \left(1-\chi\left(\frac{\eta}{\eta^*}\right)\right) \\
&=& \sum_{k,l,m,n\in \mathbb N} a_{k,l,m,n} \mathcal F\left( V^{k,l,m,n}_{\leq 1}(s) P_{\geq \eta^*} \tilde \nabla^{k,l,m,n} |\nabla|^{\frac 12}u \right) 
\eee
where
$$
\widehat{V^{k,l,m,n}_{\leq 1}}(s,\xi-\eta)=(\xi-\eta)^{k,l,m}|\xi-\eta|^{2n}\hat V_{\leq 1}(s,\xi-\eta), \quad \quad \widehat{\tilde \nabla^{k,l,m,n}u}(\eta)=\frac{\eta^{k,l,m}}{|\eta|^{2(k+l+m+n)}}\hat u(\eta).
$$
This gives the identity:
$$
|\nabla |^{\frac 12}W_{V_{\leq 1},+}P_{\geq \eta^*} |\nabla |^{-\frac 12} =\sum_{k,l,m,n\in \mathbb N}a_{k,l,m,n} W_{V^{k,l,m,n}_{\leq 1},+}P_{\geq \eta^*} \tilde \nabla^{k,l,m,n}.
$$
Above, using \fref{bd:strichartzpotential}, \fref{bd:Vlocalisedstrichartz}, the localisation of $V_{\leq 1}$ at frequencies $\leq 2$ and \fref{bd:holderschatten}:
\bea
\nonumber && \| |\nabla |^{\frac 12}W_{V_{\leq 1},+}P_{\geq \eta^*} |\nabla |^{-\frac 12}\|_{\mathfrak S^{2q'}}\leq \sum_{k,l,m,n\in \mathbb N} |a_{k,l,m,n}| \| W_{V^{k,l,m,n}_{\leq 1},+}\|_{\mathfrak S^{2q'}} \| P_{\geq \eta^*} \tilde \nabla^{k,l,m,n}\|_{C(L^2(\mathbb R^3))}\\
\label{bc:interWV3} && \quad \lesssim \sum_{k,l,m,n\in \mathbb N} 7^{k+l+m+n} 2^{k+l+m+2n}\| V_{\leq 1} \|_{L^{p'}_tL^{q'}_x} \frac{1}{\eta^{*k+l+m+2n}} \lesssim \| V_{\leq 1} \|_{L^{p'}_tL^{q'}_x}\lesssim \| V \|_{E} 
\eea
for $\eta^*>0$ large enough. We inject the bounds \fref{bc:interWV2} and \fref{bc:interWV3} in \fref{id:decompositionWVleq1}, which proves the first inequality in \fref{bd:WVlocalised}. We now turn to the second one. Recalling $|Y_0\rangle \langle Y_0|=\gamma_f$ one gets:
\bea
\nonumber \mathbb E (|W_{V,+}(Y_0)\rangle \langle Y_0)|) &=& \mathbb E(|-i\int_0^\infty S(-s)V(s)S(s)Y_0\rangle \langle Y_0|)\\
\label{id:WV+Y0}&& \quad \quad =-i\int_0^\infty S(-s)V(s)S(s) \circ  \mathbb E (|Y_0\rangle \langle Y_0| )=-i W_{V,+} \circ \gamma_f.
\eea
and hence:
\be \label{id:decompositionWVgeq1}
\langle \nabla \rangle^{\frac 12}  W_{V\geq 1,+}\gamma_f\langle \nabla \rangle^{\frac 12}=\mathbb E(|\langle \nabla \rangle^{\frac 12} W_{V\geq 1,+}(Y_0) \rangle \langle \langle \nabla \rangle^{\frac 12}  Y_0|).
\ee
From \fref{bd:potentialtorandom}, the identity $\dot{\tilde H}_{x,\omega}^0=L^2_\omega L^2_x$ and the localisation of $V_{\geq 1}$ at frequencies $\geq 1$ we obtain:
$$
\| W_{V\geq 1,+}(Y_0)\|_{L^2_\omega L^2_x}\lesssim \| V_{\geq 1}\|_{L^2_t\dot H^{-\frac 12}}\lesssim \| V_{\geq 1}\|_{L^2_tL^2_x}\lesssim \| V\|_{L^2_tL^2_x}.
$$
Similarly, from the bounds \fref{bd:continuityTT*3} and \fref{bd:potentialtorandom} with $s=0$:
\begin{multline*}
\| \nabla (W_{V\geq 1,+}(Y_0))\|_{L^2_\omega L^2_x}\leq \| W_{V\geq 1,+}(\nabla  Y_0)+ W_{\nabla V\geq 1,+}(Y_0)\|_{L^2_\omega L^2_x}\lesssim \| V_{\geq 1}\|_{L^2_t\dot H^{-1/2}_x}+\| \nabla V_{\geq 1}\|_{L^2_t\dot H^{-1/2}_x}\\
\lesssim \| V\|_{L^2_tH^{1/2}_x}.
\end{multline*}
Hence:
\be \label{bd:potentialtorandomrefined}
\| W_{V\geq 1,+}(Y_0)\|_{L^2_\omega H^{1}_x}\lesssim \| V\|_{L^2_tH^{1/2}_x}.
\ee
From \cite{CodS} Appendix A, the following pointwise bound is proved for $f>0$ is bounded, radial:
\be \label{bd:pointwisefweighted}
\| f^2(\xi) \langle \xi \rangle \|_{L^{\infty}}<\infty, \quad \mbox{ if } \pa_r f<0 \mbox{ and } \int_{\mathbb R^3} f^2\langle \xi \rangle^{1}d\xi<\infty,
\ee
and it is proved that we have for any $u\in L^2(\mathbb R^3)$, that $\langle Y_0, u \rangle$ is a centred Gaussian variable with:
\be \label{id:gaussiandual}
\langle Y_0, u \rangle=\int_{\mathbb R^3} f(\xi)\hat u(\xi)d\overline{W(\xi)}, \quad \quad  \mathbb E \left( |\langle Y_0, u \rangle|^2 \right)=\int_{\mathbb R^3} f^2(\xi)|\hat u(\xi)|^2\lesssim \| \langle \nabla \rangle^{-\frac 12} u \|_{L^2}.
\ee
We recall that for an operator $\gamma$ with kernel $k_\gamma (x,y)$ there holds $\| \gamma\|_{\mathfrak S^{2}}=\| k_\gamma \|_{L^2(\mathbb R^3\times \mathbb R^3)}$, and that the integral kernel of $\mathbb E(|a \rangle \langle b|) $ is $\mathbb E (a(x)\overline{b(y)})$. Therefore, by duality, we get that
\bee
&&\| \mathbb E(|\langle \nabla \rangle^{\frac 12} W_{V\geq 1,+}(Y_0) \rangle \langle \langle \nabla \rangle^{\frac 12}  Y_0|)  \|_{\mathfrak S^{2}}^2 = \|\mathbb E (\langle \nabla \rangle^{\frac 12}W_{V\geq 1,+}Y_0(x)  \langle \nabla\rangle^{\frac 12} \bar Y_0(y))\|_{L^2_{x,y}}\\
& &\qquad \qquad \qquad \qquad \qquad \qquad = \sup_{\| u\|_{L^2(\mathbb R^{6})}=1} \left| \int_{\mathbb R^d\times \mathbb R^d} \mathbb E(\langle \nabla \rangle^{\frac 12}W_{V\geq 1,+}Y_0(x)  \langle \nabla\rangle^{\frac 12} \bar Y_0(y) u(x,y)dxdy\right|.
\eee
Using Fubini and Cauchy-Schwarz, then \fref{id:gaussiandual}, and finally \fref{bd:potentialtorandomrefined} this is
\bee
...& \leq&  \sup_{\| u\|_{L^2(\mathbb R^{2d})}=1}\int_{\mathbb R^d} \left(\mathbb E(|\langle \nabla \rangle^{\frac 12}W_{V\geq 1,+}Y_0(x)|^2)\right)^{\frac 12} \left(\mathbb E\left(| \int_{\mathbb R^d} \langle \nabla \rangle^{\frac 12} \bar Y_0(y) u(x,y)dy|^2\right)\right)^{\frac 12}dx \\
&& \quad \lesssim \sup_{\| u\|_{L^2(\mathbb R^{2d})}=1}\int_{\mathbb R^d} \left(\mathbb E(|\langle \nabla \rangle^{\frac 12}W_{V\geq 1,+}Y_0(x)|^2)\right)^{\frac 12} \left(\int_{\mathbb R^d}|u(x,y)|^2dy \right)^{\frac 12}dx \\
&& \quad \quad \lesssim \| \langle \nabla \rangle^{\frac 12}W_{V\geq 1,+}Y_0\|_{L^2_xL^2_\omega} \lesssim \| V\|_{L^2_tH^{1/2}_x}.
\eee
The above bound, via the identity \fref{id:decompositionWVgeq1}, proves the second bound in \fref{bd:WVlocalised}.\\

\noindent \textbf{Step 3} \emph{Proof that $\gamma_+\in \mathfrak S^{\frac 12,2q'}$}. Recall $\gamma_+$ is defined by \fref{def:gammapm}. We write (where $Y_0=Y_f(t=0)$):
$$
\begin{array}{lllll}
&\gamma_+ &= \mathbb E\left( |Z_{+}\rangle \langle Z_+ |+|Z_{+}\rangle \langle Y_0|+|Y_0 \rangle \langle Z_+|\right) & \}&=:\gamma_+^1\\
&& \quad +\mathbb E\left(|W_{V,+}(Y_0)\rangle \langle W_{V,+}(Y_0) | +|W_{V,+}(Y_0)\rangle \langle Y_0|+|Y_0 \rangle \langle W_{V,+}(Y_0)|\right)&  \}&=:\gamma_+^2\\
&&\quad \quad +\mathbb E\left( |W_{V,+}(Y_0)\rangle \langle Z_+ |+ |Z_+ \rangle \langle W_{V,+}(Y_0) |\right) & \}&=:\gamma_+^3
\ea
$$
From \cite{CodS} Appendix A we have the following result: the bound for any $Z,Z' \in L^2_\omega H^{d/2-1}_x$:
\be \label{bd:ZY0operator}
\|\mathbb E(|Z \rangle \langle Y_0|)\|_{\mathfrak S^{\frac{d}{2}-1,2}}+\|\mathbb E(|Y_0 \rangle \langle Z|) \|_{\mathfrak S^{\frac{d}{2}-1,2}}\lesssim \| Z \|_{L^2_\omega H^{\frac 12}_x}, \quad \| \mathbb E(|Z \rangle \langle Z' |)\|_{\mathfrak S^{\frac{d}{2}-1,2}}\lesssim\| Z \|_{L^2_\omega H^{\frac 12}_x}\| Z' \|_{L^2_\omega H^{\frac 12}_x}.
\ee
We claim the following linear and bilinear bounds for $V,V' \in E$ and $Z\in L^2_\omega H^{\frac 12}_x$:
\be \label{bd:WVY0operator}
\| \mathbb E (|W_{V,+}(Y_0)\rangle \langle W_{V',+}(Y_0) |)\|_{\mathfrak S^{\frac 12, 2q'}}\lesssim \| V \|_E \| V' \|_E, \quad \|\E(|W_{V,+}(Y_0)\rangle \langle Y_0|)\|_{\mathfrak S^{\frac 12, 2q'}}\lesssim \| V \|_E,
\ee
\be \label{bd:WVY0Zoperator}
\| \mathbb E (|W_{V,+}(Y_0)\rangle \langle Z |)\|_{\mathfrak S^{\frac 12, 2q'}}\lesssim \| V \|_E \| Z \|_{L^2_\omega H^{\frac 12}_x}.
\ee
We now prove these bounds. We start with the bilinear terms in $V,V'$. For the first one in \fref{bd:WVY0operator} we write:
$$
W_{V,+}\gamma_f W_{V',+} = W_{V_{\leq 1},+}\gamma_f W_{V'_{\leq 1},+}+W_{V_{\leq 1},+}\gamma_f W_{V'_{\geq 1},+}+W_{V_{\geq 1},+}\gamma_f W_{V'_{\leq 1},+}+W_{V_{\geq 1},+}\gamma_f W_{V'_{\geq 1},+}.
$$
When both $V$ and $V'$ are localised in low frequencies, since $q'<2q'$, we bound, via \eqref{bd:schattenincreasing}, \fref{bd:holderschatten}, \fref{bd:WVlocalised} (for the dual operator) and \fref{bd:pointwisefweighted}:
\bee
&& \| \langle \nabla \rangle^{\frac 12} W_{V_{\leq 1},+}\gamma_f W_{V'_{\leq 1},+}\langle \nabla \rangle^{\frac 12} \|_{\mathfrak S^{2q'}} \leq \| \langle \nabla \rangle^{\frac 12} W_{V_{\leq 1},+}\gamma_f W_{V'_{\leq 1},+}\langle \nabla \rangle^{\frac 12} \|_{\mathfrak S^{q'}} \\
&=&\| \langle \nabla \rangle^{\frac 12} W_{V_{\leq 1},+}\langle \nabla \rangle^{-\frac 12}\langle \nabla \rangle^{\frac 12}  \gamma_f\langle \nabla \rangle^{\frac 12} \langle \nabla \rangle^{-\frac 12}  W_{V'_{\leq 1},+}\langle \nabla \rangle^{\frac 12} \|_{\mathfrak S^{q'}}\\
&\leq& \| \langle \nabla \rangle^{\frac 12} W_{V_{\leq 1},+}\langle \nabla \rangle^{-\frac 12}\|_{\mathfrak S^{2q'}} \|\langle \nabla \rangle^{\frac 12}  \gamma_f\langle \nabla \rangle^{\frac 12}\|_{C(L^2(\mathbb R^3))} \| \langle \nabla \rangle^{-\frac 12}  W_{V'_{\leq 1},+}\langle \nabla \rangle^{\frac 12} \|_{\mathfrak S^{2q'}} \lesssim \| V\|_{E}\| V'\|_{E}.
\eee
When either $V$ is localised in high frequencies and $V'$ in low frequencies or the opposite, using in addition that the $\mathfrak S^{0,p}$ norm of an operator is equal to the $\mathfrak S^{0,p}$  norm of its adjoint, we get, since $2 \frac{q'}{1+q'} \leq 2q'$:
\bee
& & \| \langle \nabla \rangle^{\frac 12} \left(W_{V_{\leq 1},+}\gamma_f W_{V'_{\geq 1},+}+W_{V_{\geq 1},+}\gamma_f W_{V'_{\leq 1},+}\right) \langle \nabla \rangle^{\frac 12} \|_{\mathfrak S^{2q'}} \\
& \leq & \| \langle \nabla \rangle^{\frac 12} \left(W_{V_{\leq 1},+}\gamma_f W_{V'_{\geq 1},+}+W_{V_{\geq 1},+}\gamma_f W_{V'_{\leq 1},+}\right) \langle \nabla \rangle^{\frac 12} \|_{\mathfrak S^{\frac{2q'}{1+q'}}}\\
&\leq & \| \langle \nabla \rangle^{\frac 12} W_{V_{\leq 1},+}\langle \nabla \rangle^{-\frac 12} \langle \nabla \rangle^{\frac 12} \gamma_f  W_{V'_{\geq 1},+} \langle \nabla \rangle^{\frac 12} \|_{\mathfrak S^{\frac{2q'}{1+q'}}}+ \| \langle \nabla \rangle^{\frac 12} W_{V'_{\leq 1},+}\langle \nabla \rangle^{-\frac 12} \langle \nabla \rangle^{\frac 12}\gamma_f  W_{V_{\geq 1},+} \langle \nabla \rangle^{\frac 12} \|_{\mathfrak S^{\frac{2q'}{1+q'}}}\\
&\lesssim & \| \langle \nabla \rangle^{\frac 12} W_{V_{\leq 1},+}\gamma_f \langle \nabla \rangle^{\frac 12} \|_{\mathfrak S^{2q'}} \|\langle \nabla \rangle^{-\frac 12}  W_{V'_{\geq 1},+} \langle \nabla \rangle^{\frac 12} \|_{\mathfrak S^{2}}+\text{symmetric}\\
&\lesssim  & \| V\|_{E}\| V'\|_{L^2_tL^2_x}+\| V\|_{L^2_tL^2_x}\| V'\|_{E}
\eee
(where symmetric means the same expression with $V$ and $V'$ interverted). Next, when both $V$ and $V'$ are localised in high frequencies, we write $W_{V_{\geq 1},+}\gamma_f W_{V'_{\geq 1},+}=\mathbb E |W_{V_{\geq 1},+}Y_0 \rangle \langle W_{V'_{\geq 1},+}Y_0|$, use that $2\leq 2q'$, Cauchy-Schwarz and \fref{bd:potentialtorandomrefined}:
\bee 
&& \| \langle \nabla \rangle^{\frac 12}W_{V_{\geq 1},+}\gamma_f W_{V'_{\geq 1},+}\langle \nabla \rangle^{\frac 12}\|_{\mathfrak S^{2q'}} \leq 
\| \langle \nabla \rangle^{\frac 12}W_{V_{\geq 1},+}\gamma_f W_{V'_{\geq 1},+}\langle \nabla \rangle^{\frac 12}\|_{\mathfrak S^{2}} \\
&=& \int_{\mathbb R^6}dxdy \left| \mathbb E \left(\langle \nabla \rangle^{\frac 12}W_{V_{\geq 1},+}Y_0(x)\langle \nabla \rangle^{\frac 12}W_{V'_{\geq 1},+}\overline{Y_0}(y) \right) \right|^2\\
&&\quad \lesssim  \int_{\mathbb R^6}dxdy \left(\mathbb E  |\langle \nabla \rangle^{\frac 12}W_{V_{\geq 1},+}Y_0(x)|^2\right) \left(\mathbb E |\langle \nabla \rangle^{\frac 12}W_{V'_{\geq 1},+}\overline{Y_0}(y)|^2 \right)\\
&& \quad \quad \lesssim \| W_{V_{\geq 1},+}Y_0\|_{L^2_\omega H^{\frac 12}}^2\| W_{V'_{\geq 1},+}Y_0\|_{L^2_\omega H^{\frac 12}}^2 \lesssim \| V \|_{L^2_tL^2_x}\| V' \|_{L^2_tL^2_x}.
\eee
The above decomposition and three following bounds, using \fref{bd:schattenincreasing}, prove the first bound in \fref{bd:WVY0operator}. 

The second bound in \fref{bd:WVY0operator} consists in estimating $\E(|W_{V,+}(Y_0)\rangle \langle Y_0|)$ in $\mathfrak S^{\frac12,2q'}$. We decompose it into
$$
\E(|W_{V,+}(Y_0)\rangle \langle Y_0|) = \E(|W_{V_{\leq 1},+}(Y_0)\rangle \langle Y_0|)+\E(|W_{V_{\geq 1},+}(Y_0)\rangle \langle Y_0|)
$$
thanks to the linearity of $W_{V,+}$ in $V$. For the high frequencies, since $2$ is less than $2q'$, using what we have proven in \eqref{bd:WVlocalised}:
$$
\|\E(|W_{V_{\geq 1},+}(Y_0)\rangle \langle Y_0|)\|_{\mathfrak S^{\frac12,2q'}} \leq \|\E(|W_{V_{\geq 1},+}(Y_0)\rangle \langle Y_0|)\|_{\mathfrak S^{\frac12,2}} \lesssim \|V\|_{E}.
$$
For the low frequencies, we use identity \eqref{id:WV+Y0} to get $\E(|W_{V_{\leq 1},+}(Y_0)\rangle \langle Y_0|) = W_{V_{\leq 1},+}\gamma_f$, so that:
$$
\an{\nabla}^{\frac12} \E(|W_{V_{\leq 1},+}(Y_0)\rangle \langle Y_0|) \an{\nabla}^{\frac12} = \an{\nabla}^{\frac12} W_{V_{\leq 1},+} \an{\nabla}^{-\frac12} \an{\nabla}^{\frac12} \gamma_f \an{\nabla}^{\frac12}.
$$
Using \eqref{bd:holderschatten}, we get
$$
\|\E(|W_{V_{\leq 1},+}(Y_0)\rangle \langle Y_0|)\|_{\mathfrak S^{1/2,2q'}}\leq \|\an{\nabla}^{\frac12} W_{V_{\leq 1},+} \an{\nabla}^{-\frac12} \|_{\mathfrak S^{2q'}} \|\an{\nabla}^{\frac12} \gamma_f \an{\nabla}^{\frac12}\|_{\mathfrak S^\infty}.
$$
Finally, we use that $\langle \nabla \rangle \gamma_f$ is a bounded Fourier multiplier from \fref{bd:pointwisefweighted} and \eqref{bd:WVlocalised} to get the second bound in \fref{bd:WVY0operator}. Hence \fref{bd:WVY0operator} is established.

We turn to the third and last estimate \fref{bd:WVY0Zoperator} and, to prove it, write from \fref{id:WV+Y0}:
$$
\mathbb E |W_{V,+}(Y_0)\rangle \langle Z |=\mathbb E|(W_{V_{\leq 1},+}+W_{V_{\geq 1},+})(Y_0)\rangle \langle Z |=W_{V_{\leq 1},+} \mathbb E|Y_0\rangle \langle Z |+\mathbb E |W_{V_{\geq 1},+}Y_0\rangle \langle Z |.
$$
For the first term, using H\"older, \fref{bd:WVlocalised} and \fref{bd:ZY0operator}:
\bee
&& \| \langle \nabla \rangle^{\frac 12} W_{V_{\leq 1},+} \mathbb E|Y_0\rangle \langle Z |\langle \nabla \rangle^{\frac 12} \|_{\mathfrak S^{\frac{2q'}{1+q'}}} = \| \langle \nabla \rangle^{\frac 12} W_{V_{\leq 1},+} \langle \nabla \rangle^{-\frac 12}\langle \nabla \rangle^{\frac 12}  \mathbb E|Y_0\rangle \langle Z |\langle \nabla \rangle^{\frac 12} \|_{\mathfrak S^{\frac{2q'}{1+q'}}} \\
&\lesssim & \| \langle \nabla \rangle^{\frac 12} W_{V_{\leq 1},+} \langle \nabla \rangle^{-\frac 12}\|_{\mathfrak S^{2q'}} \| \langle \nabla \rangle^{\frac 12}  \mathbb E|Y_0\rangle \langle Z |\langle \nabla \rangle^{\frac 12} \|_{\mathfrak S^{2}} \lesssim \| V\|_E \| Z\|_{L^2_\omega H^{\frac 12}_x},
\eee
while for the second term, using the third inequality in \fref{bd:ZY0operator} and \fref{bd:potentialtorandomrefined}:
$$
\| \mathbb E |W_{V_{\geq 1},+}Y_0\rangle \langle Z |\\|_{\mathfrak S^{\frac 12,2}} \lesssim \| W_{V_{\geq 1},+}Y_0\|_{L^2_\omega H^{\frac 12}_x} \| Z \|_{L^2_\omega H^{\frac 12}_x}\lesssim \| V\|_{L^2_tL^2_x} \| Z \|_{L^2_\omega H^{\frac 12}_x}.
$$
The above decomposition and the two bounds that follow, using \fref{bd:schattenincreasing}, prove \fref{bd:WVY0Zoperator}. We now apply the bounds \fref{bd:ZY0operator}, \fref{bd:WVY0operator} and \fref{bd:WVY0Zoperator} to $\gamma_+^1$, $\gamma_+^2$ and $\gamma_+^3$ respectively, which shows that:
$$
\| \gamma_+\|_{\mathfrak S^{\frac 12, 2q'}}\lesssim \left(\| V\|_{E}+\| Z_+ \|_{L^2_\omega H^{1/2}_x} \right)\left(1+\| V\|_{E}+\| Z_+ \|_{L^2_\omega H^{1/2}_x}\right).
$$
Hence, as $V\in E$ from \fref{bd:VinE} in Step 1, $\gamma_\pm \in \mathfrak S^{1/2,2q'}$ for any $q'>2$ as claimed in Corollary \ref{corollaire}.\\

\noindent \textbf{Step 4} \emph{Proof that $S(-t)\gamma S(t)\rightarrow \gamma_+$ in $\mathfrak S^{\frac 12,4}$}. From \fref{id:mainresult} we write
$$
S(-t)X(t) =Y_0 + W_{V,+}Y_0 + Z_++\tilde W_{V,+}Y_0+R
$$
with $\| R \|_{L^2_\omega H^{\frac 12}_x}\rightarrow 0$ as $t\rightarrow \infty$ and:
$$
\tilde W_{V,+}=-i\int_0^{\infty} S(-\tau)w*\tilde V(\tau) S(\tau)d\tau, \quad \quad \tilde V(\tau)=V(\tau){\bf 1}(\tau \geq t).
$$
Notice $\| \tilde V\|_{E}\rightarrow 0$ from Step 1. We then write
$$
\begin{array}{lllll}
S(-t)|X(t)\rangle \langle X(t)| S(t) -\gamma_f-\gamma_+& \\
 = \mathbb E\left( |Z_{+}\rangle \langle R |+|R\rangle \langle Z_+ |+|R\rangle \langle Y_0|+|Y_0 \rangle \langle R|\right) & \}=:\tilde \gamma_+^1\\
 \quad + \mathbb E\left(|W_{V,+}(Y_0)\rangle \langle \tilde W_{V,+}(Y_0) |+|\tilde W_{V,+}(Y_0)\rangle \langle W_{V,+}(Y_0) | +|\tilde W_{V,+}(Y_0)\rangle \langle Y_0|+|Y_0 \rangle \langle \tilde W_{V,+}(Y_0)|\right)&  \}=:\tilde \gamma_+^2 \\
\quad \quad +\mathbb E\left( |W_{V,+}(Y_0)\rangle \langle R |+ |R \rangle \langle W_{V,+}(Y_0) |+|\tilde W_{V,+}(Y_0)\rangle \langle Z_+ |+ |Z_+ \rangle \langle \tilde W_{V,+}(Y_0) |\right) & \}=:\tilde \gamma_+^3
\ea
$$
We apply the bounds \fref{bd:ZY0operator}, \fref{bd:WVY0operator} and \fref{bd:WVY0Zoperator} to $\tilde \gamma_+^1$, $\tilde \gamma_+^2$ and $\tilde \gamma_+^3$ above respectively, which shows that:
\bee
&& \| S(-t)|X(t)\rangle \langle X(t)| S(t) -\gamma_f-\gamma_+ \|_{\mathfrak S^{\frac 12, 2q'}} \\
& &\quad \quad \lesssim \left(\| \tilde V\|_E+\| R \|_{L^2_\omega H^{\frac 12}_x} \right) \left(1+\| V\|_E+\| Z_+ \|_{L^2_\omega H^{\frac 12}_x}+\| \tilde V\|_E+\| R \|_{L^2_\omega H^{\frac 12}_x}\right)\rightarrow 0
\eee
as $t\rightarrow \infty$, finishing the proof of Corollary \ref{corollaire}.

\end{proof}

\appendix

\section{About dimension 2}\label{sec:dim2}

We obtain here a scattering result near steady states for Equation \fref{hartree} in dimension 2. This is part of an appendix since its proof is simpler than the case of dimension 3, apart from a technical issue that was already tackled in \cite{FLLS,lewsab2}.

\begin{theorem}\label{th:dim2} Assume that $f$ and $w$ satisfy all the assumptions as in Theorem \ref{th:main}, except the bound $\int_0^\infty \left(\frac{|h'|(r)}{r}+|h''|(r)\right)dr<\infty$. Then, there exists $\delta >0$ such that for all $Z_0 \in \Theta_0$ with $\| Z_0 \|_{\Theta_0}\leq \delta$, the following holds true. The Cauchy problem \fref{hartree} with initial datum $Y_0+Z_0$ has a global solution in $Y  + \mathcal{C} (\R, L^2_{\omega,x})$, and what is more, there exist $Z_\pm \in L^2_{\omega,x}$ and $V\in L^{2}_{t,x}$ such that
\be \label{id:scatteringd=2}
\|X(t) - Y(t) - W_V(Y)(t)-S(t) Z_\pm\|_{L^2_{\omega,x}} \rightarrow 0 \quad \quad \mbox{as } t\rightarrow \pm \infty.
\ee
Moreover, there exists $\tilde Z_\pm \in L^4_x L^2_\omega$ such that 
$$
S(-t)W_V(Y) =  \tilde Z_\pm + o_{L^4_x L^2_\omega} (1)  \quad \quad \mbox{as } t\rightarrow \pm \infty .
$$
Defining the correlation operators $\gamma$, $\gamma_f$ and $\gamma_\pm$ by \fref{def:gamma}, \fref{def:gammaf} and \fref{def:gammapm} there holds $\gamma_\pm \in \mathfrak S^{4}$ and:
\be \label{id:scatteringdensityd=2}
\gamma=S(t)(\gamma_f+\gamma_\pm)S(-t)+o_{\mathfrak S^4}(1) \quad \quad \mbox{as }t\rightarrow \pm.
\ee

\end{theorem}

We now give its proof. At several locations, we shall go faster and skip details that are either basic or too similar with the proof in dimension 3. One issue with dimension 2 is that it is $L^2$-critical. Therefore, winning $1/2$ derivatives using homogeneous Sobolev's inequalities does not fit the numerology. In other words, even though in dimension 3, we could prove that $W_V(Y)$ belonged to $L^5_{t,x} L^2_\omega$ which was sufficient to close the argument, in dimension 2, to use the same type of proof, we should be able to prove that $W_V(Y)$ belongs to $L^4_{t,x} L^2_\omega$, but the scheme we use here provides only the proof of $W_V(Y)$ in $L^8_{t,x} L^2_\omega$, which is not enough. However, Lemma 3 in Section 4 in \cite{lewsab2} allows us to prove that $W_V^2(Y)$ belongs to $L^4_{t,x} L^2_\omega$, and even more : that, if $n+m \geq 3$ then $\E(\overline{W_V^n(Y)} W_V^m(Y))$ belongs to $L^2_{t,x}$. This is the main new technical aspect in comparison with the three dimensional case.

We provide here a slightly more general proposition than Lemma 3 in \cite{lewsab2} but where we dropped the dependence of the constant in $n,m$. The argument follows the same line as \cite{lewsab2,FLLS}, we adapted their proofs to our context.

\begin{proposition}\label{prop-plusque3V}

Let $n,m\in \N$ such that $n+m \geq 3$. There exists $C$ such that for all $(A_i)_{1\leq i \leq n}$, $(B_j)_{1\leq j\leq m}$ and $(D_i)_{1\leq i\leq n}$, $(E_j)_{1\leq j\leq m}$ families of measurable maps from $\R$ to $\barR$ and all $(V_i)_{1\leq i\leq n} \in (L^2_{t,x})^n$ with real values and all $(U_j)_{1\leq j\leq m} \in (L^2_{t,x})^m$ with real values, defining:
$$
W_{V,A,D} : 
 u \mapsto \Big( t\mapsto   -i\int_{D(t)}^{A(t)} d\tau S(t-\tau)\Big[V(\tau ) u(\tau)\Big] \Big).
$$
there holds the inequality:
$$
\big\|\E \Big(\overline{ \prod_{j=1}^m W_{U_j,B_j,E_j} (Y)} \prod_{i=1}^n W_{V_i,A_i,D_i}(Y)\Big)\big\|_{L^2_{t,x}} \leq C \prod_{j=1}^m \|U_j\|_{L^2_{t,x}}\prod_{i=1}^n \|V_i\|_{L^2_{t,x}}
$$
where the product is the composition of the linear maps $W_{V_i,A_i,D_i}$ or $W_{U_j,B_j,E_j}$.

\end{proposition}

We postpone the proof of Proposition \ref{prop-plusque3V} to the end of this section and continue with the proof of Theorem \ref{th:dim2}. The fixed point problem at hand is slightly different that the one in dimension 3. This is due to the quadratic term in \fref{fixedpoint} that was called $Q_1$ (defined in \fref{def:Q1}). To estimate it in $L^2_{t,x}$, one should be able to estimate $W_V(Y)$ in $L^4_{t,x}$, which is not possible if we keep a second order expansion as in \fref{fixedpoint}. Therefore, we expand further the nonlinear terms in \fref{fixedpoint}, injecting the $Z$ equation in the one for $V$, and get the third order fixed point problem :
\begin{equation}\label{fixpointdim2}
\left \lbrace{\begin{array}{l l}
Z = S(t)Z_0 + W_V^2(Y) + W_V(Z) \\
V = 2\re \E(\bar Y S(t) Z_0) + 2\re\E (\bar Y W_V(Y)) + \E(|Z|^2) + Q_2(V) \\
\qquad + 2\re \E \left( \overline{W_V(Y)} S(t) Z_0 + \bar Y W_V(S(t) Z_0)\right) + C_1(V) + C_2(Z,V)
\end{array}} \right.
\end{equation}
where $C_1$ and $C_2$ are cubic terms given by
$$
C_1(V) = 2\re \E\left( \overline{W_V(Y)} W_V^2(Y)+ \bar Y W_V^3(Y)\right)
$$
and 
$$
C_2(Z,V)  = 2\re \E\left( \overline{W_V(Y)} W_V(Z) + \bar Y W_V^2(Z) \right).
$$

As in our proof in dimension $3$, we state a general proposition listing sufficient properties in order to obtain Theorem \ref{th:dim2} by solving the above fixed point problem. This is the analogue of Proposition \ref{prop-principle}. We introduce the trilinearisations:
$$
C_1(V,U,W) = 2 \re \E\left( \overline{W_V(Y)} W_U \circ W_W(Y) + \bar Y W_V \circ W_U \circ W_W(Y)\right)
$$
and 
$$
C_2(V,U,Z) = 2\re \E\left( \overline{W_V(Y)} W_U(Z) + \bar Y W_V\circ W_U(Z)\right).
$$
We introduce the following function spaces for the dimension 2:
\be \label{def:spacesd2}
\Theta_0 = L^{4/3}_x L^2_\omega \cap L^2_{\omega,x}\quad , \quad \Theta_Z = \mathcal C(\R, L^2_{\omega,x}) \cap L^4_{t,x} L^2_\omega \quad \textrm{ and }\quad \Theta_V = L^2_{t,x}.
\ee
Note that they are significantly lighter than in dimension 3 and again than in higher dimension.

\begin{proposition}\label{prop:principedim2} Assume that the spaces $(\Theta_0,\Theta_Z,\Theta_V)$ defined by \fref{def:spacesd2} satisfy the list of the following properties:
\begin{description}
\item[Initial datum] $\|S(t)Z_0\|_{\Theta_Z}, \|2\re \E(\bar Y S(t) Z_0)\|_{\Theta_V}\lesssim \|Z_0\|_{\Theta_0}$,
\item[Linear invertibility] $\textrm{Id}_{\Theta_V} -  2\re\E (\bar Y W_V(Y))$ is invertible on $\Theta_V$ with continuous inverse,
\item[Linear continuity of iteration] $\|2\re \E \left( \overline{W_V(Y)} S(t) Z_0 + \bar Y W_V(S(t) Z_0)\right)\|_{\Theta_V} \lesssim \|Z_0\|_{\Theta_0} \|V\|_{\Theta_V}$,
\item[First quadratic term on $Z$] $\|W_V(Z)\|_{\Theta_Z} \lesssim \|V\|_{\Theta_V}\|Z\|_{\Theta_Z}$,
\item[Second quadratic term on $Z$] $\|W_V\circ W_U (Y)\|_{\Theta_Z} \lesssim \|V\|_{\Theta_V}\|U\|_{\Theta_V}$,
\item[Embedding] $\Theta_Z \times \Theta_Z$ is continuous embedded in $\Theta_V$,
\item[Quadratic term on $V$] $\|Q_2(U,V)\|_{\Theta_V} \lesssim \|V\|_{\Theta_V}\|U\|_{\Theta_U}$,
\item[First cubic term on $V$] $\|C_1(V,U,W)\|_{\Theta_V} \lesssim \|V\|_{\Theta_V}\|U\|_{\Theta_V}\|W\|_{\Theta_V}$,
\item[Second cubic term on $V$] $\|C_2(V,U,Z)\|_{\Theta_V} \lesssim \|V\|_{\Theta_V}\|U\|_{\Theta_V}\|Z\|_{\Theta_Z}$,
\item[Scattering space] $\Theta_Z$ is continuously embedded in $L^4_{t,x} L^2_\omega$, for all $A,B \in \barR$, and $V,U \in L^2_{t,x}$ with real values, we have 
$$
\big\| \int_A^B  S(-\tau) \Big[U(\tau) Y(\tau) \Big] d\tau \big\|_{L^4_x L^2_\omega} \lesssim \|U{\bf 1}_{(A,B)}\|_{L^2_{t,x}}
$$
and 
$$
\big\| \int_A^B S(-\tau) \Big[ U(\tau) W_V(Y)(\tau) \Big] d\tau \big\|_{L^2_x L^2_\omega} \lesssim \|U{\bf 1}_{(A,B)}\|_{L^2_{t,x}}\|V\|_{L^2_{t,x}}.
$$
\end{description}
Then the conclusions of Theorem \ref{th:dim2} hold true.
\end{proposition}

\begin{proof}

We follow the proof of Proposition \ref{prop-principle}.

\textbf{Step 1} \emph{Global existence near $Y$}. Using all linear and nonlinear estimates listed in Proposition \ref{prop:principedim2}, one can set up a fixed point argument for \fref{fixpointdim2} the very same way as in Step 1 of the proof of Proposition \ref{prop-principle}. We do not provide the details. Therefore, the problem \fref{fixpointdim2} admits a solution $(Z,V)$ satisfying the global bound:
$$
\| Z \|_{ \mathcal C(\R, L^2_{\omega,x}) \cap L^4_{t,x} L^2_\omega}+ \| V \|_{L^2_{t,x}}<\infty.
$$

\textbf{Step 2} \emph{Scattering for random fields}. We have $X=Y+S(t)Z_0+W_V(Y)+W_V^2(Y)+W_V(Z)$. By arguing as in Step 2 of Proposition \ref{prop-principle}, using the Item Scattering space in Proposition \ref{prop:principedim2}, we have that there exist $Z^1_{\pm}\in L^2_{\omega,x}$ and $\tilde Z_{\pm}\in L^4_xL^2_\omega$ such that:
$$
S(t)Z_0+W_V(Z)=S(t)Z^1_{\pm}+o_{L^2_{\omega,x}}(1) \qquad \mbox{and} \qquad W_V(Y)=S(t)\tilde Z_\pm +o_{L^4_xL^2_\omega}(1)
$$
as $t\rightarrow \pm \infty$. The new term here in comparison with dimension $2$ is $W_V^2(Y)$, and can be written as:
\begin{align*}
W_V^2(Y)& =-\int_0^t S(t-\tau) \left[w*V(\tau)W_V(Y)(\tau) \right]d\tau \\
& =S(t) \underbrace{\left(-\int_0^{\pm\infty} S(-\tau) \left[w*V(\tau)W_V(Y)(\tau) \right]d\tau\right)}_{Z^2_\pm} + \underbrace{\int_{\pm \infty}^t S(t-\tau) \left[w*V(\tau)W_V(Y)(\tau) \right]d\tau}_{=o_{L^2_{\omega,x}}(1)}
\end{align*}
where $Z^2_\pm \in L^2_{\omega,x}$ and the $o_{L^2_{\omega,x}}(1)$ are consequences of the second inequality in Item Scattering space and of the boundedness of $\hat w$. By writing $Z_\pm=Z^1_\pm+Z^2_\pm$ we get \fref{id:scatteringd=2}.

\textbf{Step 3} \emph{Scattering for density matrices}. We now prove \fref{id:scatteringdensityd=2}. We only treat the case $t\rightarrow +\infty$, and write $V$ instead of $w*V$ (as $\hat w\in L^\infty$) without loss of generality. We write from \fref{id:scatteringd=2} and :
$$
\begin{array}{lllll}
\gamma-\gamma_f  &=& \mathbb E\left( |Z\rangle \langle Z |+|Z\rangle \langle Y|+|Y \rangle \langle Z|\right) & \}&=:\gamma^1\\
&& \quad +\mathbb E\left(|W_{V}(Y)\rangle \langle W_{V}(Y) | +|W_{V}(Y)\rangle \langle Y|+|Y \rangle \langle S(-t)W_{V}(Y)|\right)&  \}&=:\gamma^2\\
&&\quad \quad +\mathbb E\left( |W_{V}(Y)\rangle \langle Z |+ |Z \rangle \langle W_{V}(Y) |\right) & \}&=:\gamma^3
\ea
$$
where $Z=S(t)Z_++R$, with $\| R \|_{L^2_\omega L^2_x}\rightarrow 0$ as $t\rightarrow \infty$. From this and \fref{bd:ZY0operator} we get:
$$
S(-t)\gamma^1S(t)=\gamma^1_+ +o_{\mathfrak S^2}(1), \quad \gamma^1_+=\mathbb E\left( |Z_+\rangle \langle Z_+ |+|Z_+\rangle \langle Y_0|+|Y_0 \rangle \langle Z_0|\right)\in \mathfrak S^2.
$$
Next, as $\int_0^t S(t-s)(V(s)Y_s)ds=S(t)\left(\int_0^t S(-s)V(s)S(s)ds\right)Y_0$ and $\mathbb E |Y_0\rangle \langle Y_0|=\gamma_f$, we get:
\bee
S(-t)\mathbb E |W_{V}(Y)\rangle \langle W_{V}(Y) |S(t)&=&\mathbb E | -i \int_0^t S(s)V(s)S(s)ds Y_0\rangle \langle-i \int_0^t S(-s)V(s)S(s)ds Y_0|\\
&&\quad =\int_0^t S(-s)V(s)S(s)ds \circ \gamma_f \circ \int_0^t S(-s)V(s)S(s)ds
\eee
and similarly $S(-t)\mathbb E |W_{V}(Y)\rangle \langle Y|S(t)=-i\int_0^t S(-s)V(s)S(s)ds \circ \gamma_f $. Using \fref{bd:strichartzpotential} with $q'=2$ we get:
$$
\int_0^t S(-s)V(s)S(s)ds=iW_{V,+}-\int_t^\infty S(-s)V(s)S(s)ds=W_{V,+}+o_{\mathfrak S^4}(1), \quad \|W_{V,+}\|_{\mathfrak S^4}\lesssim \| V \|_{L^2_tL^2_x}.
$$
From the two previous identities, the above bound, \fref{bd:holderschatten} and \fref{bd:schattenincreasing} we have:
$$
S(-t)\gamma^2S(t)=\gamma^2_+ +o_{\mathfrak S^4}(1),  \quad \quad \gamma^2_+=\mathbb E\left(|W_{V,\pm}Y_0\rangle \langle W_{V,\pm} Y_0|+|W_{V,\pm}Y_0\rangle \langle Y_0|+|Y_0\rangle \langle W_{V,\pm}Y_0|\right)\in \mathfrak S^4.
$$
Finally, combining the above identity and bound for $\int_0^t S(-s)V(s)S(s)ds$ with \fref{bd:ZY0operator}, using \fref{bd:holderschatten} with $1/4+1/2=3/4$ gives:
$$
S(-t)\gamma^3 S(t)=\gamma^3_++o_{\mathfrak S^{\frac 43}}(1), \quad \gamma^3_+=\mathbb E\left( |W_{V,+}Y_0\rangle \langle Z_+ |+ |Z_+ \rangle \langle W_{V,+}Y_0 |\right) \in \mathfrak S^{\frac 43}.
$$
Collecting the above identities for $\gamma^1$, $\gamma^2$ and $\gamma^3$, noticing that $\gamma_+=\gamma^1_++\gamma^2_++\gamma^3_+$ from \fref{def:gammapm}, and using \fref{bd:schattenincreasing}, proves the last part of the Theorem.

\end{proof}

We can now give the proof of Theorem \ref{th:dim2}.

\begin{proof}[Proof of Theorem \ref{th:dim2}].

Thanks to Proposition \ref{th:dim2}, to end the proof of Theorem \ref{th:dim2} we only need to check that $(\Theta_0,\Theta_Z,\Theta_V)$ defined by \fref{def:spacesd2} satisfy the list of properties listed in this Proposition.

Item Initial datum comes from Strichartz estimates, that ensure that
$$
\|S(t)Z_0\|_{\Theta_Z} \lesssim \|Z_0\|_{L^2_{\omega,x}}
$$
and duality. Indeed, let $U \in L^2_{t,x}$ real valued, we have 
$$
\an{ U, 2\re \E (\bar Y S(t) Z_0)}_{t,x} = 2\re \an{\int_{\R} S(-t) \Big[ U(t)Y(t)\Big] dt, Z_0}_{x,\omega}.
$$
By H\"older's inequality, we have 
$$
|\an{ U, 2\re \E (\bar Y S(t) Z_0)}_{t,x}| \leq  2\big\|\int_{\R} S(-t) \Big[ U(t)Y(t)\Big] dt\big\|_{L^4_x L^2_\omega} \|Z_0\|_{L^{4/3}_x L^2_\omega}.
$$
The inequality
$$
\big\|\int_{\R} S(-t) \Big[ U(t)Y(t)\Big] dt\big\|_{L^4_x L^2_\omega} \lesssim \|U\|_{L^2_{t,x}}
$$
follows from Item Scattering spaces that we prove later on.

Item Linear invertibility corresponds to Proposition 1 and Corollary 1 in \cite{lewsab2}.

We can deal with Item Linear continuity for iterate by duality. Indeed, We have for all $\| U\|_{L^2_{t,x}}=1$,
\begin{eqnarray*}
\an{U,\E ( \overline{W_V(Y)} S(t) Z_0)}_{t,x} &=& \an{\int_{0}^\infty S(-t) [U(t)W_V(Y) ],Z_0}_{x,\omega} \\
&\leq &\big\| \int_{0}^\infty S(-t) [U(t)W_V(Y) ]\big\|_{L^2_{x,\omega}} \|Z_0\|_{L^2_{x,\omega}}.
\end{eqnarray*}
We compute that for the first term above:
\begin{align*}
&\big\| \int_{0}^\infty S(-t) [U(t)W_V(Y) ]\big\|_{L^2_{x,\omega}}^2 = \an{  \int_{0}^\infty  S(-t) [U(t)W_V(Y)(t) ], \int_{0}^\infty S(-s)[U(s)W_V(Y)(s)]}_{x,\omega} \\
& \qquad  \qquad  \qquad = \int_{0}^\infty dt \an{   [U(t)W_V(Y)(t) ], \int_{0}^\infty S(t-s)[U(s)W_V(Y)(s)]}_{x,\omega} .
\end{align*}
Applying H\"older inequality, one gets within the range of application of Proposition \ref{prop-plusque3V}
\bee
... &\leq & \big\| \E\left(\overline{W_V(Y)} \int_{0}^\infty S(t - s) [U(s)W_V(Y)(s)]\right) \big\|_{L^2_{t,x}}=\big\| \E(\overline{W_V(Y)} W_{0,\infty,U}\circ W_V(Y) )\big\|_{L^2_{t,x}}\\
&\lesssim & \| U \|_{L^2_{t,x}}  \| V \|_{L^2_{t,x}}^2 = \| V \|_{L^2_{t,x}}^2 .
\eee

Item First quadratic term on $Z$ comes from Strichartz estimates and Christ-Kiselev lemma.

For Item Second quadratic term on $Z$, the inequality
$$
\|W_V\circ W_U (Y) \|_{L^4_{t,x} L^2_\omega} \lesssim \|V\|_{L^2_{t,x}} \|U\|_{L^2_{t,x}}
$$
follows from Proposition \ref{prop-plusque3V}. The inequality
$$
\|W_V\circ W_U (Y) \|_{\mathcal C(\R,L^2_{\omega,x})} \lesssim \|V\|_{L^2_{t,x}} \|U\|_{L^2_{t,x}}
$$
follows from Item Scattering spaces that we prove later on.

Item Embedding follows from the definition of the spaces and H\"older's inequality.

Item Quadratic term on $V$ can be dealt with in the same way as in dimension 3 (see \cite{lewsab2}).

Item First cubic term on $V$ follows from Proposition \ref{prop-plusque3V}.

Item Second cubic term on $V$ follows from Proposition \ref{prop-plusque3V} and by duality. Indeed, let a test function $\phi \in L^2_{t,x}$. We have 
$$
\an{\phi, C_2(V,U,Z)}_{t,x} = I + II
$$
with
$$
I = \an{\phi, 2\re \E( \overline{W_V(Y)} W_U(Z))}_{t,x} \quad \textrm{ and }\quad II = \an{\phi,2\re \E ( \bar Y W_V \circ W_U (Z))}_{t,x}.
$$
We can rewrite $I$ as (where $\text{Id}$ is the mapping $t\mapsto t$)
$$
I = 2 \re \an{ W_{\phi,\text{Id},\infty } \circ W_V (Y), w* U Z}_{t,x,\omega}
$$
and thus, by H\"older's inequality,
$$
|I| \leq 2 \| W_{\phi,0,\infty } \circ W_V (Y)\|_{L^4_{t,x},L^2_\omega} \|w*U\|_{L^2_{t,x}} \|Z\|_{L^4_{t,x} L^2_\omega}.
$$
Above, one has by Proposition \ref{prop-plusque3V} and since $\hat w\in L^{\infty}$:
$$
\| W_{\phi,0,\infty }  \circ W_V (Y)\|_{L^4_{t,x} L^2_\omega} \lesssim \|W\|_{L^2_{t,x}} \|V\|_{L^2_{t,x}}, \qquad \qquad \|w*U\|_{L^2_{t,x}} \lesssim \|U\|_{L^2_{t,x}},
$$
which concludes the estimate on $I$. We estimate $II$ similarly. We have the identity
$$
II = 2 \re \an{w*V,W_{\text{Id},\infty} \circ W_{\phi,\text{Id},\infty }(Y), w*U Z}_{t,x,\omega}.
$$
By H\"older's inequality, we get
$$
|II|\leq 2 \|W_{w*V,\text{Id},\infty} \circ W_{\phi,\text{Id},\infty }(Y)\|_{L^4_{t,x} L^2_\omega}\|w*U\|_{L^2_{t,x}} \|Z\|_{L^4_{t,x} L^2_\omega}.
$$
We use Proposition \ref{prop-plusque3V} to get
$$
\|W_{w*V,\text{Id},\infty} \circ W_{\phi,\text{Id},\infty }(Y)\|_{L^4_{t,x} L^2_\omega} \lesssim \|V\|_{L^2_{t,x}} \|W\|_{L^2_{t,x}}
$$
which concludes the estimate on $II$.

We finally prove Item Scattering spaces. Let $V \in L^2_{t,x}$ and $V_{A,B} = {\bf 1}_{t\in (A,B)} V$. We have 
$$
\int_{A}^B S(-\tau) \Big[ V(\tau) Y(\tau)\Big] d\tau = \int_{\R} S(-\tau) \Big[ V_{A,B}(\tau) Y(\tau)\Big] d\tau.
$$
By Lemma \ref{lem:expressionWVY}, we have 
$$
\E \left(\Big|\int_{A}^B S(-\tau) \Big[ V(\tau) Y(\tau)\Big] d\tau\Big|^2\right) =  \int_{\mathbb R^3} d\xi |f(\xi)|^2 \Big|\int_{\R} S_\xi(-\tau)  V_{A,B}(\tau)d\tau \Big|^2.
$$
Therefore,
$$
\big\| \int_{A}^B S(-\tau) \Big[ V(\tau) Y(\tau)\Big] d\tau\big\|_{L^4_x L^2_\omega}^2 \leq \int_{\mathbb R^3} d\xi |f(\xi)|^2 \big\| \int_{\R} S_\xi(-\tau)  V_{A,B}(\tau)d\tau\big\|_{L^4_x}^2.
$$
By Sobolev's inequality, we have 
$$
\big\| \int_{A}^B S(-\tau) \Big[ V(\tau) Y(\tau)\Big] d\tau\big\|_{L^4_x L^2_\omega}^2 \lesssim \int_{\mathbb R^3} d\xi |f(\xi)|^2 \big\| \int_{\R} S_\xi(-\tau)  V_{A,B}(\tau)d\tau\big\|_{\dot H^{1/2}_x}^2.
$$
By repeating the proof of the $1/2$ regularity gain in \fref{bd:potentialtorandom}, we get the first inequality of Item Scattering space
$$
\big\| \int_{A}^B S(-\tau) \Big[ V(\tau) Y(\tau)\Big] d\tau\big\|_{L^4_x L^2_\omega}^2 \leq C_h \|V_{A,B}\|_{L^2_{t,x}}.
$$
We now prove the second inequality of Item Scattering space, namely that
$$
I := \big\|\int_{A}^B S(-\tau) \Big[V W_U(Y)\Big] d\tau \big\|_{L^2_{x,\omega}} \lesssim \|V_{A,B}\|_{L^2_{t,x}} \|U\|_{L^2_{t,x}}.
$$
We have the identity
\bee
I^2  &=&\an{\int_{\mathbb R} d\tau S(-\tau) V_{A,B} (\tau) W_U(Y)(\tau) , \int_{\mathbb R} ds S(-s)V_{A,B}(s) W_U(Y)(s)}_{x,\omega}\\
& =& \int_{\mathbb R} d\tau \int_{\mathbb R} ds \an{S(s-\tau) V_{A,B}(\tau)W_U(Y) (\tau), V_{A,B}(s) W_U(Y)(s)}_{x,\omega}.
\eee
We recognise above
$$
I^2 = \int_{\mathbb R} ds \an{W_{V_{A,B},-\infty,\infty} \circ W_U(Y) , w*V_{A,B}(s) W_U(Y)(s)}_{x,\omega}
$$
and thus
$$
I^2 = \an{\E(\overline{W_U(Y)}W_{-\infty,\infty,w*V_{A,B}} \circ W_U(Y), w*V_{A,B}}_{s,x}.
$$
By H\"older's inequality and Proposition \ref{prop-plusque3V}, we get the result.

\end{proof}

To end the proof of Theorem \ref{th:dim2}, there remains to prove Proposition \ref{prop-plusque3V}. We follow the lines of \cite{FLLS}. We first need a technical Lemma. We introduce a new set of notations. Let $\underline t=(t_1,...,t_n)$ and $ \underline s=(s_1,...,s_m)$ and set
$$
T_1(t_0,\underline t) = \prod_{i=1}^n {\bf 1}_{t_i \in [D_i(t_{i-1}),A_i(t_{i-1})]}, \quad T_2(s_0,\underline s) = \prod_{j=1}^m {\bf 1}_{s_j \in [E_j(s_{j-1}),B_j(s_{j-1})]} \quad T(t,\underline s,\underline t)=T_1(t,\underline t)T_2(s,\underline s) .
$$
For all $V \in L^2_{t,x}$, we set
$$
\tilde V (t)= S^{-1}(t) V(t) S(t).
$$
and define
$$
\mathbb W_1(\underline t) =\prod_{j=1}^n \tilde V_j(t_{j}) \quad ,\quad \mathbb W_2(\underline s) =\prod_{j=1}^m \tilde U_{m+1-j}(s_{m+1-j}).
$$

\begin{lemma}\label{lem:comput}
We have, for $\rho \left[ S \mathbb W_1  \gamma_{|f|^2} \mathbb W_2S^* \right]$ the diagonal of the kernel of $S \mathbb W_1  \gamma_{|f|^2} \mathbb W_2S^*$:
$$
\E \Big(\overline{ \prod_{j=1}^m W_{U_j,B_j,E_j} (Y)} \prod_{i=1}^n W_{V_i,A_i,D_i}(Y)\Big) = i^m(-i)^n \int_{\mathbb R^{m+n}} d\underline s d\underline t  T(t,\underline s,\underline t)\rho \left[ S(t) \mathbb W_1(\underline t)  \gamma_{|f|^2} \mathbb W_2(\underline s)S(-t) \right].
$$
\end{lemma}

\begin{proof} By definitions of $W_{V_j,A_j,D_j} $, $T_1$ and $\mathbb W_1$, and since $Y(t)=S(t)Y(t=0)$, we have that:
\begin{eqnarray*}
\prod_{j=1}^n W_{V_j,A_j,D_j} (Y)(t)&= (-i)^n \int_{\mathbb R^n} d\underline t T_1(t,\underline t) S(t)\left(\prod_{j=1}^n \tilde V_j(t_{j})\right) Y(t=0)\\
&\quad \quad= (-i)^n \int_{\mathbb R^n} d\underline t T_1(t,\underline t) S(t)\mathbb W_1(\underline t) Y(t=0).
\end{eqnarray*}
Note that for all $1\leq j \leq m$, $s_j\in \mathbb R$, the operator $\tilde U_j(s_j)$ is self-adjoint in $L^2_x$. Therefore, we have 
$$
\left(\prod_{j=1}^m \tilde U_j(s_{j})\right)^*= \prod_{j=1}^m (\tilde U_{m+1-j}(s_{m+1-j}))^*=\mathbb W_2(\underline s).
$$
Consequently, we can write:
\begin{eqnarray*}
&&\overline{ \prod_{j=1}^m W_{U_j,B_j,E_j} (Y)}(t) \prod_{i=1}^n W_{V_i,A_i,D_i}(Y)(t)\\
&&  \quad \quad \quad \quad=i^m(-i)^n \int_{\mathbb R^{m+n}} T_1(t,\underline t)T_2(t,\underline s)\overline{S(t)\mathbb W_2^*(\underline t) Y(t=0)} S(t)\mathbb W_1(\underline t) Y(t=0).
\end{eqnarray*}
Note that the right hand side is the diagonal of the kernel of the operator:
\begin{eqnarray*}
&\overline{S\mathbb W_2^* Y(t=0)}S\mathbb W_1Y(t=0)=\rho\left[ | S \mathbb W_1 (Y(t=0)) \rangle \langle S \mathbb W_2^* (Y(t=0))|\right],\\
&\quad \quad \quad \quad  \quad \quad \quad \quad \quad \quad \quad \quad \quad \quad \quad =\rho \left[ S \mathbb W_1  | (Y(t=0)) \rangle \langle  (Y(t=0))| \mathbb W_2S^*\right]
\end{eqnarray*}
where the ket-bra notations are taken in $L^2_x$. Recalling $\E(| Y(t=0) \rangle \langle Y(t=0) |) = \gamma_{|f|^2}$, that $\mathbb W_1$ and $\mathbb W_2$ do not depend on the random variable, and that $T_1T_2=T$ we get the desired result:
\begin{eqnarray*}
&&\E \Big(\overline{ \prod_{j=1}^m W_{U_j,B_j,E_j} (Y)} \prod_{i=1}^n W_{V_i,A_i,D_i}(Y)\Big) \\
&&\quad  =i^m(-i)^n \int_{\mathbb R^{m+n}} d\underline s d\underline t  T_1(t,\underline t)T_2(t,\underline s)\mathbb E \left(\overline{S(t)\mathbb W_2^*(\underline t) Y(t=0)} S(t)\mathbb W_1(\underline t) Y(t=0)\right) \\
&&\quad \quad  =i^m(-i)^n \int_{\mathbb R^{m+n}} d\underline s d\underline t  T(t,\underline s,\underline t)\rho \left[ \mathbb E \left(S(t) \mathbb W_1(\bar t)  | (Y(t=0)) \rangle \langle  (Y(t=0))| \mathbb W_2(\underline s)S(-t)\right)\right] \\
&&\quad \quad \quad =i^m(-i)^n \int_{\mathbb R^{m+n}} d\underline s d\underline t  T(t,\underline s,\underline t)\rho \left[ S(t) \mathbb W_1(\underline t)  \gamma_{|f|^2} \mathbb W_2(\underline s)S(-t) \right].
\end{eqnarray*}

\end{proof}

\begin{proof}[Proof of Proposition \ref{prop-plusque3V}] We proceed by duality. The core of the argument is the use of inequalities in Schatten spaces, and convolution type inequalities on the real line. Let a test function $V \in L^2_{t,x}$. By separating between non-negative and non-positive parts for $V,V_i, U_j$, $1\leq i\leq n$, $1\leq j\leq m$, we can assume $V,V_i,U_j$ all have constant sign, for example all non-negative. We bound
$$
|I| := \left|\left\langle V,\E \Big(\overline{ \prod_{j=m}^1 W_{U_j,B_j,E_j} (Y)} \prod_{i=1}^n W_{V_i,A_i,D_i}(Y)\Big) \right\rangle_{t,x}\right|.
$$
First note that by cyclicity of the trace:
\begin{eqnarray*}
\left\langle V(t),\rho \left[ S(t) \mathbb W_1(\underline t)  \gamma_{|f|^2} \mathbb W_2(\underline s)S(-t) \right]\right\rangle_x&=\textrm{Tr}\left[ V(t)S(t) \mathbb W_1(\underline t)  \gamma_{|f|^2} \mathbb W_2(\underline s)S(-t) \right] \\
&\quad \quad \quad  =\textrm{Tr} \left[ \tilde V(t) \mathbb W_1(\underline t)  \gamma_{|f|^2} \mathbb W_2(\underline s) \right].
\end{eqnarray*}
Therefore, by the above identity and Lemma \ref{lem:comput} we have the following expression:
\begin{eqnarray*}
|I| &= \left|\int_{\R^{m+n+1}} dt d\underline t d\underline s  T(t,\underline s, \underline t) \textrm{Tr}\left( \tilde V(t) \prod_{i=1}^n \tilde V_i(t_i) \gamma_{|f|^2} \prod_{j=1}^m \tilde U_{m+j-1}(s_{m+j-s})\right)\right|\\
&  \leq  \int_{\R^{m+n+1}} dtd\underline t d\underline s \Big| \textrm{Tr}\Big( \widetilde V(t) \prod_{i=1}^n \widetilde V_i(t_i) \gamma_{|f|^2} \prod_{j=1}^m \widetilde U_j(s_j)\Big)\Big|
\end{eqnarray*}
where we made the abuse of notation of replacing $\tilde U_{m+j-1}(s_{m+j-s})$ by $\tilde U_{j}(s_{j})$ in the last line since the order does not matter. Because $V,V_i,U_j$ are non-negative, we have that $\widetilde V (t)= \widetilde{V^{1/2}(t)}^2, \widetilde V_i (t_i) = \widetilde{V^{1/2}_i(t_i)}^2$ and $\widetilde U_j(s_j) = \widetilde{U_j^{1/2}(s_j)}^2$ for all $1\leq i\leq n$ and $1\leq j\leq m$ and all $t_i,s_j$. Indeed,
$$
\widetilde{V^{1/2}(t)}^2 = S(-t) V^{1/2}(t) S(t) S(-t) V^{1/2} (t)S(t) = \widetilde V(t).
$$

By clyclicity of the trace, we have 
\begin{multline*}
I \leq \int_{\R^{m+n+1}} dtd\underline t d\underline s  \textrm{Tr}\Big( \widetilde {V^{1/2}(t)} \widetilde{V_1^{1/2}(t_1)} \prod_{i=1}^{n-1} \Big( \widetilde{ V_i^{1/2}(t_i)}\widetilde{ V_{i+1}^{1/2}(t_{i+1})} \Big)\\
\widetilde{ V_n^{1/2}(t_n)} \gamma_{|f|^2} \widetilde{U_1^{1/2}(s_1)} \prod_{j=1}^{m-1}\Big( \widetilde{U_j^{1/2}(s_j)} \widetilde{U_{j+1}^{1/2}(s_{j+1})}\Big) \widetilde{U_m^{1/2}(s_m)} \widetilde {V^{1/2}(t)}\Big).
\end{multline*}
Writing $t_0 = t$, $t_{n+1} = s_1, \hdots, t_{n+m} = s_m$ and $V_0 = V$, $V_{n+1} = U_1, \hdots, V_{n+m} = U_m$, we get
\begin{multline*}
I \leq \int_{\R^{n+m+1}} dt_0\hdots dt_{n+m}   \textrm{Tr}\Big( \prod_{i=0}^{n-1} \Big( \widetilde{ V_i^{1/2}(t_i)}\widetilde{ V_{i+1}^{1/2}(t_{i+1})} \Big)\\
\widetilde{ V_n^{1/2}(t_n)} \gamma_{|f|^2} \widetilde{V_{n+1}^{1/2}(t_{n+1})} \prod_{i=n+1}^{n+m-1}\Big( \widetilde{V_i^{1/2}(t_i)} \widetilde{V_{i+1}^{1/2}(t_{i+1})}\Big) \widetilde{V_{n+m}^{1/2}(t_{n+m})} \widetilde {V_0^{1/2}(t_0)}\Big).
\end{multline*}
By H\"older's inequality applied to Schatten spaces, we get
\begin{multline*}
I \leq \int_{\R^{n+m+1}} dt_0\hdots dt_{n+m}  \prod_{i=0}^{n-1} \big\|\widetilde{ V_i^{1/2}(t_i)}\widetilde{ V_{i+1}^{1/2}(t_{i+1})} \big\|_{\mathfrak{S}^{n+m+1}}\\
\big\|\widetilde{ V_n^{1/2}(t_n)} \gamma_{|f|^2} \widetilde{V_{n+1}^{1/2}(t_{n+1})}  \big\|_{\mathfrak{S}^{n+m+1}} \prod_{i=n+1}^{n+m-1}\big\| \widetilde{V_i^{1/2}(t_i)} \widetilde{V_{i+1}^{1/2}(t_{i+1})} \big\|_{\mathfrak{S}^{n+m+1}}\big\| \widetilde{V_{n+m}^{1/2}(t_{n+m})} \widetilde {V_0^{1/2}(t_0)} \big\|_{\mathfrak{S}^{n+m+1}}.
\end{multline*}

Since $n+m+1 \geq 4$, we have 
\begin{multline*}
I \leq \int_{\R^{n+m+1}} dt_0\hdots dt_{n+m}  \prod_{i=0}^{n-1} \big\|\widetilde{ V_i^{1/2}(t_i)}\widetilde{ V_{i+1}^{1/2}(t_{i+1})} \big\|_{\mathfrak{S}^{4}}\\
\big\|\widetilde{ V_n^{1/2}(t_n)} \gamma_{|f|^2} \widetilde{V_{n+1}^{1/2}(t_{n+1})}  \big\|_{\mathfrak{S}^{4}} \prod_{i=n+1}^{n+m-1}\big\| \widetilde{V_i^{1/2}(t_i)} \widetilde{V_{i+1}^{1/2}(t_{i+1})} \big\|_{\mathfrak{S}^{4}}\big\| \widetilde{V_{n+m}^{1/2}(t_{n+m})} \widetilde {V_0^{1/2}(t_0)} \big\|_{\mathfrak{S}^{4}}.
\end{multline*}
Using Lemma 1 p 10 in \cite{FLLS}, we get
$$
 \big\|\widetilde{ V_i^{1/2}(t_i)}\widetilde{ V_{i+1}^{1/2}(t_{i+1})} \big\|_{\mathfrak{S}^{4}} \lesssim \frac{\|V_i^{1/2}(t_i)\|_{L^4} \|V_{i+1}^{1/2}(t_{i+1})\|_{L^4}}{|t_i-t_{i+1}|^{1/2}}
 $$
that is
$$
 \big\|\widetilde{ V_i^{1/2}(t_i)}\widetilde{ V_{i+1}^{1/2}(t_{i+1})} \big\|_{\mathfrak{S}^{4}} \lesssim \frac{\|V_i(t_i)\|_{L^2}^{1/2} \|V_{i+1}(t_{i+1})\|_{L^2}^{1/2}}{|t_i-t_{i+1}|^{1/2}}.
 $$
Using Lemma 4 p18 in \cite{lewsab2}, we have 
$$
\big\|\widetilde{ V_n^{1/2}(t_n)} \gamma_{|f|^2} \widetilde{V_{n+1}^{1/2}(t_{n+1})}  \big\|_{\mathfrak{S}^{4}} \lesssim \|h\|_{L^1}\frac{\|V_n(t_n)\|_{L^2}^{1/2}\|V_{n+1}(t_{n+1})\|_{L^2}^{1/2}}{|t_n - t_{n+1}|^{1/2}}.
$$

Summing up, we get
$$
I \leq C_h \int_{\R^{n+m+1}} dt_0\hdots dt_{n+m}  \prod_{i=0}^{n+m} \|V_i(t_i)\|_{L^2} \prod_{i=0}^{n+m-1} |t_i - t_{i+1}|^{-1/2} |t_{n+m} - t_0|^{-1/2}.
$$
Write $v_i(t_i) =  \|V_i(t_i)\|_{L^2}$ such that
$$
I \leq C_h \int_{\R^{n+m+1}} dt_0\hdots dt_{n+m}  \prod_{i=0}^{n+m} v_i(t_i) \prod_{i=0}^{n+m-1} |t_i - t_{i+1}|^{-1/2} |t_{n+m} - t_0|^{-1/2}.
$$
Let
$$
J = \int_{\R^{n+m+1}} dt_0\hdots dt_{n+m}  \prod_{i=0}^{n+m} v_i(t_i) \prod_{i=0}^{n+m-1} |t_i - t_{i+1}|^{-1/2} |t_{n+m} - t_0|^{-1/2}.
$$
We have that
$$
J = \int_{\R^{n+m+2}} dt_0 \hdots dt_{n+m+1} \delta(t_{n+m+1} - t_0)\prod_{i=0}^{n+m} v_i(t_i) \prod_{i=0}^{n+m} |t_i - t_{i+1}|^{-1/2} .
$$
In other words, $J$ is the integral of the diagonal of
$$
g(t_0,t_{n+m+1}) = \int_{\R^{n+m}} dt_1 \hdots dt_{n+m} \prod_{i=0}^{n+m} v_i(t_i) \prod_{i=0}^{n+m} |t_i - t_{i+1}|^{-1/2} .
$$
Therefore $J$ is the trace of the operator $\Gamma$ with integral kernel
$$
g(t_0,t_{n+m+1}).
$$
Writing $\gamma$ the convolution with $|t|^{-1/2}$ and $\gamma_i = v_i \gamma$, we get that
$$
\Gamma = \prod_{i=0}^{n+m} \gamma_i.
$$
By cyclicity of the trace, we have 
$$
\textrm{Tr}\Gamma = \textrm{Tr} \Big( \prod_{i=0}^{n+m} \gamma^{1/2} v_i \gamma^{1/2} \Big).
$$
By H\"older's inequality applied to Schatten spaces, we get
$$
\textrm{Tr}\Gamma \leq \prod_{i=0}^{n+m} \|\gamma^{1/2} v_i \gamma^{1/2}\|_{\mathfrak S^{n+m+1}}.
$$
We have
$$
\|\gamma^{1/2} v_i \gamma^{1/2}\|_{\mathfrak S^{n+m+1}}= \|\gamma^{1/2} v_i^{1/2} \|_{\mathfrak S^{2(n+m+1)}}^2.
$$
Because $n+m+1 \geq 3> 2$, we have $2(n+m+1) > 4$, and therefore,
$$
\|\gamma^{1/2} v_i \gamma^{1/2}\|_{\mathfrak S^{n+m+1}}\leq \|\gamma^{1/2} v_i^{1/2} \|_{\mathfrak S_w^{4}}.
$$
Using Cwikel's inequality as in p 17 of \cite{FLLS}, we get
$$
\|\gamma^{1/2} v_i^{1/2} \|_{\mathfrak S_w^{4}}\lesssim \|v_i^{1/2}\|_{L^4} = \|V_i\|_{L^2_{t,x}}^{1/2}.
$$
Finally
$$
I \leq C_h \prod_{i=0}^{n+m} \|V_i\|_{L^2_{t,x}}
$$
which concludes the proof.

\end{proof}

\section{Technical results}\label{sec:tech}

We provide here basic results on the Wiener integral for unfamiliar readers. The Wiener integral associates to each $f\in L^2(\mathbb R^3)$ a complex centred Gaussian variable $X_f$ denoted by $\int_{\mathbb R^3} f(\xi) dW(\xi) $ with variance $\int_{\mathbb R^3} |f|^2d\xi$. It is an isometry onto its image in $L^2(\Omega)$ since:
\be \label{id:correlationgaussian}
\int_\Omega \overline{X_f}X_g d\omega=\mathbb E \left(\overline{\int_{\mathbb R^3} f(\xi)dW(\xi) } \int_{\mathbb R^3} g(\xi)dW(\xi) \right)=\int_{\mathbb R^3} \bar f(\xi)g(\xi)d\xi .
\ee
The existence of a probability space with random variables $X_f(\omega)$ for each $f\in L^2(\mathbb R^3)$ satisfying the above correlation relations requires no additional information and follows from the application of Kolmogorov's extension theorem. A Gaussian random field on $\mathbb R^3$ requires however to make sense of an infinite number of Gaussian variables simultaneously, at each points of space. Measurability issues are then coped with by appealing to separability, and it is relevant to have an explicit construction instead of this abstract theorem, to cope with measurability issues. Here, we take $\Omega$ a space with a sequence $g_i:\Omega \rightarrow \mathbb C$ for $i\in \mathbb N$ of centred normalised independent Gaussians. We take $(e_i)_{i\in \mathbb N}$ a basis of $L^2(\mathbb R^3)$ and define for $f\in L^2(\mathbb R^3)$:
\be \label{id:explicitwhitenoise}
X_f(\omega)=\sum_{i\in \mathbb N} \left(\int_{\mathbb R^3} f (\xi)\bar e_i(\xi)d\xi\right)g_i(\omega).
\ee
The above construction is well defined and satisfies \fref{id:correlationgaussian} for finite sums $f(\xi)=\sum_{i=1}^n a_ie_i(\xi)$, and so does its extension to $L^2(\mathbb R^3)$ by isometry. It allows us to prove standard Fubini-type results of commutation between a Lebesgue and a Wiener integral used in this paper.

The reason we choose not to use this notation throughout the paper is that it requires to fix a basis of $L^2$, which makes the notations heavier. Nevertheless, by fixing 
$$
Y_f = \sum_i \an{fe^{-it(m+|\xi|^2) + ix\cdot \xi}, e_i}_{L^2} g_i 
$$\footnote{and indeed, one may check that this formula is satisfied for $g_i = \int e_i(\xi)dW(\xi)$}
we get the same result and indeed the same results for explicit computations. 

\begin{lemma}

For any $f\in L^2(\mathbb R^3\times \mathbb R^3)$ there holds:
\be \label{id:fubinischrodingerwiener}
S(t) \left(\int_{\mathbb R^3} f(\xi,\cdot )dW(\xi) \right)(x)= \int_{\mathbb R^3} (S(t) f)(\xi,x)dW(\xi).
\ee
For any $f\in L^1(\mathbb R^n,L^2 (\mathbb R^3))$ there holds:
\be \label{id:fubinilebesguewiener}
\int_{\mathbb R^n} \left(\int_{\mathbb R^3} f(y,\xi)dW(\xi)\right)dy=\int_{\mathbb R^3} \left(\int_{\mathbb R^n}  f(y,\xi)dy \right)dW(\xi).
\ee
\end{lemma}

\begin{proof}

Let $f\in L^2(\mathbb R^3\times \mathbb R^3)$. Thanks to the explicit construction \fref{id:correlationgaussian}, \fref{id:explicitwhitenoise} and a density argument, $\int_{\mathbb R^3} f(\xi,x)dW(\xi)$ is well defined, measurable, with $\|\int fdW\|_{L^2(\mathbb R^3\times \Omega)}=\| f \|_{L^2(\mathbb R^3\times \mathbb R^3)}$. Hence almost surely $ (\int  f(\xi,\cdot)dW(\xi))(\omega)$ belongs to $L^2(\mathbb R^3)$ so the left hand side of \fref{id:fubinischrodingerwiener} is well defined as the free evolution of an $L^2$ function. Since $\|S(t)f\|_{L^2(\mathbb R^3\times \mathbb R^3)}=\|f\|_{L^2(\mathbb R^3\times \mathbb R^3)}$ by Parseval, we get as previously that the right hand side of \fref{id:fubinischrodingerwiener} is well defined as an $L^2(\mathbb R^3\times \Omega)$ function. To show the equality \fref{id:fubinischrodingerwiener} of the two constructions, we therefore only need to show it for a dense subset of $L^2(\mathbb R^3\times \mathbb R^3)$ and the conclusion follows from density and isometry. The proof is then ended by considering the subset of finite sums $f(\xi,x)=\sum_{i=1}^n f_i(x)e_i(\xi)$ with $(f_i)_{1\leq i\leq n} \in L^2(\mathbb R^3)$ since:
\bee
S(t) \left(\int_{\mathbb R^3} f(\xi,\cdot )dW(\xi) \right)(\omega,x)&=&S(t) \left(\sum_{i=1}^n f_i(\cdot) g_i(\omega) \right)(x)\\
&&\quad \quad  =\sum_{i=1}^n S(t)(f_i)(x) g_i(\omega)  = \int_{\mathbb R^3} (S(t) f)(\xi,x)dW(\xi).
\eee
The second equality can be proved very similarly, and is left to the reader.

\end{proof}

We finally recall standard Strichartz estimates in dimension three, and their extension via Christ-Kiselev's Lemma. We refer to the textbook \cite{tao} for additional information.

\begin{lemma} \label{lem:strichartz}
Assume $0\leq s<3/2$ and $2\leq p,q \leq \infty$, satisfy
\begin{equation} \label{id:strichartzexponent}
\frac2{p}+\frac3{q} = \frac 32-s.
\end{equation}
Then the following holds true for a constant $C=C(s,p,q)$, for any $u_0\in \dot H^s_x$:
\be \label{bd:homostrichartz}
\| S(t)u_0\|_{L^p_tL^q_x\cap C_t\dot H^s_x}\leq C \| u_0\|_{\dot H^s_x}.
\ee
Moreover, for any $p,q$ and $\tilde p,\tilde q$ satisfying \fref{id:strichartzexponent} with $s=0$, for any $f\in L^{\tilde p'}_tL^{\tilde q'}_x$ where $\tilde p',\tilde q'$ are the H\"older conjugate exponents of $\tilde p,\tilde q$:
\be \label{bd:dualstrichartz}
\| \int_{\mathbb R} S(t-s)f(s)ds \|_{L^{p}_tL^q_x \cap C_t L^2_x}\leq C\| f\|_{L^{\tilde p'}_tL^{\tilde q'}_x}.
\ee

\end{lemma}

\bibliographystyle{amsplain}
\bibliography{bibeqonrv} 

\end{document}